\documentclass[11pt, a4paper, oneside, reqno]{article}

\usepackage[numbers]{natbib}
\usepackage[colorlinks=true]{hyperref}
\usepackage{amssymb, amsthm, amsmath, amsfonts}
\usepackage{authblk, dsfont, mathrsfs, color, bm, enumitem, url}
\usepackage{subfigure, graphicx, epstopdf, transparent} 
\usepackage{algorithm, algorithmicx, multirow, titlecaps}

\newcommand{\Min}{\min\limits}
\newcommand{\Sup}{\sup\limits}
\newcommand{\Inf}{\inf\limits}
\newcommand{\Sum}{\sum\limits}
\newcommand{\SumN}{\sum\limits_{i=1}^N}

\newcommand{\PP}{\mathds{P}}
\newcommand{\QQ}{\mathds{Q}}
\newcommand{\EE}{\mathds{E}}
\newcommand{\RR}{\mathbb{R}}
\newcommand{\WW}{\mathbb{W}}
\newcommand{\YY}{\mathbb{Y}}
\newcommand{\HH}{\mathbb{H}}
\newcommand{\XX}{\mathbb{X}}

\newcommand{\KK}{\bm{\mathcal{K}}}

\newcommand{\CC}{\mathcal{C}}

\newcommand{\ZZ}{\mathcal{Z}}

\newcommand{\x}{\bm x}
\newcommand{\w}{\bm w}

\newcommand{\q}{\bm q}
\newcommand{\p}{\bm p}

\newcommand{\z}{\bm z}
\newcommand{\xxi}{\bm \xi}
\newcommand{\Beta}{\bm \beta}
\newcommand{\BB}{\mathds{B}}

\newcommand{\Prob}{\mathbb{P}}
\newcommand{\Obj}{\lambda \eps + \frac{1}{N} \Sum_{i=1}^N s_i}
\newcommand{\risk}{\mathcal{R}}
\newcommand{\error}{\mathcal{E}}
\newcommand{\loss}{\ell_{\w}}

\newcommand{\Pem}{\widehat{\PP}_N}
\newcommand{\Pdel}{\frac{1}{N} \sum_{i=1}^{N} \delta_{(\hat{\x}_i,\hat{y}_i)}}
\newcommand{\diff}{\mathrm{d}}

\newcommand{\eps}{\rho}
\newcommand{\lip}{\mathop{\rm lip}(L)}

\newcommand{\opt}{^\star}
\newcommand{\Wball}{\mathds{B}_{\eps}(\Pem)}
\newcommand{\inner}[2]{\langle #1, #2 \rangle }
\newcommand{\optimize}[1]{\left\{ \begin{array}{cl@{\qquad}l} #1 \end{array} \right. }

\theoremstyle{definition}
\newtheorem{Thrm}{Theorem}[section]
\newtheorem{Def}[Thrm]{Definition}
\newtheorem{Asmp}[Thrm]{Assumption}

\newtheorem{Crl}[Thrm]{Corollary}

\newtheorem{Rmk}[Thrm]{Remark}
\newtheorem{Exp}[Thrm]{Example}

\newtheorem{Thrm2}{Theorem}[section]

\newtheorem{Asmp2}[Thrm2]{Assumption}

\newtheorem{Lmm2}[Thrm2]{Lemma}
\newtheorem{Rmk2}[Thrm2]{Remark}

\graphicspath{{figures/}}

\usepackage{xcolor}
\hypersetup{
	colorlinks,
	linkcolor={blue!55!black},
	citecolor={red!55!black},
	urlcolor={blue!80!black}
}
\usepackage[a4paper, total={7in, 10in}]{geometry}
\allowdisplaybreaks
\newcommand{\change}[1]{{\color{black}{#1}}} 

\title {Regularization via Mass Transportation}

\author[$\dagger$]{Soroosh Shafieezadeh-Abadeh}
\author[$\dagger$]{Daniel Kuhn}
\author[$\ddagger$]{Peyman Mohajerin Esfahani}
\affil[$\dagger$]{Risk Analytics and Optimization Chair, EPFL Lausanne \authorcr
	\texttt{soroosh.shafiee,daniel.kuhn@epfl.ch}}
\affil[$\ddagger$]{Delft Center for Systems and Control, TU Delft \authorcr
	\texttt{p.mohajerinesfahani@tudelft.nl}}

\begin{document}
	\maketitle
	\begin{abstract}
		The goal of regression and classification methods in supervised learning is to minimize the empirical risk, that is, the expectation of some loss function quantifying the prediction error under the empirical distribution. When facing scarce training data, overfitting is typically mitigated by adding regularization terms to the objective that penalize hypothesis complexity. In this paper we introduce new regularization techniques using ideas from distributionally robust optimization, and we give new probabilistic interpretations to existing techniques. Specifically, we propose to minimize the worst-case expected loss, where the worst case is taken over the ball of all (continuous or discrete) distributions that have a bounded transportation distance from the (discrete) empirical distribution. By choosing the radius of this ball judiciously, we can guarantee that the worst-case expected loss provides an upper confidence bound on the loss on test data, thus offering new generalization bounds. We prove that the resulting regularized learning problems are tractable and can be tractably kernelized for many popular loss functions. \change{The proposed approach to regluarization is also extended to neural networks.} We validate our theoretical out-of-sample guarantees through simulated and empirical experiments. 
		\\ \\
		\textit{\textbf{Keywords:}} Distributionally robust optimization, optimal transport, Wasserstein distance, robust optimization, regularization
	\end{abstract}
	
	\section{Introduction} 
	\label{sec:int}
	
	The fields of machine learning and optimization are closely intertwined. On the one hand, optimization algorithms are routinely used for the solution of classical machine learning problems. Conversely, recent advances in optimization under uncertainty have inspired many new machine learning models.
	
	From a conceptual point of view, many statistical learning tasks give naturally rise to stochastic optimization problems. Indeed, they aim to find an estimator from within a prescribed hypothesis space that minimizes the expected value of some loss function. The loss function quantifies the estimator's ability to correctly predict random outputs ({\em i.e.}, dependent variables or labels) from random inputs ({\em i.e.}, independent variables or features). Unfortunately, such stochastic optimization problems cannot be solved exactly because the input-output distribution, which is needed to evaluate the expected loss in the objective function, is not accessible and only indirectly observable through finitely many training samples. Approximating the expected loss with the empirical loss, that is, the average loss across all training samples, yields fragile estimators that are sensitive to perturbations in the data and suffer from overfitting.
	
	Regularization is the standard remedy to combat overfitting. Regularized learning models minimize the sum of the empirical loss and a penalty for hypothesis complexity, which is typically chosen as a norm of the hypothesis. There is ample empirical evidence that regularization reduces a model's generalization error. Statistical learning theory reasons that regularization implicitly restricts the hypothesis space, thereby controlling the gap between the training error and the testing error, see, {\em e.g.}, \citet{bartlett2002}. However, alternative explanations for the practical success of regularization are possible. In particular, ideas from modern robust optimization (\citet{ben2009robust}) recently led to a fresh perspective on regularization.
	
	Robust regression and classification models seek estimators that are immunized against adversarial perturbations in the training data. They have received considerable attention since the seminal treatise on robust least-squares regression by \citet{el1997robust}, who seem to be the first authors to discover an intimate connection between robustification and regularization. Specifically, they show that minimizing the worst-case residual error with respect to all perturbations in a Frobenius norm-uncertainty set is equivalent to a Tikhonov regularization procedure. \citet{xu2010robust} disclose a similar equivalence between robust least-squares regression with a feature-wise independent uncertainty set and the celebrated Lasso (least absolute shrinkage and selection operator) algorithm. Leveraging this new robustness interpretation, they extend Lasso to a wider class of regularization schemes tailored to regression problems with disturbances that are coupled across features. In the context of classification, \citet{xu2009robustness} provide a linkage between robustification over non-box-typed uncertainty sets and the standard regularization scheme of support vector machines. A comprehensive characterization of the conditions under which robustification and regularization are equivalent has recently been compiled by \citet{bertsimas2017}.
	
	New learning models have also been inspired by recent advances in the emerging field of {\em distributionally} robust optimization, which bridges the gap between the conservatism of robust optimization and the specificity of stochastic programming. Distributionally robust optimization seeks to minimize a worst-case expected loss, where the worst case is taken with respect to all distributions in an ambiguity set, that is, a family of distributions consistent with the given prior information on the uncertainty, see, \change{{\em e.g.}, \citet{calafiore2006distributionally}, \citet{delage2010distributionally}, \citet{GS10:dro}, \citet{wiesemann2014distributionally} and the references therein.} Ambiguity sets are often characterized through generalized moment conditions. \change{For instance, \citet{lanckriet2002minimax} propose a distributionally robust minimax probability machine for binary classification, where both classes are encoded by the first and second moments of their features, and the goal is to find a linear classifier that minimizes the worst-case misclassification error in view of all possible input distributions consistent with the given moment information. By construction, this approach forces the worst-case accuracies of both classes to be equal. \citet{huang2004minimum} propose a generalization of the minimax probability machine that allows for uneven worst-case classification accuracies. \citet{lanckriet2002robust} extend the minimax probability machine to account for estimation errors in the mean vectors and covariance matrices. \citet{strohmann2003formulation} and \citet{bhattacharyya2004second} develop minimax probability machines for regression and feature selection, respectively.} \citet{shivaswamy2006second} study linear classification problems trained with incomplete and noisy features, where each training sample is modeled by an ambiguous distribution with known first and second-order moments. The authors propose to address such classification problems with a distributionally robust soft margin support vector machine and then prove that it is equivalent to a classical robust support vector machine with a feature-wise uncertainty set. \citet{farnia2016minimax} investigate distributionally robust learning models with moment ambiguity sets that restrict the marginal of the features to the empirical marginal. The authors highlight similarities and differences to classical regression models.
	
	Ambiguity sets containing all distributions that share certain low-order moments are computationally attractive but fail to converge to a singleton when the number $N$ of training samples tends to infinity. Thus, they preclude any asymptotic consistency results. A possible remedy is to design spherical ambiguity sets with respect to some probability distance functions and to drive their radii to zero as $N$ grows. Examples include the $\phi$-divergence ambiguity sets proposed by \citet{ben2013robust} or the Wasserstein ambiguity sets studied by \citet{MohKun-14} and \citet{zhao2015data}. \citet{blanchet2016quantifying} and \citet{gao2016distributionally} consider generalized Wasserstein ambiguity sets defined over Polish spaces.
	
	In this paper we investigate distributionally robust learning models with Wasserstein ambiguity sets. The Wasserstein distance between two distributions is defined as the minimum cost of transporting one distribution to the other, where the cost of moving a unit point mass is determined by the ground metric on the space of uncertainty realizations. In computer science the Wasserstein distance is therefore sometimes aptly termed the `earth mover's distance' (\citet{rubner2000earth}). Following \citet{MohKun-14}, we define Wasserstein ambiguity sets as balls with respect to the Wasserstein distance that are centered at the empirical distribution on the training samples. These ambiguity sets contain all (continuous or discrete) distributions that can be converted to the (discrete) empirical distribution at bounded transportation cost. 
	
	Wasserstein distances are widely used in machine learning to compare histograms.
	\change{For example, \citet{rubner2000earth} use the Wasserstein distance as a metric for image retrieval with a focus on applications to color and texture. \citet{cuturi2013sinkhorn} and \citet{benamou2015iterative} propose fast iterative algorithms to compute a regularized Wasserstein distance between two high-dimensional discrete distributions for image classification tasks. Moreover, \citet{cuturi2014fast} develop first-order algorithms to compute the Wasserstein barycenter between several empirical probability distributions, which has applications in clustering.  \citet{arjovsky2017wasserstein} utilize the Wasserstein distance to measure the distance between the data distribution and the model distribution in generative adversarial networks. Furthermore, \citet{frogner2015learning} propose a learning algorithm based on the Wasserstein distance to predict multi-label outputs.}
	
	Distributionally robust optimization models with Wasserstein ambiguity sets were introduced to the realm of supervised learning by \citet{shafieezadeh2015distributionally}, who show that distributionally robust logistic regression problems admit a tractable reformulation and encapsulate the classical as well as the popular regularized logistic regression problems as special cases. When the Wasserstein ball is restricted to distributions on a compact set, the problem becomes intractable but can still be addressed with an efficient decomposition algorithm due to \citet{luo2017decomposition}. Support vector machine models with distributionally robust chance constraints over Wasserstein ambiguity sets are studied by~\citet{leedistributionally}. These models are equivalent to hard semi-infinite programs and can be solved approximately with a cutting plane algorithm. 
	
	Wasserstein ambiguity sets are popular for their attractive statistical properties. For example, \citet{fournier2015rate} prove that the empirical distribution on $N$ training samples converges in Wasserstein distance to the true distribution at rate $\mathcal O(N^{-1/(n+1)})$, where $n$ denotes the feature dimension. This implies that properly scaled Wasserstein balls constitute natural confidence regions for the data-generating distribution. The worst-case expected prediction loss over all distributions in a Wasserstein ball thus provides an upper confidence bound on the expected loss under the unknown true distribution; see \citet{MohKun-14}. \citet{blanchet2016robust} show, however, that radii of the order $\mathcal O(N^{-1/2})$ are asymptotically optimal even though the corresponding Wasserstein balls are too small to contain the true distribution with constant confidence. For Wasserstein distances of type two (where the transportation cost equals the squared ground metric) \citet{blanchet2017data}
	develop a systematic methodology for selecting the ground metric. Generalization bounds for the {\em worst-case} prediction loss with respect to a Wasserstein ball centered at the {\em true} distribution are derived by \citet{lee2017minimax} in order to address emerging challenges in domain adaptation problems, where the distributions of the training and test samples can differ.
	
	This paper extends the results by \citet{shafieezadeh2015distributionally} on distributionally robust logistic regression along several dimensions. Our main contributions can be summarized as follows:
	
	\begin{itemize}
		\item \textbf{Tractability:} We propose data-driven distributionally robust regression and classification models that hedge against all input-output distributions in a Wasserstein ball. We demonstrate that the emerging semi-infinite optimization problems admit equivalent reformulations as tractable convex programs for many commonly used loss functions and for spaces of linear hypotheses. We also show that lifted variants of these new learning models are kernelizable and thus offer an efficient procedure for optimizing over all nonlinear hypotheses in a reproducible kernel Hilbert space. \change{Finally, we study distributionally robust learning models over families of feed-forward neural networks. We show that these models can be approximated by regularized empirical loss minimization problems with a convex regularization term and can be addressed with a stochastic proximal gradient descent algorithm.}
		

		\item \textbf{Probabilistic Interpretation of Existing Regularization Techniques:} We show that the classical regularized learning models emerge as special cases of our framework when the cost of moving probability mass along the output space tends to infinity. In this case, the regularization function and its regularization weight are determined by the transportation cost on the input space and the radius of the Wasserstein ball underlying the distributionally robust optimization model, respectively.
		
		\item \textbf{Generalization Bounds:} We demonstrate that the proposed distributionally robust learning models enjoy new generalization bounds that can be obtained under minimal assumptions. In particular, they do not rely on any notions of hypothesis complexity and may therefore even extend to hypothesis spaces with infinite VC-dimensions. A na\"ive generalization bound is obtained by leveraging modern measure concentration results, which imply that Wasserstein balls constitute confidence sets for the unknown data-generating distribution. Unfortunately, this generalization bound suffers from a curse of dimensionality and converges slowly for high input dimensions. By imposing bounds on the hypothesis space, however, we can derive an improved generalization bound, which essentially follows a dimension-independent square root law reminiscent of the central limit theorem.
		
		\item \textbf{Relation to Robust Optimization:} In classical robust regression and classification the training samples are viewed as uncertain variables that range over a joint uncertainty set, and the best hypothesis is found by minimizing the worst-case loss over this set. We prove that the classical robust and new distributionally robust learning models are equivalent if the data satisfies a dispersion condition (for regression) or a separability condition (for classification). While there is no efficient algorithm for solving the robust learning models in the absence of this condition, the distributionally robust models are efficiently solvable irrespective of the underlying training datasets.
		
		\item \textbf{Confidence Intervals for Error and Risk:} Using distributionally robust optimization techniques based on the Wasserstein ball, we develop two tractable linear programs whose optimal values provide a confidence interval for the absolute prediction error of any fixed regressor or the misclassification risk of any fixed classifier.
		
		\item \textbf{Worst-Case Distributions:} We formulate tractable convex programs that enable us to efficiently compute a worst-case distribution in the Wasserstein ball for any fixed hypothesis. This worst-case distribution can be useful for stress tests or contamination experiments.
	\end{itemize}
	
	The rest of the paper develops as follows. Section~\ref{sec:Stat} introduces our new distributionally robust learning models. Section~\ref{sec:Tractable} provides finite convex reformulations for learning problems over linear and nonlinear hypothesis spaces and describes efficient procedures for constructing worst-case distributions. Moreover, it compares the new distributionally robust method against existing robust optimization and regularization approaches. Section~\ref{sec:Probabilistic} develops new generalization bounds, while Section~\ref{subsec:Uncertainty} addresses error and risk estimation. Numerical experiments are reported in Section~\ref{sec:Simulation}. All proofs are relegated to the appendix.

	
	\subsection{Notation} 	
	Throughout this paper, we adopt the conventions of extended arithmetics, whereby $\infty \cdot 0 = 0 \cdot \infty = 0 / 0 = 0$ and $\infty - \infty = - \infty + \infty = 1 / 0 = \infty$.
	The inner product of two vectors $\x,\bm x' \in \RR^n$ is denoted by $\inner{\x}{\bm x'}$, and for any norm $\| \cdot \|$ on $\RR^n$, we use $\| \cdot \|_*$ to denote its dual norm defined through $\| \x\|_* = \sup ~ \{ \inner{\x}{\x'}: \| \x' \| \leq 1 \}$. The conjugate of an extended real-valued function $f(\x)$ on $\RR^n$ is defined as $f^*(\bm x) = \sup_{\x'} \inner{\x}{\x'} - f(\x')$. The indicator function of a set $\XX \subseteq \RR^n$ is defined as $ \delta_\XX (\x) = 0 $ if $\x \in \XX; = \infty$ otherwise. Its conjugate $S_\XX (\bm x) = \sup \{\inner{\bm x'}{\x}: \x' \in \XX\} $ is termed the support function of $ \XX $. The characteristic function of $\XX$ is defined through $ \mathds{1}_\XX(\x) = 1 $ if $\x \in \XX$; $= 0$ otherwise. For a proper cone $\CC \subseteq \RR^n$ the relation $\x \succeq_{\CC} \x'$ indicates that $\x - \x' \in \CC$. The cone dual to $\CC$ is defined as $\CC^* = \{ \bm x': \inner{\bm x'}{\x} \geq 0\, \forall \x \in \CC \}$. The Lipschitz modulus of a function $L:\XX\rightarrow\RR$ is denoted by $\lip=\sup_{\x, \x'\in\XX} \{|L(\x)-L(\x')|/\|\x-\x'\|:\x\neq\x'\}$. If $\PP$ is a distribution on a set $\Xi$, then $\Prob^N$ denotes the $N$-fold product of $\PP$ on the Cartesian product $\Xi^N$. For $N \in \mathbb{N}$ we define $[N] = \{1, \dots , N \}$. \change{A list of commonly used symbols is provided in the following table.
		\begin{table}[h]
			\centering
			\bgroup
			\def\arraystretch{1.1}
			\begin{tabular}{|c|l||c|l|}
				\hline
				$\XX$ & input space & $\YY$ & output space \\ \hline
				$\ell$ & loss function & $L$ & univariate loss function \\ \hline
				$\HH$ & hypothesis space & $\KK$ & kernel matrix \\ \hline
				$f^*$ & conjugate of $f$ & $\mathrm{lip}(f)$ & Lipschitz modulus of $f$ \\ \hline			
				$\mathcal C^*$ & dual cone of $\mathcal C$ & $\| \cdot \|_*$ & dual norm of $\| \cdot \|$ \\ \hline
				$S_{\XX}$ & support function of $\XX$ & $ \delta_{\XX} $ & indicator function of $\XX$ \\ \hline
				$\mathds{1}_{\XX}$ & characteristic function of $\XX$ & $[N]$ & $\{1, \dots , N \}$ \\ \hline			
			\end{tabular}
			\egroup
		\end{table}
	}
	
	\section{Problem Statement} 
	\label{sec:Stat}
	We first introduce the basic terminology and then describe our new perspective on regularization.

	\subsection{Classical Statistical Learning}
	\label{sec:stat-learn}
	The goal of supervised learning is to infer an unknown target function $f:\XX\rightarrow \YY$ from limited data. The target function maps any input $\x\in \XX$ ({\em e.g.}, information on the frequency of certain keywords in an email) to some output $y\in \YY$ ({\em e.g.}, a label $+1$ ($-1$) if the email is likely (unlikely) to be a spam message). If the true target function was accessible, it could be used as a means to reliably predict outputs from inputs ({\em e.g.}, it could be used to recognize spam messages in an automated fashion). In a supervised learning framework, however, one has only access to finitely many input-output examples $(\widehat{\x}_i,\widehat{y}_i)$ for $i=1,\ldots, N$ ({\em e.g.}, a database of emails which have been classified by a human as legitimate or as spam messages). We will henceforth refer to these examples as the {\em training data} or the {\em in-sample data}. It is assumed that the training samples are mutually independent and follow an unknown distribution $\PP$ on $\XX\times\YY$.
	
	The supervised learning problems are commonly subdivided into \emph{regression problems}, where the output $y$ is continuous and $\YY= \RR$, and \emph{classification problems}, where $y$ is categorical and $\YY=\{+1,-1\}$. As the space of all functions from $\XX$ to $\YY$ is typically vast, it may be very difficult to learn the target function from finitely many training samples. Thus, it is convenient to restrict the search space to a structured family of candidate functions $\mathbb H\subseteq \RR^\XX$ such as the space of all linear functions\change{, some reproducible kernel Hilbert space or the family of all feed-forward neural networks with a fixed number of layers}. We henceforth refer to each candidate function $h\in\mathbb H$ as a {\em hypothesis} and to $\mathbb H$ as the {\em hypothesis space}. 
	
	A {\em learning algorithm} is a method for finding a hypothesis $h\in\mathbb H$ that faithfully replicates the unknown target function~$f$. Specifically, in regression we seek to approximate $f$ with a hypothesis $h$, and in classification we seek to approximate $f$ with a {\em thresholded} hypothesis $\text{sign}(h)$. Many learning algorithms achieve this goal by minimizing the  in-sample error, that is, the empirical average of a loss function $\ell:\RR \times \YY \rightarrow \RR_+$ that estimates the mismatch between the output predicted by $h(\x)$ and the actual output $y$ for a particular input-output pair $(\x,y)$. Any such algorithm solves a minimization problem of the form
	\begin{align} \label{in-sample}
	\Inf _ {h\in\mathbb H}  \bigg\{ \frac{1}{N} \sum_{i=1}^N \ell(h(\widehat{\x}_i),\widehat{y}_i) = 
	\EE^{\Pem} \left[ \ell(h(\x),y) \right] \bigg\},
	\end{align}
	where $\Pem=\Pdel$ denotes the {\em empirical distribution}, that is, the uniform distribution on the training data. For different choices of the the loss function $\ell$, the generic supervised learning problem~\eqref{in-sample} reduces to different popular regression and classification problems from the literature. 

	\subsubsection*{Examples of Regression Models}
	For ease of exposition, we focus here on learning models with $\XX\subseteq\RR^n$ and $\YY\subseteq\RR$, where $\mathbb H$ is set to the space of all linear hypotheses $h(\x)= \inner{\w}{\x}$ with $\w\in\RR^n$. Thus, there is a one-to-one correspondence between hypotheses and weight vectors $\w$. Moreover, we focus on loss functions of the form $\ell(h(\x),y)=L(h(\x)-y)=L(\inner{\w}{\x})-y)$ that are generated by a univariate loss function $L$.
	
	\begin{enumerate}
		
		\item A rich class of {\em robust regression} problems is obtained from \eqref{in-sample} if $\ell$ is generated by the {\em Huber loss function} with robustness parameter $\delta> 0$, which is defined as $L(z)= \frac{1}{2}z^2$ if $|z|\leq \delta$; $= \delta( |z|-\frac{1}{2}\delta)$ otherwise. Note that the Huber loss function is both convex and smooth and reduces to the squared loss $L(z)=\frac{1}{2}z^2$ for $\delta\uparrow\infty$, which is routinely used in ordinary least squares regression. Problem~\eqref{in-sample} with squared loss seeks a hypothesis $\w$ under which $\inner{\w}{\x}$ approximates the mean of $y$ conditional on~$\x$. The Huber loss function for finite $\delta$ favors similar hypotheses but is less sensitive to outliers.
		
		\item The {\em support vector regression} problem \citep{smola2004tutorial} emerges as a special case of \eqref{in-sample} if $\ell$ is generated by the {\em $\epsilon$-insensitive loss function} $L(z) = \max  \{0, | z| - \epsilon \}$ with $\epsilon\geq 0$. In this setting, a training sample $(\widehat{\x}_i,\widehat y_i)$ is penalized in~\eqref{in-sample} only if the output $\inner{\w}{\widehat{\x}_i}$ predicted by hypothesis $\w$ differs from the true output $\widehat y_i$ by more than $\epsilon$. Support vector regression thus seeks hypotheses $\w$ under which all training samples reside within a slab of width $2\epsilon$ centered around the hyperplane $\{(\x,y): \inner{\w}{\x}=y\}$.
		
		\item The {\em quantile regression} problem \citep{koenker2005quantile} is obtained from \eqref{in-sample} if  $\ell$ is generated by the {\em pinball loss function} $L(z) = \max \{ -\tau z, (1-\tau)z \}$ defined for $\tau\in[0,1]$. Quantile regression seeks hypotheses that approximate the $\tau\times100\%$-quantile of the output conditional on the input. More precisely, it seeks hypotheses $\w$ for which $\tau\times 100\%$ of all training samples lie in the halfspace $\{(\x,y): \inner{\w}{\x} \geq y \}$.

	\end{enumerate}
	
	\subsubsection*{Examples of Classification Models}
	We focus here on linear learning models with $\XX\subseteq \RR^n$ and $\YY=\{+1,-1\}$, where $\mathbb H$ is again identified with the space of all linear hypotheses $h(\x)=\inner{\w}{\x}$ with $\w\in\RR^n$. Moreover, we focus on loss functions of the form $\ell(\x,y)=L(yh(\x))=L(y\inner{\w}{\x})$ generated by a univariate loss function $L$.
	
	\begin{enumerate}
		\item The {\em support vector machine} problem \citep{cortes1995support} is obtained from~\eqref{in-sample} if $\ell$ is generated by the {\em hinge loss function} $L(z) = \max \{ 0, 1 - z\}$, which is large if $z$ is small. Thus, a training sample $(\widehat{\x}_i,\widehat y_i)$ is penalized in~\eqref{in-sample} if the output $\text{sign}(\inner{\w}{\widehat{\x}_i})$ predicted by hypothesis $\w$ and the true output $\widehat y_i$ have opposite signs. More precisely, support vector machines seek hypotheses $\w$ under which the inputs of all training samples with output $+1$ reside in the halfspace $\{\x:\inner{\w}{\x}\geq 1\}$, while the inputs of training samples with output $-1$ are confined to $\{\x:\inner{\w}{\x}\leq -1\}$. 
		\item An alternative support vector machine problem is obtained from~\eqref{in-sample} if $\ell$ is generated by the  {\em smooth hinge loss function}, which is defined as $L(z) = \frac{1}{2} - z$ if $z \leq 0$; $=\frac{1}{2}(1-z)^2$ if $0<z<1$; $=0$ otherwise. The smooth hinge loss inherits many properties of the ordinary hinge loss but has a continuous derivative. Thus, it may be amenable to faster optimization algorithms.
		\item The {\em logistic regression} problem \citep{hosmer2013applied} emerges as a special case of \eqref{in-sample} if $\ell$ is generated by the {\em logloss function} $L(z)= \log ( 1 + e^{-z})$, which is large if $z$ is small---similar to the hinge loss function. In this case the objective function of \eqref{in-sample} can be viewed as the log-likelihood function corresponding to the logistic model $\PP(y=1|\x)=[1+\exp(-\inner{\w}{\x})]^{-1}$ for the conditional probability of $y=1$ given $\x$. Thus, logistic regression allows us to learn the conditional distribution of $y$ given $\x$.
	\end{enumerate}	
	
	\begin{Rmk} [Convex approximation]
		Note that the hinge loss and the logloss functions represent convex approximations for the (discontinuous) {\em one-zero loss} defined through $L(z)=1$ if $z\leq 0$; $=0$ otherwise.
	\end{Rmk}

	In practice there may be many hypotheses that are compatible with the given training data and thus achieve a small empirical loss in~\eqref{in-sample}. Any such hypothesis would accurately predict outputs from inputs {\em within the training dataset} \citep{defourny2010machine}. However, due to overfitting, these hypotheses might constitute poor predictors {\em beyond the training dataset}, that is, on inputs that have not yet been recorded in the database. Mathematically, even if the in-sample error $\EE^{\Pem} [ \ell(\inner{\w}{\x},y) ]$ of a given hypothesis $\w$ is small, the out-of-sample error $\EE^{\PP} [ \ell(\inner{\w}{\x},y)]$ with respect to the unknown true input-output distribution $\PP$ may be~large.
	
	Regularization is the standard remedy to combat overfitting. Instead of na\"ively minimizing the in-sample error as is done in~\eqref{in-sample}, it may thus be advisable to solve the regularized learning problem
	\begin{align} \label{in-sample-regularized}
	\Inf _ {\w} ~ \EE^{\Pem} \left[ \ell(\inner{\w}{\x},y) \right] +c\, \Omega(\w),
	\end{align}
	which minimizes the sum of the emiprial average loss and a penalty for hpothesis complexity, which consists of a regularization function $\Omega(\w)$ and its associated regularization weight $c$. Tikhonov regularization~\citep{tikhonov1977solutions}, for example, corresponds to the choice $\Omega(\w)=\|\boldsymbol \Gamma\w\|_2^2$ for some Tikhonov matrix $\boldsymbol \Gamma\in \RR^{n\times n}$. Setting $\boldsymbol \Gamma$ to the identity matrix gives rise to standard $L_2$-regularization. Similarly, Lasso (least absolute shrinkage and selection operator) regularization or $L_1$-regularization~\citep{Tibshirani94regressionshrinkage} is obtained by setting $\Omega(\w)=\|\w\|_1$. Lasso regularization has gained popularity because it favors parsimonious interpretable hypotheses.

	Most popular regualization methods admit probabilistic interpretations. However, these interpretations typically rely on prior distributional assumptions that remain to some extent arbitrary ({\em e.g.}, $L_2$- and $L_1$-regularization can be justified if $\w$ is governed by a Gaussian or Laplacian prior distribution, respectively~\citep{Tibshirani94regressionshrinkage}). Thus, in spite of their many desirable theoretical properties, there is a consensus that ``{\em most of the (regularization) methods used successfully in practice are heuristic methods}'' \citep{abu2012learning}.

	\subsection{A New Perspective on Regularization}
	
	When {\em linear} hypotheses are used, problem~\eqref{in-sample} minimizes the in-sample error $\EE^{\Pem} [ \ell(\inner{\w}{\x},y) ]$. However, a hypothesis $\w$ enjoying a low in-sample error may still suffer from a high out-of-sample error $\EE^{\PP} [ \ell(\inner{\w}{\x},y)]$ due to overfitting. This is unfortunate as we seek hypotheses that offer high prediction accuracy on {\em future} data, meaning that the out-of-sample error is the actual quantity of interest. An ideal learning model would therefore minimize the out-of-sample error. This is impossible, however, for the following reasons:
	\begin{itemize}
		\item The true input-output distribution $\PP$ is unknown and only indirectly observable through the $N$ training samples. Thus, we lack essential information to compute the out-of-sample error.
		\item Even if the distribution $\PP$ was known, computing the out-of-sample error would typically be hard due to the intractability of high-dimensional integration; see, {\em e.g.}, \citep[Corollary~1]{hanasusanto2016comment}. 
	\end{itemize}
	The regularized loss $\EE^{\Pem} [ \ell(\inner{\w}{\x},y) ]+c \,\Omega(\w)$ used in~\eqref{in-sample-regularized}, which consists of the in-sample error and an overfitting penalty, can be viewed as an in-sample estimate of the out-of-sample error. Yet, problem~\eqref{in-sample-regularized} remains difficult to justify rigorously. Therefore, we advocate here a more principled approach to regularization. Specifically, we propose to take into account the expected loss of hypothesis $\w$ under {\em every} distribution $\QQ$ that is close to the empirical distribution $\Pem$, that is, every $\QQ$ that could have generated the training data with high confidence. To this end, we first introduce a distance measure for distributions. For ease of notation, we henceforth denote the input-output pair $(\x,y)$ by $\xxi$, and we set $\Xi= \XX\times\YY$. 	
	\begin{Def} [Wasserstein metric] \label{Def:Wass}
		The Wasserstein distance between two distributions $\QQ$ and $\QQ'$ supported on $\Xi$ is defined as
		\begin{align*}
		W (\QQ,\QQ') := \Inf _ {\Pi}  \left\{
		\int_{\Xi^2} d(\xxi, \xxi') \, \Pi(\diff \xxi, \diff \xxi') ~: \begin{array}{l} \Pi \text{ is a joint distribution of $\xxi$ and $\xxi'$} \\ \text{with marginals $\mathds{Q}$ and $\mathds{Q}'$, respectively} \end{array} 
		\right\},
		\end{align*}
		where $d$ is a metric on $\Xi$. 
	\end{Def}
	
	By definition, $W (\QQ,\QQ')$ represents the solution of an infinite-dimensional transportation problem, that is, it corresponds to the minimal cost for moving the distribution $\QQ$ to $\QQ'$, where the cost for moving a unit probability mass from $\xxi$ to $\xxi'$ is given by the transportation distance $d(\xxi,\xxi')$. Due to this interpretation, the metric $d$ is often referred to as the transportation cost \citep{villani2008optimal} or ground metric \citep{cuturi2014ground}, while the Wasserstein metric is sometimes termed the mass transportation distance or earth mover's distance \citep{rubner2000earth}.

	Consider now the Wasserstein ball  of radius $\eps\geq 0$ around the empirical distribution $\Pem$,
	\begin{align} \label{wass_ball}
	\BB _ \eps (\Pem)=\big\{ \QQ\; :\; \QQ(\Xi)=1,\; W (\QQ,\Pem) \leq \eps \big\},
	\end{align}
	which contains all input-output distributions $\QQ$ supported on $\Xi$ whose Wasserstein distance from $\Pem$ does not exceed~$\eps$. This means that $\QQ$ can be transported to $\Pem$ (or vice versa) at a cost of at most~$\eps$. The hope is that a large enough Wasserstein ball will contain distributions that are representative of the unknown true input-output distribution $\PP$, such that the worst-case expectation $\sup _ {\QQ \in \Wball} \EE ^ \QQ[ \ell(\inner{\w}{\x},y)]$ can serve as an upper confidence bound on the out-of-sample error $\EE ^ \PP[ \ell(\inner{\w}{\x},y)]$. This motivates us to introduce a new regularized learning model, which minimizes precisely this worst-case expectation.
	\begin{align} \label{distrob:wass} 
	\Inf _ {\w} \Sup _ {\QQ \in \Wball} \EE ^ \QQ \left[ \ell(\inner{\w}{\x},y) \right]
	\end{align}
	Problem~\eqref{distrob:wass} represents a distributionally robust convex program of the type considered in~\citep{MohKun-14}. Note that if $\ell(\inner{\w}{\x},y)$ is convex in $\w$ for every fixed $(\x,y)$, {\em i.e.}, if $\ell$ is convex in its first argument, then the objective function of \eqref{distrob:wass} is convex because convexity is preserved under integration and maximization. Note also that if $\eps$ is set to zero, then~\eqref{distrob:wass} collapses to the unregularized in-sample error minimization problem~\eqref{in-sample}. 
	
	\change{
	\begin{Rmk}[Support information]
		The uncertainty set~$\Xi$ captures prior information on the range of the inputs and outputs. In image processing, for example, pixel intensities range over a known interval. Similarly, in diagnostic medicine, physiological parameters such as blood glucose or cholesterol concentrations are restricted to be non-negative. Sometimes it is also useful to construct~$\Xi$ as a confidence set that covers the support of~$\PP$ with a prescribed probability. Such confidence sets are often constructed as ellipsoids, as intersections of different norm balls \citep{ben2009robust,delage2010distributionally} or as sublevel sets of kernel expansions \citep{scholkopf2001estimating}.
	\end{Rmk}
	}
	
	In the remainder we establish that the distributionally robust learning problem~\eqref{distrob:wass} has several desirable properties. (i)~Problem~\eqref{distrob:wass} is computationally tractable under standard assumptions about the loss function~$\ell$, the input-output space~$\Xi$ and the transportation metric~$d$. For specific choices of $d$ it even reduces to a regularized learning problem of the form~\eqref{in-sample-regularized}. (ii)~For all univariate loss functions reviewed in Section~\ref{sec:stat-learn}, a tight conservative approximation of~\eqref{distrob:wass} is kernelizable, that is, it can be solved implicitly over high-dimensional spaces of nonlinear hypotheses at the same computational cost required for linear hypothesis spaces. (iii)~Leveraging modern measure concentration results, the optimal value of~\eqref{distrob:wass} can be shown to provide an upper confidence bound on the out-of-sample error. This obviates the need to mobilize the full machinery of VC theory and, in particular, to estimate the VC dimension of the hypothesis space in order to establish generalization bounds. (iv)~If the number of training samples tends to infinity while the Wasserstein ball shrinks at an appropriate rate, then problem~\eqref{distrob:wass} asymptotically recovers the {\em ex post} optimal hypothesis that attains the minimal out-of-sample error.	
	
	\section{Tractable Reformulations}
	\label{sec:Tractable}
	
	In this section we demonstrate that the distributionally robust learning problem~\eqref{distrob:wass} over linear hypotheses is amenable to efficient computational solution procedures. \change{We also discuss generalizations to nonlinear hypothesis classes such as reproducing kernel Hilbert spaces and families of feed-forward neural networks.}
	
	\subsection{Distributionally Robust Linear Regression}
	\label{subsec:drregression}
	Throughout this section we focus on linear regression problems, where $\ell(\inner{\w}{\x},y)=L(\inner{\w}{\x}-y)$ for some convex univariate loss function $L$. We also assume that $\XX$ and $\YY$ are both convex and closed and that the transportation metric $d$ is induced by a norm $\|\cdot\|$ on the input-output space $\RR^{n+1}$. In this setting, the distributionally robust regression problem \eqref{distrob:wass} admits an equivalent reformulation as a finite convex optimization problem if either (i) the univariate loss function $L$ is piecewise affine or (ii) $\Xi=\RR^{n+1}$ and $L$ is Lipschitz continuous (but not necessarily piecewise affine). 	
	
	\begin{Thrm} [Distributionally robust linear regression]
		\label{Thrm:Regression}		
		The following statements hold.
		\begin{enumerate}
			\item[(i)] If $L(z)= \max_{j\leq J} \{ a_j z + b_j \}$, then \eqref{distrob:wass} is equivalent to
			\begin{align} \label{tractable:regression:linear}
			\optimize{
				\Inf_{ \substack{\w, \lambda,s_i \\ \p_{ij}, u_{ij}} } & \displaystyle \Obj & \\
				\mathrm{s.t.} & S_\Xi(a_j \w - \p_{ij}, -a_j - u_{ij}) + b_j + \inner{\p_{ij}} {\widehat{\x}_i} + u_{ij} \widehat{y}_i \leq s_i & i \in [N], j \in [J] \\
				& \| (\p_{ij}, u_{ij}) \|_* \leq \lambda & i \in [N], j \in [J],}	
			\end{align}
			where $S_{\Xi}$ denotes the support function of $\Xi$.
			\item[(ii)] If $\Xi = \RR^{n+1}$ and $L(z)$ is Lipschitz continuous, then \eqref{distrob:wass} is equivalent to
			\begin{align} \label{tractable:regression:convex}
			\Inf_{\w}  ~  \frac{1}{N} \sum_{i=1}^N L (\inner{\w}{\widehat{\x}_i} - \widehat{y}_i) + \eps \lip \, \| (\w,-1) \|_*.
			\end{align}
		\end{enumerate}	
	\end{Thrm}
	
	In the following, we exemplify Theorem~\ref{Thrm:Regression} for the Huber, $\epsilon$-insensitive and pinball loss functions under the assumption that the uncertainty set $\Xi$ admits the conic representation
	\begin{equation}
	\label{uncertainty-set}
	\Xi = \{ (\x, y) \in \RR^{n+1} : \bm C_1 \x + \bm c_2 y \preceq_{\CC} \bm d \}
	\end{equation}
	for some matrix $\bm C_1$, vectors $\bm c_2$ and $\bm d$ and proper convex cone $\CC$ of appropriate dimensions. We also assume that $\Xi$ admits a Slater point $(\x_{\rm S}, y_{\rm S}) \in \RR^{n+1}$ with $ \bm C_1 \x_{\rm S} + \bm c_2 y_{\rm S} \prec_{\CC} \bm d$.

	\begin{Crl} [Robust regression]
		\label{Crl:hr}
		If $L$ represents the Huber loss function with threshold $\delta \geq 0$ and $\Xi=\RR^{n+1}$, then \eqref{distrob:wass} is equivalent to
		\begin{align} \label{Huber:rob}
		\Inf_{\w, z_i} ~ \frac{1}{N} \SumN \frac{1}{2} z_i^2 + \delta | \inner{\w}{\widehat{\x}_i} - \widehat{y}_i - z_i| +  \eps \, \delta \, \| (\w,-1) \|_*.
		\end{align}		
	\end{Crl}
	
	\begin{Crl} [Support vector regression] 
		\label{Crl:svr}
		If $L$ represents the $\epsilon$-insensitive loss function for some $\epsilon\geq 0$ and $\Xi$ is of the form~\eqref{uncertainty-set}, then \eqref{distrob:wass} is equivalent to
		\begin{align} \label{SVR:rob}
		\optimize{
			\Inf_{\substack{\w, \lambda, s_i \\ \p^+_i,\p^-_i}} & \displaystyle \Obj & \\ [0.2em]
			\mathrm{s.t.} & \widehat{y}_i - \inner{\w}{\widehat{\x}_i} - \epsilon + \inner{\p^+_i}{\bm d - \bm C_1 \widehat{\x}_i - \bm c_2 \widehat{y}_i} \leq s_i & i \in [N] \\	[0.2em]
			& \inner{\w}{\widehat{\x}_i} - \widehat{y}_i - \epsilon + \inner{\p^-_i}{\bm d - \bm C_1 \widehat{\x}_i - \bm c_2 \widehat{y}_i} \leq s_i & i \in [N] \\ [0.2em]	
			& \left\| \left( \bm C_1^\top \p^+_i + \w, \bm c_2^\top \p^+_i - 1 \right) \right\|_* \leq \lambda & i \in [N] \\ [0.2em]
			& \left\| \left( \bm C_1^\top \p^-_i - \w, \bm c_2^\top \p^-_i + 1 \right) \right\|_* \leq \lambda & i \in [N] \\ [0.2em]	
			& \p^+_i,\p^-_i \in \CC^* & i \in [N] \\ [0.2em]
			& s_i \geq 0 & i \in [N].}
		\end{align}	
	\end{Crl}
	
	\begin{Crl} [Quantile regression]
		\label{Crl:qr}
		If $L$ represents the pinball loss function for some $\tau\in [0,1]$ and $\Xi$ is of the form~\eqref{uncertainty-set}, then \eqref{distrob:wass} is equivalent to
		\begin{align} \label{Quant:rob}
		\optimize{
			\Inf_{\substack{\w, \lambda, s_i \\ \p^+_i,\p^-_i}} & \displaystyle \Obj  & \\
			\mathrm{s.t.} & \tau \left(\widehat y_i - \inner{\w}{\widehat \x_i} \right)  + \inner{\p^+_i}{\bm d - \bm C_1 \widehat{\x}_i - \bm c_2 \widehat{y}_i} \leq s_i & i \in [N] \\ [0.2em]
			& (1-\tau) \left( \inner{\w}{\widehat \x_i} - \widehat y_i \right) + \inner{\p^-_i}{\bm d - \bm C_1 \widehat{\x}_i - \bm c_2 \widehat{y}_i} \leq s_i & i \in [N] \\ [0.2em]
			& \left\| \left( \bm C_1^\top \p^+_i + \tau \w, \bm c_2^\top \p^+_i - \tau \right) \right\|_* \leq \lambda & i \in [N] \\ [0.2em]
			& \left\| \left( \bm C_1^\top \p^-_i - ( 1 - \tau ) \w, \bm c_2^\top \p^-_i + 1 - \tau \right) \right\|_* \leq \lambda & i \in [N] \\ [0.2em]
			& \p^+_i,\p^-_i \in \CC^* & i \in [N] \\ [0.2em]
			& s_i \geq 0 & i \in [N].}
		\end{align}	
	\end{Crl}
	
	\begin{Rmk} [Relation to classical regularization]
		Assume now that the mass transportation costs are additively separable with respect to inputs and outputs, that is,
		\begin{equation} \label{eq:metric:additive}
		d((\x_1,y_1),(\x_2,y_2))= \| \x_1 - \x_2 \| + \kappa | y_1 - y_2 |
		\end{equation}
		for some $\kappa>0$.\footnote{By slight abuse of notation, the symbol $\|\cdot\|$ now denotes a norm on $\RR^n$.} Note that $\kappa$ captures the costs of moving probability mass along the output space. For $\kappa = \infty$ all distributions in the Wasserstein ball $\Wball$ are thus obtained by reshaping $\Pem$ only along the input space. It is easy to verify that for $\kappa = \infty$ and $ \Xi = \RR^{n+1}$ the learning models portrayed in Corollaries~\ref{Crl:hr}-\ref{Crl:qr} all simplify to 
		\begin{align} \label{reg_reg}
		\Inf_{\w} ~  \frac{1}{N} \SumN L(\inner{\w}{\widehat{\x}_i}- \widehat{y}_i) + c \| \w \|_*,
		\end{align}	
		where $c = \eps \delta$ for robust regression with Huber loss, $c=\eps$ for support vector regression with $\epsilon$-insensitive loss and $c=\max \{\tau, 1-\tau \} \, \eps$ for quantile regression with pinball loss. Thus, \eqref{reg_reg} is easily identified as an instance of the classical regularized learning problem~\eqref{in-sample-regularized}, where the dual norm term $\| \w \|_*$ plays the role of the regularization function, while $c$ represents the usual regularization weight. 
		By definition of the dual norm, the penalty $\| \w \|_*$ assigned to a hypothesis $\w$ is maximal (minimal) if the cost of moving probability mass along $\w$ is minimal (maximal). We emphasize that if $\kappa=\infty$, then the marginal distribution of $y$ corresponding to {\em every} $\QQ \in \Wball$ coincides with the empirical distribution $\frac{1}{N}\sum_{i=1}^N \delta_{\widehat y_i}$. Thus, classical regularization methods, which correspond to $\kappa = \infty$, are explained by a counterintuitive probabilistic model, which pretends that any training sample must have an output that has already been recordeded in the training dataset. In other words, classical regularization implicitly assumes that there is no uncertainty in the outputs. More intuitively appealing regularization schemes are obtained for finite values of $\kappa$.
	\end{Rmk}
	
	To establish a connection between distributionally robust and classical robust regression as discussed in \citep{el1997robust, xu2010robust}, we further investigate the worst-case expected loss of a fixed linear hypothesis $\w$.
	\begin{equation}
	\label{wc-expectation-regression}
	\sup_ {\QQ \in \Wball} \EE ^ \QQ [ L(\inner{\w}{\x}-y) ]
	\end{equation}
	\begin{Thrm} [Extremal distributions in linear regression]
		\label{Thrm:worst-case:regression}
		The following statements hold.
		\begin{enumerate}
			\item[(i)] If $L(z)= \max_{j\leq J} \{ a_j z + b_j \}$, then the worst-case expectation~\eqref{wc-expectation-regression} coincides with
			\begin{align} \label{worst:regression:linear}
			& \optimize{
				\Sup_{ \alpha_{ij}, \q_{ij}, v_{ij} } & \displaystyle \frac{1}{N} \sum_{i=1}^N\sum_{j=1}^J \alpha_{ij} \big( a_j (\inner{\w}{\widehat \x_i} - \widehat y_i) + b_j \big) + a_j (\inner{\w}{\q_{ij}} - v_{ij}) & \\
				\mathrm{s.t.} &  \displaystyle  \frac{1}{N}\sum_{i=1}^N\sum_{j=1}^J \| (\q_{ij},v_{ij}) \| \leq \eps &  \\
				& \displaystyle\sum_{j=1}^J \alpha_{ij} = 1 & \hspace{-3cm}  i \in [N] \\
				& \left( \widehat{\x}_i + \q_{ij}/\alpha_{ij}, \widehat y_i + v_{ij} / \alpha_{ij} \right) \in \Xi & \hspace{-3cm} i \in [N], j \in [J] \\
				& \alpha_{ij} \geq 0 &  \hspace{-3cm} i \in [N], j \in [J]}
			\end{align}	
			for any fixed hypothesis $\w$. Moreover, if $(\alpha\opt_{ij}, \q\opt_{ij}, v\opt_{ij})$ maximizes~\eqref{worst:regression:linear}, then the discrete distribution
			\begin{align*}
			\QQ^* = \frac{1}{N} \sum_{i = 1}^{N}\sum_{j = 1}^{J} \alpha\opt_{ij}\, \delta_{(\widehat{\x}_i + \q\opt_{ij} / \alpha\opt_{ij}, \; \widehat y_i + v\opt_{ij} / \alpha\opt_{ij})},
			\end{align*}
			represents a maximizer for \eqref{wc-expectation-regression}.
			\item[(ii)] If $\Xi = \RR^{n+1}$ and $L(z)$ is Lipschitz continuous, then the discrete distributions
			\begin{align*}
			\QQ_\gamma = \frac{1}{N} \sum_{i = 2}^{N} \delta_{(\widehat \x_{i}, \,\widehat y_i)}  + \frac{1 - \gamma}{N} \delta_{(\widehat \x_{1}, \,\widehat y_1)} + \frac{\gamma}{N} \delta_{(\widehat \x_1 + \frac{\eps N}{\gamma} \x\opt, \widehat y_1 + \frac{\eps N}{\gamma} y\opt) } \quad\text{for}\quad \gamma\in(0,1],
			\end{align*}
			where $(\x\opt, y\opt)$ solves $\max_{\| (\x, y) \| \leq 1}\inner{\w}{\x} - y$, are feasible and asymptotically optimal in \eqref{wc-expectation-regression} for $\gamma\downarrow 0$.
		\end{enumerate}	
	\end{Thrm}
	Recall that $0/0 = 0$ and $1/0 = \infty$ by our conventions of extended arithmetic. Thus, any solution feasible in~\eqref{worst:regression:linear} with $\alpha_{ij}=0$ must satisfy $\q_{ij} = \bm 0$ and $v_{ij} = 0$ because otherwise $\left( \widehat{\x}_i + \q_{ij}/\alpha_{ij}, \widehat y_i + v_{ij} / \alpha_{ij} \right) \notin \Xi$.
	
	Theorem~\ref{Thrm:worst-case:regression} shows how one can use convex optimization to construct a sequence of distributions that are asymptotically optimal in~\eqref{wc-expectation-regression}. Next, we argue that the worst-case expected cost~\eqref{wc-expectation-regression} is equivalent to a (robust) worst-case cost over a suitably defined uncertainty set if the following assumption holds.
	\begin{Asmp}[Minimal dispersion]
		\label{Asmp:sep:reg}
		For every $\w \in \RR^n$ there is a training sample $(\widehat \x_k, \widehat y_k)$ for some $k\leq N$ such that the derivative $L'$ exists at $\inner{\w}{\widehat \x_k}- \widehat y_k$ and satisfies $|L'(\inner{\w}{\widehat \x_k}- \widehat y_k)| = \lip$.
	\end{Asmp}
	
	\begin{Rmk} [Minimal dispersion]
		Assumption~\ref{Asmp:sep:reg} is reminiscent of the non-separability condition in~\citep[Theorem~3]{xu2009robustness}, which is necessary to prove the equivalence of robust and regularized support vector machines. In the regression context studied here, Assumption~\ref{Asmp:sep:reg} ensures that, for every $\w$, there exists a training sample that activates the largest absolute slope of $L$.
		
		For instance, in support vector regression, it means that for every $\w$ there exists a data point outside of the slab of width $2\epsilon/\|(\w,-1)\|_2$ centered around the hyperplane $H_{\w}=\{(\x, y) \in \RR^n \times \RR: \inner{\w}{\x} - y = 0 \}$ ({\em i.e.}, the empirical $\epsilon$-insensitive loss is not zero). Similarly, in robust regression with the Huber loss function, Assumption~\ref{Asmp:sep:reg} stipulates that for every $\w$ there exists a data point outside of the slab of width $2\delta/\|(\w,-1)\|_2$ centered around $H_{\w}$.
		However, quantile regression with $\tau \neq 0.5$ fails to satisfy Assumption~\ref{Asmp:sep:reg}. Indeed, for any training dataset there always exists some $\w$ such that all data points reside on the side of $H_{\w}$ where the pinball loss function is less steep.
	\end{Rmk}
	
	\begin{Thrm}[Robust regression] \label{Thrm:RobConnection}
		If $\Xi = \RR^{n+1}$ and the loss function $L(z)$ is Lipschitz continuous, then the worst-case expected loss~\eqref{wc-expectation-regression} provides an upper bound on the (robust) worst-case loss
		\begin{align}
		\label{wc-regression-loss}
		\optimize{
			\Sup_{\Delta \x_{i}, \Delta y_{i}} & \displaystyle \frac{1}{N} \sum_{i=1}^N  [ L(\inner{\w}{\widehat \x_i + \Delta \x_{i}}-\widehat y_i - \Delta y_{i})]\\
			\mathrm{s.t.} & \displaystyle \frac{1}{N}\sum_{i=1}^N \| (\Delta \x_{i}, \Delta y_{i}) \| \leq  \eps. & }
		\end{align}
		Moreover, if Assumption~\ref{Asmp:sep:reg} holds, then \eqref{wc-expectation-regression} and \eqref{wc-regression-loss} are equal.	
	\end{Thrm}
	
	\begin{Rmk}[Tractability of robust regression]
		Assume that $\Xi=\RR^{n+1}$, while $L$ and $\|\cdot\|$ both admit a tractable conic representation. By Theorem~\ref{Thrm:Regression}, the worst-case {\em expected} loss~\eqref{wc-expectation-regression} can then be computed in polynomial time by solving a tractable convex program. Theorem~\ref{Thrm:RobConnection} thus implies that the worst-case loss~\eqref{wc-classification-loss} can also be computed in polynomial time if Assumption~\ref{Asmp:sep:reg} holds. To our best knowledge, there exists no generic efficient method for computing~\eqref{wc-classification-loss} if Assumption~\ref{Asmp:sep:reg} fails to hold and $L$ is {\em not} piecewise affine. This reinforces our belief that a distributionally robust approach to regression is more natural.
	\end{Rmk}
	
	\subsection{Distributionally Robust Linear Classification}
	\label{subsec:drclassification}
	Throughout this section we focus on linear classification problems, where $\ell(\inner{\w}{\x},y)=L(y\inner{\w}{\x})$ for some convex univariate loss function $L$. We also assume that $\XX$ is both convex and closed and that $\YY=\{+1,-1\}$. Moreover, we assume that the transportation metric $d$ is defined via
	\begin{align} \label{metric}
	\change{d((\x,y),(\x',y')) = \|\x-\x'\| + \kappa \mathds{1}_{ \{ y \neq y' \} } },
	\end{align}
	where $\|\cdot\|$ represents a norm on the input space $\RR^n$, and $\kappa>0$ quantifies the cost of switching a label. In this setting, the distributionally robust classification problem~\eqref{distrob:wass} admits an equivalent reformulation as a finite convex optimization problem if either ({\em i}) the univariate loss function $L$ is piecewise affine or ({\em ii}) $\XX=\RR^{n}$ and $L$ is Lipschitz continuous (but not necessarily piecewise affine).	
	\begin{Thrm} [Distributionally robust linear classification]
		\label{Thrm:Classification}
		The following statements hold. 
		\begin{enumerate} [label=\Roman*.]
			\item[(i)] If $L(z)= \max_{j\in J} \{ a_j z + b_j \}$, then \eqref{distrob:wass} is equivalent to
			\begin{align} \label{tractable:classification:linear}
			\optimize{
				\Inf_{ \substack{\w, \lambda,s_i \\ \p^+_{ij}, \p^-_{ij} } } & \displaystyle \Obj & \\
				\mathrm{s.t.} & S_\XX(a_j \widehat y_i \w - \p^+_{ij}) + b_j + \inner{\p^+_{ij}}{\widehat \x_i} \leq s_i & i \in [N], j \in [J] \\
				& S_\XX(-a_j \widehat y_i \w - \p^-_{ij}) + b_j + \inner{\p^-_{ij}}{\widehat \x_i} - \kappa \lambda \leq s_i & i \in [N], j \in [J] \\
				& \| \p^+_{ij} \|_* \leq \lambda, \| \p^-_{ij} \|_* \leq \lambda & i \in [N], j \in [J],}
			\end{align}
			where $S_{\XX}$ denotes the support function of $\XX$.
			\item[(ii)] If $\XX = \RR^{n}$ and $L$ is Lipschitz continuous, then \eqref{distrob:wass} is equivalent to		
			\begin{align} \label{tractable:classification:convex}
			\optimize{
				\Inf_{\w,\lambda,s_i} & \displaystyle \Obj & \\
				\text{s.t.} & L (\widehat{y}_i \inner{\w}{\widehat{\x}_i}) \leq s_i & i \in [N] \\
				& L (-\widehat{y}_i \inner{\w}{\widehat{\x}_i}) - \kappa \lambda \leq s_i & i \in [N] \\
				& \lip \| \w \|_* \leq \lambda.}
			\end{align}
		\end{enumerate}	
	\end{Thrm}
	In the following, we exemplify Theorem~\ref{Thrm:Classification} for the hinge loss, logloss and smoothed hinge loss functions under the assumption that the input space $\XX$ admits the conic representation
	\begin{equation}
	\label{input-set}
	\XX = \{ \x \in \RR^n : \bm C \x \preceq_{\CC} \bm d \}
	\end{equation}
	for some matrix $\bm C$, vector $\bm d$ and proper convex cone $\CC$ of appropriate dimensions. We also assume that $\XX$ admits a Slater point $\x_{\rm S}\in\RR^n$ with $ \bm C \x_{\rm S} \prec_{\CC} \bm d$.
	
	\begin{Crl} [Support vector machine]
		\label{Crl:svm}
		If $L$ represents the hinge loss function and $\XX$ is of the form~\eqref{input-set}, then \eqref{distrob:wass} is equivalent to
		\begin{align} \label{HLP}
		\optimize{
			\Inf_{ \substack{\w,\lambda \\ s_i,\p^+_i,\p^-_i} } & \displaystyle \Obj & \\
			\mathrm{s.t.} & 1 - \widehat{y}_i \left( \inner{\w}{\widehat{\x}_i} \right) + \inner{\p^+_i}{\bm d - \bm C \widehat{\x}_i} \leq s_i & i \in [N] \\ [0.2em]
			& 1 + \widehat{y}_i \left( \inner{\w}{\widehat{\x}_i}  \right) + \inner{\p^-_i}{\bm d - \bm C \widehat{\x}_i} - \kappa \lambda \leq s_i & i \in [N] \\ [0.2em]
			& \| \bm C^\top \p^+_i + \widehat{y}_i \w \|_* \leq \lambda,\quad  \| \bm C^\top \p^-_i - \widehat{y}_i \w \|_* \leq \lambda & i \in [N] \\ [0.2em]
			&  s_i \geq 0,\quad \p^+_i, \p^-_i \in \CC^* & i \in [N].}
		\end{align}
	\end{Crl}
	
	\begin{Crl} [Support vector machine with smooth hinge loss]
		\label{Crl:hrc}
		If $L$ represents the smooth hinge loss function and $\XX=\RR^n$, then \eqref{distrob:wass} is equivalent~to
		\begin{align} \label{MHLM}
		\optimize{
			\Min_{\substack{\w,\lambda, s_i \\ z_i^+, z_i^-, t_i^+, t_i^-}} & \displaystyle \Obj & \\ 
			\text{s.t.} & \frac{1}{2} (z_i^+ - \widehat y_i \inner{\w}{\widehat \x_i})^2 + t_i^+  \leq s_i & i \in [N] \\
			& \frac{1}{2} (z_i^- + \widehat y_i \inner{\w}{\widehat \x_i})^2 + t_i^-  - \kappa \lambda \leq s_i & i \in [N] \\
			& 1 - z_i^+ \leq t_i^+,\quad 1 - z_i^- \leq t_i^- & i \in [N] \\ 
			& t_i^+, t_i^- \geq 0 & i \in [N] \\
			& \|\w\|_* \leq \lambda.}
		\end{align}	
	\end{Crl}
	
	\begin{Crl} [Logistic regression]
		\label{Crl:lr}
		If $L$ represents the logloss function and $\XX=\RR^n$, then \eqref{distrob:wass} is equivalent~to
		\begin{align} \label{LogisticRegression}
		\optimize{
			\Min_{\w,\lambda,s_i} & \displaystyle \Obj & \\ [0.5em]
			\text{s.t.} & \log \Big( 1 + \exp \big( - \widehat y_i \inner{\w}{\widehat \x_i}  \big) \Big) \leq s_i & i \in [N] \\ [0.5em]
			& \log \Big( 1 + \exp \big( \widehat y_i \inner{\w}{\widehat \x_i} \big) \Big) - \kappa \lambda \leq s_i & i \in [N] \\ [0.5em]
			& \|\w\|_* \leq \lambda.}
		\end{align}	
	\end{Crl}
	
	\begin{Rmk} [Relation to classical regularization]
		\label{rem:drclassification->regularization}
		If $\XX= \RR^{n}$ and the weight parameter $\kappa$ in the transportation metric~\eqref{metric} is set to infinity, then the learning problems portrayed in Corollaries~\ref{Crl:svm}--\ref{Crl:lr} all simplify~to 
		\begin{align} \label{reg_cla}
		\Inf_{\w} ~  \frac{1}{N} \SumN L(\widehat{y}_i \inner{\w}{\widehat{\x}_i}) + \eps \| \w \|_*.
		\end{align}	
		Thus, in analogy to the case of regression, \eqref{reg_cla} reduces to an instance of the classical regularized learning problem~\eqref{in-sample-regularized}, where the dual norm term $\| \w \|_*$ plays the role of the regularization function, while the Wasserstein radius $\eps$ represents the usual regularization weight. Note that if $\kappa=\infty$, then mass transportation along the output space is infinitely expensive, that is, any distribution $\QQ\in \Wball$ can smear out the training samples along $\XX$, but it cannot flip outputs from $+1$ to $-1$ or vice versa. Thus, classical regularization schemes, which are recovered for $\kappa = \infty$, implicitly assume that output measurements are exact. As this belief is not tenable in most applications, an approach with $\kappa<\infty$ may be more satisfying. \change{We remark that alternative approaches for learning with noisy labels have previously been studied by \citet{lawrence2001estimating}, \citet{natarajan2013learning}, and \citet{yang2012multiple}.}
	\end{Rmk}

	\begin{Rmk} [Relation to Tikhonov regularization]
		The learning problem 
		\begin{align} \label{standardSVM}
		\Inf_{\w}  \frac{1}{N} \Sum_{i=1}^{N}  L(\widehat{y}_i \inner{\w}{\widehat{\x}_i}) + c \| \w \|_2^2
		\end{align}
		with Tikhonov regularizer enjoys wide popularity. If $L$ represents the hinge loss, for example, then \eqref{standardSVM} reduces to the celebrated soft margin support vector machine problem. However, the Tikhonov regularizer appearing in~\eqref{standardSVM} is not explained by a distributionally robust learning problem of the form~\eqref{distrob:wass}. It is known, however, that \eqref{reg_cla} with $\|\cdot\|_*=\|\cdot\|_2$ and \eqref{standardSVM} are equivalent in the sense that for every $\eps\geq0$ there exists $c\geq 0$ such that the solution of \eqref{reg_cla} also solves \eqref{standardSVM} and vice versa~\citep[Corollary 6]{xu2009robustness}. 
	\end{Rmk}
	
	To establish a connection between distributionally robust and classical robust classification as discussed in \citep{xu2009robustness}, we further investigate the worst-case expected loss of a fixed linear hypothesis $\w$.
	\begin{equation}
	\label{wc-expectation-classification}
	\sup_ {\QQ \in \Wball} \EE ^ \QQ [ L(y\inner{\w}{\x}) ]
	\end{equation}

	\begin{Thrm} [Extremal distributions in linear classification]
		\label{Thrm:worst-case:classification}
		The following statements hold. 
		\begin{enumerate}
			\item[(i)] If $L(z)= \max_{j\in J} \{ a_j z + b_j \}$, then the worst-case expectation~\eqref{wc-expectation-classification} coincides with
			\begin{align} \label{worst:classification:linear}
			\optimize{
				\Sup_{ \substack{\alpha^+_{ij}, \alpha^-_{ij} \\ \q^+_{ij}, \q^-_{ij}} } & \displaystyle \frac{1}{N} \sum_{i = 1}^{N}\sum_{j = 1}^{J} (\alpha^+_{ij}-\alpha^-_{ij}) a_j \widehat y_i \inner{\w}{\widehat \x_i} + a_j \widehat y_i \inner{\w}{\q^+_{ij} - \q^-_{ij}} +\sum_{j=1}^J b_j & \\		
				\mathrm{s.t.} & \displaystyle \sum_{i = 1}^{N}\sum_{j = 1}^{J} \| \q^+_{ij} \| + \| \q^-_{ij} \| + \kappa \alpha^-_{ij} \leq N \eps &  \\
				& \displaystyle \sum_{j=1}^J \alpha^+_{ij} + \alpha^-_{ij} = 1 & \hspace{-3cm}  i \in [N] \\
				& \widehat{\x}_i + \q^+_{ij}/\alpha^+_{ij} \in \XX,\quad  \widehat{\x}_i + \q^-_{ij}/\alpha^-_{ij} \in \XX & \hspace{-3cm} i \in [N], j \in [J] \\
				& \alpha^+_{ij}, \alpha^-_{ij} \geq 0 & \hspace{-3cm} i \in [N], j \in [J]}
			\end{align}
			for any fixed $\w$. Also, if $({\alpha_{ij}^{+}}\opt, {\alpha_{ij}^{-}}\opt, {\q_{ij}^{+}}\opt, {\q_{ij}^{-}}\opt)$ maximizes~\eqref{worst:classification:linear}, then the discrete distribution
			\begin{align*}
			\QQ = \frac{1}{N} \sum_{i = 1}^{N}\sum_{j = 1}^{J} {\alpha^{+}_{ij}}\opt \,\delta_{(\widehat{\x}_i - {\q^{+}_{ij}}\opt / {\alpha^{+}_{ij}}\opt,\;  \widehat y_i)} + {\alpha^{-}_{ij}}\opt \, \delta_{(\widehat{\x}_i - {\q^{-}_{ij}}\opt / {\alpha^{-}_{ij}}\opt,\; - \widehat y_i)}
			\end{align*}
			represents a maximizer for \eqref{wc-expectation-classification}.
			\item[(ii)] If $\XX = \RR^{n}$, then the worst-case expectation~\eqref{wc-expectation-classification} coincides with the optimal value of 
			\begin{align} \label{worst:classification:convex}
			\optimize{
				\Sup_{\alpha_i, \theta} & \displaystyle \lip \|(\w) \|_* \theta + \frac{1}{N} \Sum_{i=1}^N (1 - \alpha_i) L(\widehat y_i\inner{\w}{\widehat x_i}) + \alpha_i L(-\widehat y_i\inner{\w}{\widehat x_i}) \\
				\text{s.t.} & \displaystyle \theta +  \frac{\kappa}{N} \SumN \alpha_i = \eps - \gamma & \\
				& 0 \leq \alpha_i \leq 1 \qquad\qquad\qquad\quad i \in [N] \\
				& \theta \geq 0 &}
			\end{align}
			for $\gamma = 0$. Moreover, if $(\alpha_i\opt(\gamma), \theta\opt(\gamma))$ maximizes~\eqref{worst:classification:convex} for $\gamma > 0$, $\eta(\gamma)=\gamma/(\theta\opt(\gamma)+\kappa-\eps+\gamma+1)$ and $\x\opt$ solves $\max_{\x} \{\inner{\w}{\x}: \| \x \| \leq 1\}$, then the discrete distributions
			\begin{align*}
			\QQ_\gamma 
			=& \displaystyle \frac{1}{N} \sum_{i = 2}^{N} (1-\alpha_{i}\opt(\gamma)) \, \delta_{(\widehat \x_{i}, \; \widehat y_i)} + \alpha_{i}\opt(\gamma) \, \delta_{(\widehat \x_{i}, \;- \widehat y_i)} + \frac{1 - \eta(\gamma)}{N} \Big[ (1-\alpha_{1}\opt(\gamma)) \, \delta_{(\widehat \x_{1}, \; \widehat y_1)} + \alpha_{1}\opt(\gamma) \, \delta_{(\widehat \x_{1}, \;- \widehat y_1)} \Big] \\ 
			& \displaystyle + \frac{\eta(\gamma)}{N} \,\delta_{(\widehat \x_1 + \frac{\theta\opt(\gamma) N}{\eta(\gamma)} \x\opt, \;\widehat y_1) } 
			\end{align*}
			for $\gamma\in[0,\min\{\eps,1\}]$ are feasible and asymptotically optimal in \eqref{wc-expectation-classification} for $\gamma\downarrow 0$.
		\end{enumerate}	
	\end{Thrm}	
	
	Theorem~\ref{Thrm:worst-case:classification} shows how one can use convex optimization to construct a sequence of distributions that are asymptotically optimal in~\eqref{wc-expectation-classification}. Next, we show that the worst-case expected cost~\eqref{wc-expectation-classification} is equivalent to a (robust) worst-case cost over a suitably defined uncertainty set if the following assumption holds.
	\begin{Asmp}[Non-separability]
		\label{Asmp:sep:class}	
		For every $\w \in \RR^n$ there is a training sample $(\widehat \x_k, \widehat y_k)$ for some $k\leq N$ such that the derivative $L'$ exists at $\widehat y_k \inner{\w}{\widehat \x_k}$ and satisfies $|L'(\widehat y_k \inner{\w}{\widehat \x_k})| = \lip$.
	\end{Asmp}
	
	\begin{Rmk} [Non-separability]
		Assumption~\ref{Asmp:sep:class} generalizes the non-separability condition in \citep[Theorem~3]{xu2009robustness} for the classical and smooth hinge loss functions to more general Lipschitz continuous losses. Note that, in the case of the hinge loss, Assumption~\ref{Asmp:sep:class} effectively stipulates that for any $\w$ there exists a training sample $(\widehat \x_k, \widehat y_k)$ with $\widehat y_k \inner{\w}{\widehat \x_k} < 1$, implying that the dataset cannot be perfectly separated by any linear hypothesis $\w$. An equivalent requirement is that the empirical hinge loss is nonzero for every $\w$. Similarly, in the case of the smooth hinge loss, Assumption~\ref{Asmp:sep:class} ensures that for any $\w$ there is a training sample with $\widehat y_k \inner{\w}{\widehat \x_k} < 0$, which implies again that the dataset admits no perfect linear separation. Note, however, that the logloss fails to satisfy Assumption~\ref{Asmp:sep:class} as its steepest slope is attained at infinity.
	\end{Rmk}
	
	\begin{Thrm}[Robust classification] \label{Thrm:RobConnection:class}
		Suppose that $\XX= \RR^{n}$, the loss function $L$ is Lipschitz continuous and the cost of flipping a label in the transportation metric~\eqref{metric} is set to $\kappa=\infty$. Then, the worst-case expected loss~\eqref{wc-expectation-classification} provides an upper bound on the (robust) worst-case loss
		\begin{align}
		\label{wc-classification-loss}
		\optimize{
			\Sup_{\Delta \x_{i}} & \displaystyle \frac{1}{N} \sum_{i=1}^N
			L \left(\widehat y_i\inner{\w}{\widehat \x_i + \Delta \x_{i}} \right)\\
			\mathrm{s.t.} & \displaystyle \frac{1}{N}\sum_{i=1}^N \| \Delta \x_{i} \| \leq  \eps. & }
		\end{align}
		Moreover, if Assumption~\ref{Asmp:sep:class} holds, then \eqref{wc-expectation-classification} and \eqref{wc-classification-loss} are equal.	
	\end{Thrm}

	\begin{Rmk}[Tractability of robust classification]
		Assume that $\XX=\RR^n$, while $L$ and $\|\cdot\|$ both admit a tractable conic representation. By Theorem~\ref{Thrm:Classification}, the worst-case expected loss~\eqref{wc-expectation-classification} can then be computed in polynomial time by solving a tractable convex program. Theorem~\ref{Thrm:RobConnection:class} thus implies that the worst-case loss~\eqref{wc-classification-loss} can also be computed in polynomial time if Assumption~\ref{Asmp:sep:class} holds. This confirms Proposition~4 in~\citep{xu2009robustness}. No efficient method for computing~\eqref{wc-classification-loss} is know if Assumption~\ref{Asmp:sep:class} fails to hold.
	\end{Rmk}
	
	\subsection{Nonlinear Hypotheses: Reproducing Kernel Hilbert Spaces}
	\label{subsec:Kernelization}
	We now generalize the learning models from Sections~\ref{subsec:drregression} and~\ref{subsec:drclassification} to nonlinear hypotheses that range over a {\em reproducing kernel Hilbert space} (RKHS) $\HH\subseteq \RR^\XX$ with inner product $\inner{\cdot}{\cdot}_\HH$. By definition, $\HH$ thus constitutes a complete metric space with respect to the norm $\|\cdot\|_\HH$ induced by the inner product, and the point evaluation $h\mapsto h(\x)$ of the functions $h\in\HH$ represents a continuous linear functional on $\HH$ for any fixed $\x\in\XX$. The Riesz representation theorem then implies that for every $\x\in\XX$ there exists a unique function $\Phi(\x)\in\HH$ such that $h(\x)=\inner{h}{\Phi(\x)}_\HH$ for all $h\in\HH$. We henceforth refer to $\Phi:\XX\rightarrow \HH$ as the {\em feature map} and to $k:\XX \times \XX \rightarrow \RR_+$ with $k(\x,\x') = \inner{\Phi(\x)}{\Phi(\x')}_\HH$ as the {\em kernel function}. By construction, the kernel function is symmetric and positive definite, that is, the {\em kernel matrix} $\KK \in\RR^{N\times N}$ defined through $\KK_{ij} = k(\x_i,\x_j)$ is positive definite for all $N\in\mathbb N$ and $\{\x_i\}_{i\leq N}\subseteq \XX$.
	
	By the Moore-Aronszajn theorem, any symmetric and positive definite kernel function $k$ on $\XX$ induces a unique RKHS $\HH\subseteq \RR^\XX$, which can be represented as
	\[
	\HH=\left\{h\in \RR^\XX: \exists \beta_i \in \RR,~\x_i \in \XX~\forall i\in\mathbb N \text{ with }h(\x)=\sum_{i=1}^\infty \beta_i k(\x_i,\x) \text{ and } \sum_{i=1}^\infty\sum_{j=1}^\infty \beta_ik(\x_i,\x_j)\beta_j< \infty \right\},
	\]
	and where the inner product of two arbitrary functions $h_1,h_2\in\HH$ with $h_1(\x)= \sum_{i=1}^\infty \beta_i k(\x_i,\x)$ and $h_2(\x)= \sum_{j=1}^\infty \beta'_j k(\x'_j,\x)$ is defined as $\inner{h_1}{h_2}_\HH=\sum_{i=1}^\infty\sum_{j=1}^\infty \beta_ik(\x_i,\x'_j)\beta'_j$. One may now use the kernel function to define the feature map $\Phi$ through $[\Phi(\x')](\x)=k(\x',\x)$ for all $\x,\x'\in\XX$. This choice is admissible because it respects the consistency condition $\inner{\Phi(\x)}{\Phi(\x')}_\HH=k(\x,\x')$ for all $\x,\x'\in\HH$, and because it implies the desired {\em reproducing property} $\inner{f}{\Phi(\x')}_\HH= \sum_{i=1}^\infty \beta_i k(\x_i, \x')=f(\x')$ for all $f\in\HH$ and $\x'\in\XX$.
	
	In summary, given a symmetric and positive definite kernel function~$k$, there exists an associated RKHS $\HH$ and a feature map $\Phi$ with the reproducing property. As we will see below, however, to optimize over nonlinear hypotheses in $\HH$, knowledge of $k$ is sufficient, and there is no need to construct $\HH$ and $\Phi$ explicitly.

	Assume now that we are given any symmetric and positive definite kernel function~$k$, and construct a distributionally robust learning problem over all {\em nonlinear} hypotheses in the corresponding RKHS $\HH$ via
	\begin{align} \label{distrob:wass:hilbert}
	\widehat{J}(\eps) = \Inf_ {h \in \HH} \Sup_{\QQ \in \BB_{\eps} (\Pem)} \EE ^ \QQ \left[ \ell(h(\x),y) \right],
	\end{align}
	where the transportation metric is given by the Euclidean norm on $\XX\times\YY$ (for regression problems) or the separable metric~\eqref{metric} with the Euclidean norm on $\XX$ (for classification problems).
	While problem~\eqref{distrob:wass:hilbert} is hard to solve in general due to the nonlinearity of the hypotheses $h\in\HH$, it is easy to solve a lifted learning problem where the inputs $\x\in \XX$ are replaced with features $\x_\HH\in\HH$, while each {\em nonlinear} hypothesis $h\in\HH$ over the input space $\XX$ is identified with a {\em linear} hypothesis $h_\HH\in\HH$ over the feature space $\HH$ through the identity $h_\HH(\x_\HH)=\inner{h}{\x_\HH}_\HH$. Thus, the lifted learning problem can be represented as
	\begin{align} \label{distrob:wass:ker}
	\widehat{J}_{\HH}(\eps) = \Inf_ {h \in \HH} \Sup_{\QQ \in \BB_{\eps} (\Pem^{\HH})} \EE ^ \QQ \left[ \ell(\inner{h}{\bm x_\HH}_{\HH},y) \right],
	\end{align}
	where $\Pem^{\HH} = 1/N \sum_{i=1}^N \delta_{(\Phi(\widehat{\x}_i),\widehat{y}_i)}$ on $\HH\times\YY$ denotes the pushforward measure of the emprical distribution~$\Pem$ under the feature map~$\Phi$ induced by~$k$, while $\BB_{\eps} (\Pem^{\HH})$ constitutes the Wasserstein ball of radius $\eps$ around $\Pem^{\HH}$ corresponding to the transportation metric 
	\[
	d_\HH \left( (\x_\HH, y), (\x_\HH', y')\right) = \left\{ \begin{array}{ll} \sqrt{\| \x_\HH - \x_\HH' \|^2_\HH + (y - y')^2} & \text{for regression problems,}\\
	\quad \| \x_\HH - \x_\HH' \|_\HH + \kappa \mathds{1}_{ \{ y \neq y' \}  } & \text{for classification problems.} \end{array}\right.
	\]
	Even though $\Pem^{\HH}$ constitutes the pushforward measure of $\Pem$ under $\Phi$, not every distribution $\QQ^\HH\in \BB_{\eps} (\Pem^{\HH})$ can be obtained as the pushforward measure of some $\QQ\in\BB_{\eps} (\Pem)$. Thus, we should not expect \eqref{distrob:wass:hilbert} to be equivalent to \eqref{distrob:wass:ker}. Instead, one can show that under a judicious transformation of the Wasserstein radius, \eqref{distrob:wass:ker} provides an upper bound on~\eqref{distrob:wass:hilbert} whenever the kernel function satisfies a calmness condition.

	\begin{Asmp} [Calmness of the kernel]
		\label{ass:calmness}
		The kernel function $k$ is calm from above, that is, there exist a concave smooth growth function $g: \RR_+ \rightarrow \RR_+$ with $g(0)=0$ and $g'(z)\geq 1$ for all $z\in\RR_+$ such that
		\begin{align*}
		\sqrt{k(\x_1, \x_1) - 2 k(\x_1, \x_2) + k(\x_2, \x_2)} \leq g(\| \x_1 - \x_2 \|_2)\quad \forall \x_1,\x_2\in\XX.
		\end{align*}
	\end{Asmp}
	
	The calmness condition is non-restrictive. In fact, it is satisfied by most commonly used kernels.
	
	\begin{Exp} [Growth Functions for Popular Kernels]
	\label{ex:kernels}
		For most commonly used kernels $k$ on $\XX\subseteq\RR^n$, we can construct an explicit growth function $g$ that certifies the calmness of $k$ in the sense of Assumption~\ref{ass:calmness}. This construction typically relies on elementary estimates. Derivations are omitted for brevity.
		\begin{enumerate} 
			\item \textbf{Linear kernel:} For $k(\x_1, \x_2) = \inner{\x_1}{\x_2}$, we may set $g(z)=z$.
			\item \textbf{Gaussian kernel:} For $k(\x_1, \x_2) = e^{-\gamma \| \x_1 - \x_2 \|_2^2}$ with $\gamma>0$, we may set $g(z) = \max \{ \sqrt{2 \gamma}, 1 \} z$.
			\item \textbf{Laplacian kernel:} For $k(\x_1, \x_2) = e^{-\gamma \| \x_1 - \x_2 \|_1}$ with $\gamma>0$, we may set $g(z) = \sqrt{2 \gamma z \sqrt{n}}$ if $ 0 \leq z \leq \gamma \sqrt{n}/2$ and $g(z)= z + \gamma \sqrt{n}/2$ otherwise. 
			\item \textbf{Polynomial kernel:}
			The kernel $k(\x_1, \x_2) = (\gamma \inner{\x_1}{\x_2} + 1)^d$ with $\gamma>0$ and $d\in\mathbb N$ fails to satisfy the calmness condition if $\XX$ is unbounded and $d>1$, in which case $\sqrt{k(\x_1, \x_1) - 2 k(\x_1, \x_2) + k(\x_2, \x_2)}$ grows superlinearly. If $\XX\subseteq \{\x\in\RR^n:\| \x \|_2 \leq R\}$ for some $R>0$, however, the polynomial kernel is calm with respect to the growth function
			\[
			g(z) = \left\{ \begin{array}{ll} \max\{ \frac{1}{2R}\sqrt{2(\gamma R^2 + 1)^d}, 1 \} z & \text{if $d$ is even,}\\[0.5ex]
			\max\{ \frac{1}{2R}\sqrt{2(\gamma R^2 + 1)^d - 2 (1 - \gamma R^2)^d}, 1 \} z & \text{if $d$ is odd.}
			\end{array}\right.
			\]
		\end{enumerate}
	\end{Exp}
	
	\begin{Thrm} [Lifted learning problems]
		\label{Thrm:relation}
		If Assumption~\ref{ass:calmness} holds for some growth function $g$, then the following statements hold for all Wasserstein radii $\eps\geq 0$.
		\begin{itemize}
			\item[(i)] For regression problems we have $\widehat J (\eps) \leq \widehat J_\HH (\sqrt{2}g(\eps))$.
			\item[(ii)] For classification problems we have $\widehat J (\eps) \leq \widehat J_\HH (g(\eps))$.
		\end{itemize}
	\end{Thrm}
	
	We now argue that the lifted learning problem~\eqref{distrob:wass:ker} can be solved efficiently by leveraging the following representer theorem, which generalizes \citep[Theorem~4.2]{scholkopf2001learning} to non-separable loss functions. 
	
	\begin{Thrm} [Representer theorem] \label{Thrm:rep}
		Assume that we are given a symmetric positive definite kernel~$k$ on $ \XX$ with corresponding RKHS $\HH$, a set of training samples $(\widehat{\x}_i,\widehat{y}_i)\in\XX\times\YY$, $i\leq N$, and an arbitrary loss function $f: (\XX \times \YY \times \RR)^N \times \RR_+ \rightarrow \RR$ that is non-decreasing in its last argument. Then, there exist $\beta_i\in\RR$, $i\leq N$, such that the learning problem
		\begin{align} \label{rep_th}
		\min_{h\in\HH} ~ f( (\widehat \x_1,\widehat y_1, h(\widehat \x_1)), \ldots, (\widehat \x_N,\widehat y_N, h(\widehat \x_N)), \| h \|_{\HH})
		\end{align}
		is solved by a hypothesis $h^\star\in\HH$ representable as $h^\star(\x) = \sum_{i=1}^N \beta_i k(\x, \widehat \x_i)$.  
	\end{Thrm}
	
	The subsequent results involve the Kernel matrix $\KK = [\KK_{ij}]$ defined through $\KK_{ij} = k(\widehat \x_i,\widehat \x_j)$, $i,j\leq N$. The following theorems demonstrate that the lifted learning problem~\eqref{distrob:wass:ker} admits a kernel representation.
	
	\begin{Thrm} [Kernelized distributionally robust regression]
		\label{Thrm:Ker_Reg}
		Suppose that $\XX=\RR^n$, $\YY=\RR$ and $k$ is a symmetric positive definite kernel on $\XX$ with associated RKHS $\HH$. If $\ell$ is generated by a convex and Lipschitz continuous loss function~$L$, that is, $\ell(h(\x),y)=L(h(\x)-y)$, then~\eqref{distrob:wass:ker} is equivalent~to
		$$ \min_{\bm \beta \in \RR^N} \frac{1}{N} \SumN L \Big(\Sum_{j=1}^N \KK_{ij}\beta_j - \widehat y_i \Big) + \eps \lip \| (\KK^{\frac{1}{2}} \bm \beta,1) \|_{2},$$
		and for any of its minimizers $\bm \beta\opt$ the hypothesis $h\opt(\x) = \sum_{i=1}^N \beta_i\opt k(\x, \widehat \x_i)$ is optimal in~\eqref{distrob:wass:ker}.
	\end{Thrm}
	
	\begin{Thrm} [Kernelized distributionally robust classification]
		\label{Thrm:Ker_Class}
		Suppose that $\XX=\RR^n$, $\YY=\{+1,-1\}$ and $k$ is a symmetric positive definite kernel on $\XX=\RR^n$ with associated RKHS $\HH$. If $\ell$ is generated by a convex and Lipschitz continuous loss function~$L$, that is, $\ell(h(\x),y)=L(yh(\x))$, then~\eqref{distrob:wass:ker} is equivalent~to
		\begin{align}
		\label{eq:kernelized-classification-problem}
		\optimize{
			\Min_{\beta_i,\lambda,s_i} & \displaystyle \Obj & \\
			\mathrm{s.t.} & \displaystyle L \Big(\Sum_{j=1}^N \widehat y_i \KK_{ij}\beta_j\Big) \leq s_i & i \in [N] \\
			& \displaystyle L \Big(-\Sum_{j=1}^N  \widehat y_i \KK_{ij}\beta_j\Big) - \kappa \lambda \leq s_i & i \in [N] \\
			& \lip \| \KK^{\frac{1}{2}} \bm \beta \|_2 \leq \lambda ,}
		\end{align}
		and for any of its minimizers $\bm \beta\opt$ the hypothesis $h\opt(\x) = \sum_{i=1}^N \beta_i\opt k(\x, \widehat \x_i)$ is optimal in~\eqref{distrob:wass:ker}.
	\end{Thrm}
	
	Theorems~\ref{Thrm:Ker_Reg} and \ref{Thrm:Ker_Class} show that the lifted learning problem~\eqref{distrob:wass:ker} can be solved with similar computational effort as problem~\eqref{distrob:wass}, that is, optimizing over a possibly infinite-dimensional RKHS of nonlinear hypotheses is not substantially harder than optimizing over the space of linear hypotheses.

	\begin{Rmk} [Kernelization in robust regression and classification]
		Recall from Theorem~\ref{Thrm:RobConnection} that distributionally robust and classical robust linear regression are equivalent if $\Xi=\RR^{n+1}$ and the training samples are sufficiently dispersed in the sense of Assumption~\ref{Asmp:sep:reg}. Similarly, Theorem~\ref{Thrm:RobConnection:class} implies that distributionally robust and classical robust linear classification are equivalent if $\kappa=\infty$ and the training samples are non-separable in the sense of Assumption~\ref{Asmp:sep:class}. One can show that Theorems~\ref{Thrm:RobConnection} and~\ref{Thrm:RobConnection:class} naturally extend to nonlinear regression and classification models over an RKHS induced by some symmetric and positive definite kernel. Specifically, one can show that some {\em lifted} robust learning problem is equivalent to the {\em lifted} distributionally robust learning problem~\eqref{distrob:wass:ker} whenever the lifted training samples $(\Phi(\widehat \x_1),\widehat y_1), \cdots, (\Phi(\widehat \x_N),\widehat y_N)$ satisfy Assumption~\ref{Asmp:sep:reg} (for regression) or~\ref{Asmp:sep:class} (for classification). Theorems~\ref{Thrm:Ker_Reg} and~\ref{Thrm:Ker_Class} thus imply that the lifted robust regression and classification problems can be solved efficiently under mild regularity conditions whenever Assumptions~\ref{Asmp:sep:reg} and~\ref{Asmp:sep:class} hold, respectively. Unfortunately, these conditions are often violated for popular kernels. For example, the lifted samples are always linearly separable under the Gaussian kernel~\citep[p.~1496]{xu2009robustness}. In this case, the lifted robust classification problem can never be reduced to an efficiently solvable lifted distributionally robust classification problem of the form~\eqref{distrob:wass:ker}. In fact, no efficient method for solving the lifted robust classification problem seems to be known. In contrast, the lifted {\em distributionally} robust learning problems are always efficiently solvable under standard regularity conditions. 
	\end{Rmk}	

\change{
\subsection{Nonlinear Hypotheses: Neural Networks\protect\footnote{We are grateful to an anonymous referee for encouraging us to write this section.}}
\label{subsec:neural}
Families of {\em neural networks} represent particularly expressive classes of nonlinear hypotheses. In the following, we characterize a family $\HH$ of neural networks with $M\in\mathbb N$ layers through $M$ continuous activation functions $\sigma_m:\RR^{n_{m+1}} \rightarrow \RR^{n_{m+1}}$ and $M$ weight matrices $\bm W_m\in\RR^{n_{m+1}\times n_m}$, $m \in [M]$. The weight matrices can encode {\em fully connected} or {\em convolutional} layers, for example. If $n_1 = n$ and $n_{M+1} = 1$, then we may set
\begin{align*}
\HH = \left\{ h \in \RR^{\XX} ~:~ \exists \bm W_m\in\RR^{n_{m+1}\times n_m}, ~m \in [M], ~ h(\x) = \sigma_M\Big(\bm W_M \cdots \sigma_{2} \big(\bm W_2 \sigma_1(\bm W_1\x)\big) \cdots \Big)
\right\}.
\end{align*}
Each hypothesis $h\in\HH$ constitutes a neural network and is uniquely determined by the collection of all weight matrices $\bm W_{[M]} := (\bm W_1, \dots, \bm W_M)$. In order to emphasize the dependence on $\bm W_{[M]}$, we will sometimes use $h(\bm x;\bm W_{[M]})$ to denote the hypotheses in~$\HH$. Setting $\bm x_1=\bm x$, the features of the neural network are defined recursively through $\bm x_{m+1}=\sigma_m(\bm z_m)$, where $\bm z_m=\bm W_m \bm x_m$, $m\in [M]$. The features~$\bm x_m$, $m=2,\ldots,M$, correspond to the {\em hidden} layers of the neural network, while $\bm x_{M+1}$ determines its output.

\begin{Exp}[Activation functions]\label{exp:layers}
	The following activation functions are most widely used.
	\begin{enumerate}
		\item \textbf{Hyperbolic tangent}: $[\sigma_m(\bm z_m)]_i = {(\exp([\bm z_m]_i) - \exp(- [\bm z_m]_i))}/{(\exp([\bm z_m]_i) + \exp(- [\bm z_m]_i))}$
		\item \textbf{Sigmoid:} $[\sigma_m(\bm z_m)]_i = {1}/{(1 + \exp(- [\bm z_m]_i))}$
		\item \textbf{Softmax:} $[\sigma_m(\bm z_m)]_i = {\exp([\bm z_m]_i)}/{\sum_{j=1}^{n_{m+1}} \exp([\bm z_m]_j)}$
		\item \textbf{Rectified linear unit (ReLU):} $[\sigma_m(\bm z_m)]_i = \max\{ 0 , [\bm z_m]_i \}$
		\item \textbf{Exponential linear unit (ELU):} $[\sigma_m(\bm z_m)]_i = \max\{ 0 , [\bm z_m]_i \} + \min \{ 0, \alpha (\exp([\bm z_m]_i) - 1) \}$                        
	\end{enumerate}
\end{Exp}

The distributionally robust learning model over the hypothesis class $\HH$ can now be represented as
\begin{align}\label{eq:deepdro}
\inf_{h \in \HH} \, \sup_{\QQ \in \Wball} ~ \EE^{\QQ} \left[ \ell\big(h(\x), y \big) \right]  =
\inf_{\bm W_{[M]}} \, \sup_{\QQ \in \Wball} ~ \EE^{\QQ} \left[ \ell\big(h(\x; \bm W_{[M]}), y \big) \right],
\end{align}
where we use the transportation metrics~\eqref{eq:metric:additive} and~\eqref{metric} for regression and classification problems, respectively. Moreover, we adopt the standard convention that $\ell(h(\x),y) = L(h(\x)- y)$ for regression problems and $\ell(h(\x,y)) = L( y h(\x) )$ for classification problems, where $L$ is a convex and Lipschitz continuous univariate loss function. In the following we equip each feature space $\RR^{n_m}$ with a norm $\|\cdot\|$, $m\in[M+1]$. By slight abuse of notation, we use the same symbol for all norms even though the norms on different feature spaces may differ. Using the norm on $\RR^{n_{m+1}}$, we define the Lipschitz modulus of $\sigma_m$ as
\[
\mathop{\rm lip}(\sigma_m):=\sup_{\bm z,\bm z'\in\RR^{n_{m+1}}} \left\{ \frac{\left\|\sigma(\bm z) - \sigma(\bm z')\right\|}{\| \bm z - \bm z' \|} : \bm z\neq \bm z'\right\}.
\]
We are now ready to state the main result of this section, which provides a conservative upper bound on the distributionally robust learning model~\eqref{eq:deepdro}.

\begin{Thrm}[Distributionally robust learning with neural networks] \label{theorem:lips}
	The distributionally robust learning model~\eqref{eq:deepdro} is bounded above by the regularized empirical loss minimization problem
	\begin{align} \label{eq:deepreg}
	\inf_{\bm W_{[M]}} ~ \frac{1}{N} \sum_{i=1}^N \ell(h(\widehat \x_i; \bm W_{[M]}), \widehat y_i) + \rho \lip \max \left\{  \prod_{m=1}^M \mathop{\rm lip}(\sigma_m) \|\bm W_m\|, \frac{c}{\kappa} \right\},
	\end{align}
	where $c = 1$ for regression problems and $c= \max \{1 , 2 \sup_{{h \in \mathbb{H}, \bm x \in \mathbb X}} |h(\bm x)| \}$ for classification problems. Moreover, $\| \bm W_m \| = \sup_{\| \bm x_m \| = 1 } \| \bm W_m \bm x_m \|$ is the operator norm induced by the norms on $\RR^{n_m}$ and $\RR^{n_{m+1}}$.
\end{Thrm}

\begin{Rmk} [Uniform upper bound on all neural networks]
	For classification problems the constant~$c$ in~\eqref{eq:deepreg} represents a uniform upper bound on all neural networks and may be difficult to evaluate in general. It is easy to estimate~$c$, however, if the last activation function is itself bounded such as the softmax function, which yields a probability distribution over the output space. In this case one may simply set $c=2$.
\end{Rmk}

The product term $\prod_{m=1}^M \mathop{\rm lip}(\sigma_m) \|\bm W_m\|$ in~\eqref{eq:deepreg} represents an upper bound on the Lipschitz modulus of $h(\x; \bm W_{[M]})$. We emphasize that computing the exact Lipschitz modulus of a neural network is NP-hard even if there are only two layers and all activation functions are of the ReLU type \citep[Theorem~2]{virmaux2018lipschitz}. In contrast, the upper bound at hand is easy to compute as all activation functions listed in Example~\ref{exp:layers} have Lipschitz modulus~1 with respect to the Euclidean norms on the domain and range spaces \citep{gouk2018regularisation,wiatowski2016discrete}. 
For more details on how to estimate the Lipschitz moduli of neural networks we refer to \citep{gouk2018regularisation,miyato2018spectral,neyshabur2018a,szegedy2013intriguing}.


Note that even though~\eqref{eq:deepreg} represents a finite-dimensional optimization problem over the weight matrices of the neural network, both the empirical prediction loss as well as the regularization term are non-convex in $\bm W_{[M]}$, which complicates numerical solution. If $\kappa=\infty$, however, one can derive an alternative upper bound on the distributionally robust learning model~\eqref{eq:deepdro} with a {\em convex} regularization~term.

\begin{Crl}[Convex regularization term] \label{corollary:safe}
	If $\kappa = \infty$, then there is $\overline \rho\ge 0$ such that the distributionally robust learning model~\eqref{eq:deepdro} is bounded above by the regularized empirical loss minimization problem
	\begin{align} \label{eq:deepreg2}
	\inf_{\bm W_{[M]}} ~ \frac{1}{N} \sum_{i=1}^N \ell(h(\widehat \x_i; \bm W_{[M]}), \widehat y_i) + \overline \rho \sum_{m=1}^M \| \bm W_m \|.
	\end{align}
\end{Crl}

As the empirical prediction loss remains non-convex, it is expedient to address problem~\eqref{eq:deepreg2} with local optimization methods such as stochastic gradient descent algorithms. For a comprehensive review of first- and the second-order stochastic gradient algorithms we refer to~\citep{agarwal2017second} and the references therein. In the numerical experiments we will use a stochastic proximal gradient descent algorithm that exploits the convexity of the regularization term and generates iterates $\bm W^k_{[M]}$ for $k\in\mathbb N$ according to the update rule
\begin{equation*}
	\bm W_m^{k+1}= \mathrm{prox}_{\eta_k\overline\rho \|\bm W_m\|} \left( \bm W^k_m- \eta_k\nabla_{\bm W_m} \ell(h(\widehat \x_{i_k}; \bm W^k_{[M]}), \widehat y_{i_k}) \right)\quad \forall m\in[M],
\end{equation*}
where $\eta_k>0$ is a given step size and $i_k$ is drawn randomly from the index set $[N]$, see, {\em e.g.},~\cite{nitanda2014stochastic}. Here, the proximal operator associated with a convex function $\varphi: \RR^{n_{m+1}\times n_m} \rightarrow \RR$ is defined through
\begin{align*}
\mathrm{prox}_\varphi(\bm W_m) := \arg\min_{\bm W'_m} ~ \varphi(\bm W'_m) + \frac{1}{2} \| \bm W'_m - \bm W_m \|_F^2,
\end{align*}
where $\| \cdot \|_F$ stands for the Frobenius norm. The algorithm is stopped as soon as the improvement of the objective value falls below a prescribed threshold. As the empirical prediction loss is non-convex and potentially non-smooth, the algorithm fails to offer any strong performance guarantees. For the scalability of the algorithm, however, it is essential that the proximal operator can be evaluated efficiently. 

\begin{Exp}[Proximal operator]
	Suppose that all feature spaces $\RR^{n_m}$ are equipped with the $p$-norm for some $p \in \{1, 2, \infty \}$, which implies that all parameter spaces $\RR^{n_{m+1}\times n_m}$ are equipped with the corresponding matrix $p$-norm. In this case the proximal operator of $\varphi(\bm W_m)=\eta\|\bm W_m\|_p$ for some fixed $\eta>0$ can be evaluated highly efficiently. 
	\begin{enumerate}
		\item \textbf{MACS ($p=1$):} The matrix $1$-norm returns the maximum absolute column sum (MACS). Evaluating the proximal operator of~$\varphi(\bm W_m)=\eta\|\bm W_m\|_1$ amounts to solving the minimization problem
		\begin{align*}
		\mathrm{prox}_{\varphi}(\bm W_m) = \optimize{
			\displaystyle \min_{\bm W_m', u} & \displaystyle \eta u + \sum_{i=1}^{n_m} \| [\bm W'_m]_{:,i} - [\bm W_m]_{:,i} \|_2^2 \\
			\mathrm{s.t.} & \| [\bm W'_m]_{:,i} \|_1 \leq u\quad i\in [n_m],}
		\end{align*}
		where $[\bm W_m]_{:,i}$ and $[\bm W'_m]_{:,i}$ represent the $i$-th columns of $\bm W_m$ and $\bm W'_m$, respectively. For any fixed $u$, the above problem decomposes into $n_m$ projections of the vectors $[\bm W_m]_{:,i}$, $i\in[n_m]$, to the $\ell_1$-ball of radius $u$ centered at the origin. Each of these projections can be computed via an efficient sorting algorithm proposed in~\citep{duchi2008efficient}. Next, we can use any line search method such as the golden-section search algorithm to optimize over $u$, thereby solving the full proximal problem.
		
		\item \textbf{Spectral  ($p=2$):} The matrix $2$-norm coincides with the spectral norm, which returns the maximum singular value. In this case, the proximal problem for $\varphi(\bm W_m)=\eta\|\bm W_m\|_2$ can be solved analytically via singular value thresholding \citep[Theorem 2.1]{cai2010singular}, that is, given the singular value decomposition $\bm W_m = \bm U \bm S \bm V^\top$ with $\bm U\in\RR^{n_{m+1}\times n_{m+1}}$ orthogonal, $\bm S\in\RR_+^{n_{m+1}\times n_{m}}$ diagonal and $\bm V\in\RR^{n_{m}\times n_{m}}$ orthogonal, the proximal operator satisfies
		\begin{align*}
		\mathrm{prox}_{\varphi}(\bm W_m) = \mathrm{prox}_{\varphi}(\bm U \bm S \bm V^\top) = \bm U \widetilde{\bm S} \bm V^\top, \qquad     \mbox{where} \quad \widetilde{\bm S}_{ij} = \max \{ \bm S_{i j} - \eta, 0 \} \,.
		\end{align*}
		The singular value decomposition can be accelerated using a randomized algorithm proposed in~\citep{halko2011finding}.
		
		\item \textbf{MARS ($p=\infty$):} The matrix $\infty$-norm returns the maximum absolute row sum (MARS) and thus satisfies $\| \bm W_m \|_{\infty} = \| \bm W_m^\top \|_1$. Therefore, one can use the iterative scheme developed for MACS to compute the proximal operator of $\varphi(\bm W_m)=\eta\| \bm W_m \|_\infty$ by simply transposing the weight matrix~$\bm W_m$.
	\end{enumerate}
\end{Exp}

The convergence behavior of the stochastic proximal gradient descent algorithm can be further improved by including a momentum term inside the proximal operator, see, {\em e.g.},~\cite{loizou2017momentum}.
}

	
	\section{Generalization Bounds}
	\label{sec:Probabilistic}
	
	Generalization bounds constitute upper confidence bounds on the out-of-sample error. Traditionally, generalization bounds are derived by controlling the complexity of the hypothesis space, which is typically quantified in terms of its VC-dimension or via covering numbers or Rademacher averages~\citep{shalev2014understanding}. \change{Strengthened generalization bounds for large margin classifiers can be obtained by improving the estimates of the VC-dimension and the Rademacher average \citep{shivaswamy2007ellipsoidal, shivaswamy2010maximum}.} We will now demonstrate that distributionally robust learning models of the type~\eqref{distrob:wass} or~\eqref{distrob:wass:ker} enjoy simple new generalization bounds that can be obtained under minimal assumptions. In particular, they do not rely on any notions of hypothesis complexity and may therefore even extend to hypothesis spaces with infinite VC-dimensions. Our approach is reminiscent of the generalization theory for robust support vector machines portrayed in~\citep{xu2009robustness}, which also replaces measures of hypothesis complexity with robustness properties. However, we derive explicit finite sample guarantees, while~\citep{xu2009robustness} establishes asymptotic consistency results. Moreover, we relax some technical conditions used in~\citep{xu2009robustness} such as the compactness of the input space~$\XX$.
	
	The key enabling mechanism of our analysis is a measure concentration property of the Wasserstein metric, which holds whenever the unknown data-generating distribution has exponentially decaying tails.  
	
	\begin{Asmp2} [Light-tailed distribution]
		\label{Asmp:light}
		There exist constants $a>1$ and $A>0$ and a reference point $\xxi'\in\mathbb R^{n+1}$ such that $\EE^\PP[\exp \left(d(\xxi,\xxi'))^a \right)] \le A$, where $d$ denotes the usual mass transportation cost.
	\end{Asmp2}

	\begin{Thrm2}[Measure concentration {\citep[Theorem 2]{fournier2015rate}}]
		\label{thm:concentration}
		If Assumption~\ref{Asmp:light} holds, then we have
		\begin{align}
		\label{concentration}
		\PP^N \Big\{ W(\PP, \Pem) \ge \eps \Big \} \le \left\{ \begin{array}{ll} c_1 \exp\big({-c_2N\eps^{\max\{n+1,2\}}}\big) & \text{if } \eps \le 1, \\ c_1 \exp\big({-c_2N\eps^a}\big) & \text{if } \eps > 1,\end{array}\right.
		\end{align}
		for all $N \ge 1$, $n\neq 1$, and $\eps>0$, where the constants $c_1, c_2>0$ depend only on $a$, $A$, $d$ and $n$.\footnote{{A similar but slightly more complicated inequality also holds for the special case $n = 1$; see \citep[Theorem 2]{fournier2015rate} for details.}}
	\end{Thrm2}
	
	Theorem~\ref{concentration} asserts that the empirical distribution $\Pem$ converges exponentially fast to the unknown data-generating distribution $\PP$, in probability with respect to the Wasserstein metric, as the sample size $N$ tends to infinity. We can now derive simple generalization bounds by increasing the Wasserstein radius $\eps$ until the violation probability on the of the right hand side of~\eqref{concentration} drops below a prescribed significance level $\eta\in (0,1]$. Specifically, Theorem~\ref{concentration} implies that $\Prob^N \{ \mathds{P} \in \mathds{B}_\eps(\widehat{\mathds{P}}_N) \} \geq 1-\eta$ for any $\eps\ge \eps_N(\eta)$, where 
	\begin{align}
	\label{eps:opt}
	\eps_N(\eta) = \begin{cases}
	\displaystyle \left( \frac{\log(c_1/ \eta)}{c_2 N} \right)^{\frac{1}{\max\{n+1,2\}}} & \mathrm{if} ~ N \geq  \frac{\log(c_1/ \eta)}{c_2}, \\
	\displaystyle \left( \frac{\log(c_1/ \eta)}{c_2 N} \right)^{\frac{1}{a}} & \mathrm{if} ~ N <  \frac{\log(c_1 /\eta)}{c_2}.
	\end{cases}	
	\end{align}
	
	\begin{Thrm2} [Basic generalization bound {\citep[Theorem~3.5]{MohKun-14}}] 
		\label{Thrm:uniform}
		If Assumption~\ref{Asmp:light} holds, then  
		\begin{align} 
		\label{finite:sample:guarantee}
		\Prob^N \left\{ \mathds{E}^\mathds{P} \left[ \ell(\inner{\w}{\x}, y) \right] \leq \sup_{\QQ \in \Wball} 
		\mathds{E}^\mathds{\QQ} \left[ \ell(\inner{\w}{\x}, y) \right] ~\forall \w \in \RR^n\right\} \geq 1 - \eta
		\end{align}
		for any $N\ge 1$, $n\neq 1$, $\eta\in(0,1]$ and $\eps\ge \eps_N(\eta)$.  
	\end{Thrm2}
	
	\begin{Rmk2} [Discussion of basic generalization bound] \label{Rmk:perf_guarantee}
		The following comments are in order.
		
		\begin{enumerate} [label=\Roman*.] \setlength{\itemsep}{0mm}
			
			\item {\textbf{Performance guarantees for optimal hypotheses:}} 
			If $\widehat J(\eps)$ denotes the minimum and $\widehat \w$ a minimizer of the distributionally robust learning problem~\eqref{distrob:wass}, then Theorem~\ref{Thrm:uniform} implies that
			$$ \Prob^N \left\{ \EE^\PP \left[ \ell(\inner{\widehat \w}{\x}, y) \right] \leq \widehat J(\eps) \right\} \ge 1 - \eta$$
			for any $N\ge 1$, $n> 1$, $\eta\in(0,1]$ and $\eps\ge \eps_N(\eta)$. 
			
			\item {\textbf{Light-tail assumption:}} Assumption~\ref{Asmp:light} is restrictive but unavoidable for any measure concentration result of the type described in Theorem~\ref{thm:concentration}. It is automatically satisfied if the input-output pair has bounded support or is known to follow a Gaussian or exponential distribution, for instance. 
			
			\item {\textbf{Asymptotic consistency:}} It is clear from~\eqref{eps:opt} that for any fixed $\eta \in (0,1]$, the radius $\eps_N(\eta)$ tends to 0 as $N$ increases. Moreover, Theorem~3.6 in~\citep{MohKun-14} implies that if $\eta_N$ converges to 0 at a carefully chosen rate ({\em e.g.}, $\eta_N=\exp(-\sqrt{N})$), then the solution of the distributionally robust learning problem~\eqref{distrob:wass} with Wasserstein radius $\eps=\eps_N(\eta_N)$ converges almost surely to the solution of the ideal learning problem that minimizes the out-of-sample error under the unknown true distribution $\PP$.	
			
			\item {\textbf{Curse of dimensionality:}} The Wasserstein radius~\eqref{eps:opt} has two decay regimes. For small~$N$, $\eps_N(\eta)$ decays as $N^{-\frac{1}{a}}$, and for large $N$ it is proportional to $N^{-\frac{1}{n+1}}$. We thus face a curse of dimensionality for large sample sizes. In order to half the Wasserstein radius, one has to increase $N$ by a factor of~$2^n$. This curse of dimensionality is fundamental, {\em i.e.}, the dependence of the measure concentration result in Theorem~\ref{thm:concentration} on the input dimension $n$ cannot be improved for generic distributions $\PP$; see \citep{weed2017} or \citep[Section~1.3]{fournier2015rate}. Improvements are only possible in special cases, {\em e.g.}, if $\PP$ is finitely supported.
			
			\item{\textbf{Extension to nonlinear hypotheses:}}
			Theorem~\ref{Thrm:uniform} directly extends to any distributionally robust learning problem over an RKHS $\HH$ induced by some symmetric and positive definite kernel function~$k$. Specifically, if~$k$ is calm in the sense of Assumption~\ref{ass:calmness} with growth function~$g$, then we have
			\begin{align} 
			\label{eq:RKHS-generalization}
			\Prob^N \left\{ \mathds{E}^\mathds{P}[\ell(h(\x),y)]  \leq \Sup_{\QQ \in \BB_{\eps} (\Pem^{\HH})} \EE ^ \QQ \left[ \ell(\inner{h}{\bm x_\HH}_{\HH},y) \right] ~\forall h
			\in\HH
			\right\} \geq 1 - \eta
			\end{align}
			for any $N\ge 1$, $n\neq 1$, $\eta\in(0,1]$ and $\eps\ge cg(\eps_N(\eta))$, where $c=\sqrt{2}$ for regression problems and $c=1$ for classification problems. To see this, note that the inclusion $\PP\in  \BB_{\eps}(\Pem)$ implies 
			\begin{equation}
			\label{eq:RKHS-generalization2}
			\mathds{E}^\mathds{P}[\ell(h(\x),y)]  \leq \Sup_{\QQ \in \BB_{\eps} (\Pem)} \EE ^ \QQ \left[ \ell(h(\x),y) \right]
			\leq \Sup_{\QQ \in \BB_{\eps} (\Pem^{\HH})} \EE ^ \QQ \left[ \ell(\inner{h}{\bm x_\HH}_{\HH},y) \right] \quad \forall h
			\in\HH,
			\end{equation}
			where the second inequality follows from the proof of Theorem~\ref{Thrm:relation}. The generalization bound~\eqref{eq:RKHS-generalization} thus holds because $\Prob^N \{ \mathds{P} \in \mathds{B}_\eps(\widehat{\mathds{P}}_N) \} \geq 1-\eta$ for any $\eps\ge \eps_N(\eta)$. Note that the rightmost term in~\eqref{eq:RKHS-generalization2} can be computed for any finitely generated hypothesis $h\in\HH$ representable as $h(\x) = \sum_{i=1}^N \beta_i k(\x, \widehat \x_i)$, which follows from Theorems~\ref{Thrm:Ker_Reg} and~\ref{Thrm:Ker_Class}, while the middle term is hard to compute. We emphasize that the generalization bound~\eqref{eq:RKHS-generalization} does not rely on any notion of hypothesis complexity and remains valid even if $\HH$ has infinite VC-dimension ({\em e.g.}, if $\HH$ is generated by the Gaussian kernel).
		\end{enumerate}
	\end{Rmk2}

	Theorem~\ref{thm:concentration} provides a confidence set for the unknown probability distribution $\PP$, and Theorem~\ref{Thrm:uniform} uses this confidence set to construct a uniform generalization bound on the prediction error under $\PP$. The radius of the confidence set for $\PP$ decreases slowly due to a curse of dimensionality, but the decay rate is essentially optimal. This does not imply that the decay rate of the generalization bound~\eqref{finite:sample:guarantee} is optimal, too. In fact, the worst-case expected error over a Wasserstein ball of radius $\eps$ can be a $(1-\eta)$-confidence bound on the expected error under $\PP$ even if the Wasserstein ball fails to contain $\PP$ with confidence $1-\eta$. Thus, the measure concentration result of Theorem~\ref{thm:concentration} is too powerful for our purposes and leads to an over-conservative generalization bound. Below we will show that the curse of dimensionality in the generalization bound \eqref{finite:sample:guarantee} can be broken if we impose the following restriction on the hypothesis space. 
	
	\begin{Asmp2}[Hypothesis space] \label{Asmp:improved}
		The space of admissible hypotheses in~\eqref{distrob:wass} is restricted to $\WW\subseteq \mathbb R^n$. There exists $\underline \Omega>0$ with $\inf_{\w\in\WW} \| (\w, -1) \|_* \geq \underline \Omega$ if \eqref{distrob:wass} is a regression problem and $\inf_{\w\in\WW} \| \w \|_* \geq \underline \Omega$ if \eqref{distrob:wass} is a classification problem. Similarly, there exists $\overline\Omega\ge 0$ with $\sup_{\w, \w' \in \WW} \| \w - \w' \|_\infty \leq \overline \Omega$.
	\end{Asmp2}
	
	\begin{Thrm2} [Improved generalization bound] \label{Thrm:improved}
		Suppose that Assumptions~\ref{Asmp:light} and \ref{Asmp:improved} hold, and the function $L$ is Lipschitz continuous. Moreover, assume that $\Xi=\mathbb R^{n+1}$ and $ M_n = \max_{i\leq n} \| \bm e^{n+1}_i \|_* $ if~\eqref{distrob:wass} is a regression problem, while $\Xi =\mathbb R^n\times\{-1,1\}$ and $ M_n = \max_{i\leq n} \| \bm e^{n}_i \|_* $ if~\eqref{distrob:wass} is a classification problem, where $\bm e^{n}_i$ is the $i$-th standard basis vector in $\RR^{n}$. Then, there exist constants $c_3\ge1$, $c_4>0$ depending only on the light tail constants $a$ and $A$ such that the generalization bound~\eqref{finite:sample:guarantee} holds for any $N \ge \max \left\{ (16n/c_4)^2, 16 \log (c_3/\eta)/c_4 \right\}$, $\eta\in(0,1]$ and $\eps\ge \eps'_N(\eta)$, where
		\begin{align*}
		\eps'_N(\eta) = {2 \overline \Omega \over \sqrt{N} \underline \Omega} \bigg[ M_nnA + {\sqrt{n\log (\sqrt{N}) + \log (c_3/\eta) \over c_4}}\, \bigg].	
		\end{align*}
	\end{Thrm2}
	
	The improved generalization bound from Theorem~\ref{Asmp:improved} does not suffer from a curse of dimensionality. In fact, in order to half the Wasserstein radius $\eps'_N(\eta)$, it suffices to increase the sample size $N$ by a factor of~$4$, irrespective of the input dimension~$n$.

	\begin{Rmk2}[Discussion of improved generalization bound] \label{Rmk:improved}
		The following comments are in order.
		\begin{enumerate}[label=\Roman*.]
			\setlength{\itemsep}{0mm}			
			\item {\textbf{Bounds on hypothesis space:}} Assumption~\ref{Asmp:improved} imposes upper and lower bounds on~$\WW$. The upper bound enables us to control the difference between the empirical and the true expected loss uniformly across all admissible hypotheses. This bound is less restrictive than the uniform bound on the loss function used to derive Rademacher generalization bounds (see, {\em e.g.},~\citep[Theorem~26.4]{shalev2014understanding}), which essentially imposes upper bounds both on the hypotheses and the input-output pairs. The lower bound in Assumption~\ref{Asmp:improved} is restrictive for classification problems but trivially holds for regression problems because $\|(\w,-1)\|_*$ is uniformly bounded away from zero for any (dual) norm on $\RR^{n+1}$.
			
			\item {\textbf{Breaking the curse of dimensionality:}} By leveraging Assumption~\ref{Asmp:improved}, Theorem~\ref{Thrm:improved} reduces the critical Wasserstein radius in the generalization bound~\eqref{finite:sample:guarantee} from $\eps_N(\eta)\propto \mathcal{O}([\log(\eta^{-1}) / N]^{1/(n+1)})$, which suffers from a curse of dimensionality, to $\eps'_N(\eta)\propto \mathcal{O}([ \log(\eta^{-1}) + n\log(N) ) / N]^{1/2})$, which essentially follows a square root law reminiscent of the central limit theorem.
		\end{enumerate}		
	\end{Rmk2}
	
	
	\section{Error and Risk Estimation}
	\label{subsec:Uncertainty}
	
	Once a hypothesis $h(\x)$ has been chosen, it is instructive to derive pessimistic {\em and} optimistic estimates of its out-of-sample prediction error (in the case of regression) or its out-of-sample risk (in the case of classification). We will argue below that the distributionally robust optimization techniques developed in this paper also offer new perspectives on error and risk estimation. For ease of exposition, we ignore any support constraints, that is, we set $\XX=\RR^{n}$ and $\YY=\RR$ (for regression) or $\XX=\RR^n$ and $\YY=\{+1,-1\}$ (for classification). Moreover, we focus on linear hypotheses of the form $h(\x)=\inner{\w}{\x}$. Note, however, that all results extend directly to conic representable support sets  and to nonlinear hypotheses.
	
	In the context of regression, we aim to estimate the prediction error $\error(\w) = \EE^\PP \left[ |y - \inner{\w}{\x}| \right]$ or, more precisely, the mean absolute prediction error under the unknown data-generating distribution $\PP$. As usual, we assume that the transportation metric $d$ is induced by a norm $\|\cdot\|$ on the input-output space $\RR^{n+1}$.

	\begin{Thrm2} [Error bounds in linear regression]
		\label{Thrm:Error}
		The prediction error admits the following estimates.
		\begin{subequations} \label{err}
			\begin{enumerate} [label=(\roman*)]
				\item The worst-case error $\error_{\max}(\w) = \sup_ {\QQ \in \Wball} \EE^\QQ  \left[ | y - \inner{\w}{\x} | \right]$ is given by
				\begin{align} \label{worst_case:err}
				\error_{\max}(\w) = \frac{1}{N} \SumN |\widehat y_i - \inner{\w}{\widehat \x_i} | + \eps \| (\w, -1) \|_*.
				\end{align}
				\item The best-case error $\error_{\min}(\w) = \inf_ {\QQ \in \Wball} \EE^\QQ  \left[ |y - \inner{\w}{\x} | \right]$ is given by
				\begin{align} \label{best_case:err}
				\error_{\min}(\w) = \max \left\{ \frac{1}{N} \SumN |\widehat y_i - \inner{\w}{\widehat \x_i} | - \eps \| (\w, -1) \|_* , 0 \right\}
				\end{align}
			\end{enumerate}		
		\end{subequations}		
	\end{Thrm2}
	
	In the context of classification, we aim to quantify the risk $\risk(\w) = \PP \left[ y \neq \text{sign} (\inner{\w}{\x})\right]$, that is, the misclassification probability under the unknown true distribution~$\PP$. Note that the risk can equivalently be defined as the expectation of a characteristic function, that is, $\risk(\w) = \EE^\PP  [\mathds{1}_{\lbrace y \neq \inner{\w}{\x} \rbrace}]$. As usual, we assume that the transportation metric $d$ is of the form~\eqref{metric}, where $\kappa\geq 0$ is the cost of flipping a label.

	\begin{Thrm2} [Risk bounds in linear classification] \label{Thrm:Risk}
		The risk admits the following estimates.
		\begin{subequations} \label{risk}
			\begin{enumerate} [label=(\roman*)]
				\item The worst-case risk $\risk_{\max}(\w) = \sup_ {\QQ \in \Wball} \EE^\QQ  [\mathds{1}_{\lbrace y \neq \inner{\w}{\x} \rbrace}]$ is given by
				\begin{align} \label{worst_case:risk}
				\risk_{\max}(\w) = \optimize{
					\Min_{ \substack{\lambda,s_i \\ r_i, t_i} } & \displaystyle \Obj \\
					\mathrm{s.t.} & 1 - r_i \widehat{y_i} \inner{\w}{\widehat{\x}_i} \leq s_i & i \in [N]\\ [0.2em]
					& 1 + t_i \widehat{y_i} \inner{\w}{\widehat{\x}_i} - \lambda \kappa \leq s_i & i \in [N]\\ [0.2em]
					& r_i \| \w \|_* \leq \lambda , ~ t_i \| \w \|_* \leq \lambda & i \in [N] \\ [0.2em]
					& r_i, t_i, s_i \geq 0 &i \in [N].}
				\end{align}
				\item The best-case risk $\risk_{\min}(\w) = \inf_ {\QQ \in \Wball} \EE^\QQ  [\mathds{1}_{\lbrace y \neq \inner{\w}{\x} \rbrace}]$ is given by
				\begin{align} \label{best_case:risk}
				\risk_{\min}(\w) = 1 - \optimize{
					\Min_{ \substack{\lambda,s_i \\ r_i, t_i} } & \displaystyle \Obj \\
					\mathrm{s.t.} & 1 + r_i \widehat{y_i} \inner{\w}{\widehat{\x}_i} \leq s_i & i \in [N]\\ [0.2em]
					& 1 - t_i \widehat{y_i} \inner{\w}{\widehat{\x}_i} - \lambda \kappa \leq s_i & i \in [N]\\ [0.2em]
					& r_i \| \w \|_* \leq \lambda , ~ t_i \| \w \|_* \leq \lambda & i \in [N] \\ [0.2em]
					& r_i, t_i, s_i \geq 0 &i \in [N].}
				\end{align}
			\end{enumerate}		
		\end{subequations}		
	\end{Thrm2}
	
	We emphasize that, as the hypothesis $\w$ is fixed, the error and risk estimation problems~\eqref{err} and \eqref{risk} constitute tractable linear programs that can be solved highly efficiently.
	
	\begin{Rmk2} [Confidence intervals for error and risk]
		\label{Rmk:conf}
		If the Wasserstein radius is set to $\eps_N(\eta/2)$ defined in~\eqref{eps:opt}, where $\eta \in (0,1]$ is a prescribed significance level, then Theorem~\ref{Thrm:uniform} implies that $\error (\widehat \w)\in [\error_{\min} (\widehat \w), \error_{\max} (\widehat \w)]$ and $\risk(\widehat \w) \in [\risk_{\min}(\widehat \w), \risk_{\max}(\widehat \w)]$ with confidence $1-\eta$ for any $\widehat \w\in\RR^n$ that may even depend on the training data. Theorem~\ref{Thrm:Error} implies that the confidence interval for the true error $\error(\widehat \w)$ can be calculated analytically from~\eqref{err}, while Theorem~\ref{Thrm:Risk} implies that the confidence interval for the true risk $\risk(\widehat \w)$ can be computed efficiently by solving the tractable linear programs~\eqref{risk}.
	\end{Rmk2}
	
	\begin{Rmk2} [Extension to nonlinear hypotheses]
		By using the tools of Section~\ref{subsec:Kernelization}, Theorems~\ref{Thrm:Error} and~\ref{Thrm:Risk} generalize immediately to nonlinear hypotheses that range over a RKHS. Specifically, we can formulate lifted error and risk estimation problems where the inputs $\x\in \XX$ are replaced with features $\x_\HH\in\HH$, while each {\em nonlinear} hypothesis $h\in\HH$ over the input space $\XX$ is identified with a {\em linear} hypothesis $h_\HH\in\HH$ over the feature space $\HH$ through the identity $h_\HH(\x_\HH)=\inner{h}{\x_\HH}_\HH$. Tractability is again facilitated by Theorem~\ref{Thrm:rep}, which allows us to focus on finitely parameterized hypotheses of the form $h(\x) = \sum_{i=1}^N \beta_i k(\x, \widehat \x_i)$.
	\end{Rmk2}
	
	
	\section{Numerical Results}
	\label{sec:Simulation}
	We showcase the power of regularization via mass transportation in various applications based on standard datasets from the literature. All optimization problems are implemented in Python and solved with Gurobi~7.5.1 All experiments are run on an Intel XEON CPU (3.40GHz), and the corresponding codes are made publicly available at \href{https://github.com/sorooshafiee/Regularization-via-Transportation}{https://github.com/sorooshafiee/Regularization-via-Transportation}.

	\subsection{Regularization with Pre-selected Parameters}
	\label{sec:pre-selected-parameters}
	We first assess how the out-of-sample performance of a distributionally robust support vector machine (DRSVM) is impacted by the choice of the Wasserstein radius~$\eps$, the cost~$\kappa$ of flipping a label, and the kernel function~$k$. To this end, we solve three binary classification problems from the MNIST database~\citep{MNIST} targeted at distinguishing pairs of similar handwritten digits (1-vs-7, 3-vs-8, 4-vs-9). In the first experiment we optimize over linear hypotheses and use the separable transporation metric~\eqref{metric} involving the $\infty$-norm on the input space. \change{All results are averaged over 100 independent trials. In each trial, we randomly select 500 images to train the DRSVM model~\eqref{HLP} and use the remaining $12{,}000$ images for testing.} The correct classification rate (CCR) on the test data, averaged across all 100 trials, is visualized in Figure~\ref{3figs} as a function of the Wasserstein radius $\eps$ for each $\kappa\in\{0.1,0.25,0.5,0.75,\infty\}$. The best out-of-sample CCR is obtained for $\kappa= 0.25$ uniformly across all Wasserstein radii, and performance deteriorates significantly when $\kappa$ is reduced or increased. Recall from Remark~\ref{rem:drclassification->regularization} that, as $\kappa$ tends to infinity, the DRSVM reduces to the classical regularized support vector machine (RSVM) with $1$-norm regularizer. Thus, the results of Figure~\ref{3figs} indicate that regularization via mass transportation may be preferable to classical regularization in terms of the maximum achievable out-of-sample CCR. More specifically, we observe that the out-of-sample CCR of the best DRSVM ($\kappa= 0.25$) displays a slightly higher and significantly wider plateau around the optimal regularization parameter $\eps$ than the classical RSVM ($\kappa=\infty$). This suggests that the regularization parameter in DRSVMs may be easier to calibrate from data than in RSVMs, a conjecture that will be put to scrutiny in Section~\ref{sec:learned-parameters}. Finally, Figure~\ref{3figs} reveals that the standard (unregularized) support vector machine (SVM), which can be viewed as a special case of the DRSVM with $\eps=0$, is dominated by the RSVMs and DRSVMs across a wide range of regularization parameters.\footnote{\change{By slight abuse of notation, we use the acronym `SVM' to refer to the unregularized empirical hinge loss minimization problem even though the traditional formulations of the support vector machine involve a Tikhonov regularization term.}} Note that the SVM problem~\eqref{HLP} with $\eps=0$ reduces to a linear program and may thus suffer from multiple optimal solutions. This explains why the limiting out-of-sample CCR for $\eps\downarrow 0$ changes with $\kappa$. 
	
	\subsection{Regularization with Learned Parameters}
	\label{sec:learned-parameters}
	It is easy to read off the best regularization parameters $\eps$ and $\kappa$ from the charts in Figure~\ref{3figs}. As these charts are constructed from more than 12{,}000 test samples, however, they are not accessible in the training phase. In practice, $\eps$ and $\kappa$ must be calibrated from the training data alone. This motivates us to revisit the three classification problems from Section~\ref{sec:pre-selected-parameters} using a fully data-driven procedure, where all free model parameters are calibrated via $5$-fold cross validation; see, {\em e.g.},~\citep[\S~4.3.3]{abu2012learning}. Moreover, to evaluate the benefits of kernelization, we now solve a generalized DRSVM model of the form~\eqref{eq:kernelized-classification-problem}, which implicitly optimizes over all nonlinear hypotheses in some RKHS. As explained in Section~\ref{subsec:Kernelization}, kernelization necessitates the use of the separable transportation metric~\eqref{metric} with the Euclidean norm on the input space. 
	
	All free parameters of the resulting DRSVM model are restricted to finite search grids in order to ease the computational burden of cross validation. Specifically, we select the Wasserstein radius $\eps$ from within $\{ b \cdot 10^e : b \in \{1,5\}, e \in \{1, 2, 3, 4\}  \}$ and the label flipping cost $\kappa$ from within $\{0.1,0.25,0.5,0.75,\infty\}$. Moreover, we select the degree $d$ of the polynomial kernel from within $\{1,2,3,4,5\}$ and the peakedness parameter $\gamma$ of the Laplacian and Gaussian kernels from within $\{\frac{1}{100}, \frac{1}{81}, \frac{1}{64}, \frac{1}{49}, \frac{1}{36}, \frac{1}{25}\}$. Otherwise, we use the same experimental setup as in Section~\ref{sec:pre-selected-parameters}. Table~\ref{table-kernel} reports the averages and standard deviations of the CCR scores on the test data based on 100 independent trials. We observe that the DRSVM ($\eps$, $\kappa$, $d$, and $\gamma$ learned by cross validation) outperforms the RSVM ($\eps$, $d$ and $\gamma$ learned by cross validation, $\kappa=\infty$) consistently across all tested kernel functions (Polynomial, Laplacian, Gaussian). Note that the DRSVM with polynomial kernel subsumes the non-kernelized DRSVM~\eqref{HLP} as a special case because the polynomial kernel with $d=1$ coincides with the linear kernel.

	\begin{figure*}
		\centering
		\hspace{-6pt} \subfigure[1-vs-7]{\label{3figs-a} \includegraphics[width=0.31\columnwidth]{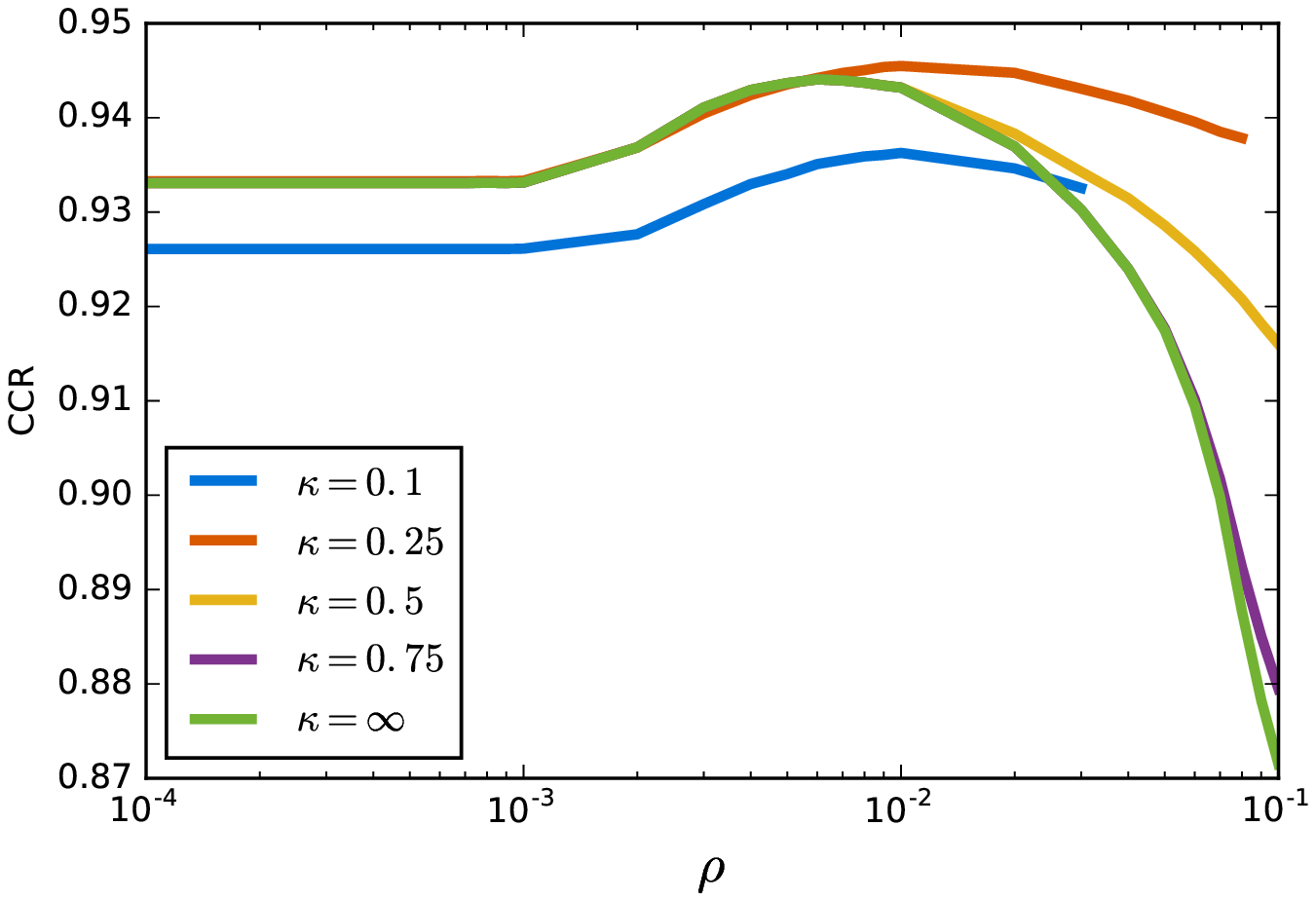}} \hspace{0pt}
		\subfigure[3-vs-8]{\label{3figs-b} \includegraphics[width=0.31\columnwidth]{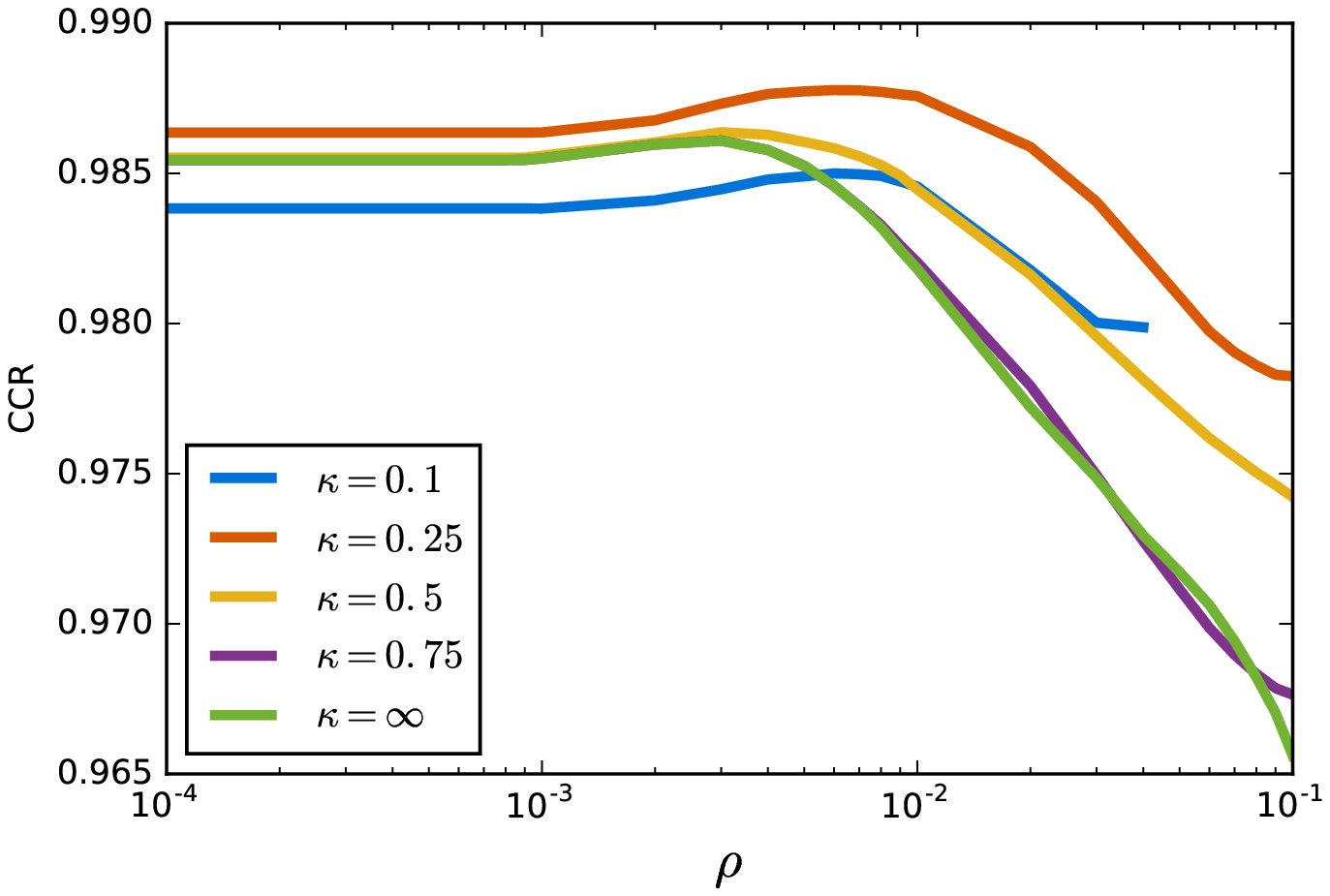}} \hspace{0pt}
		\subfigure[4-vs-9]{\label{3figs-c} \includegraphics[width=0.31\columnwidth]{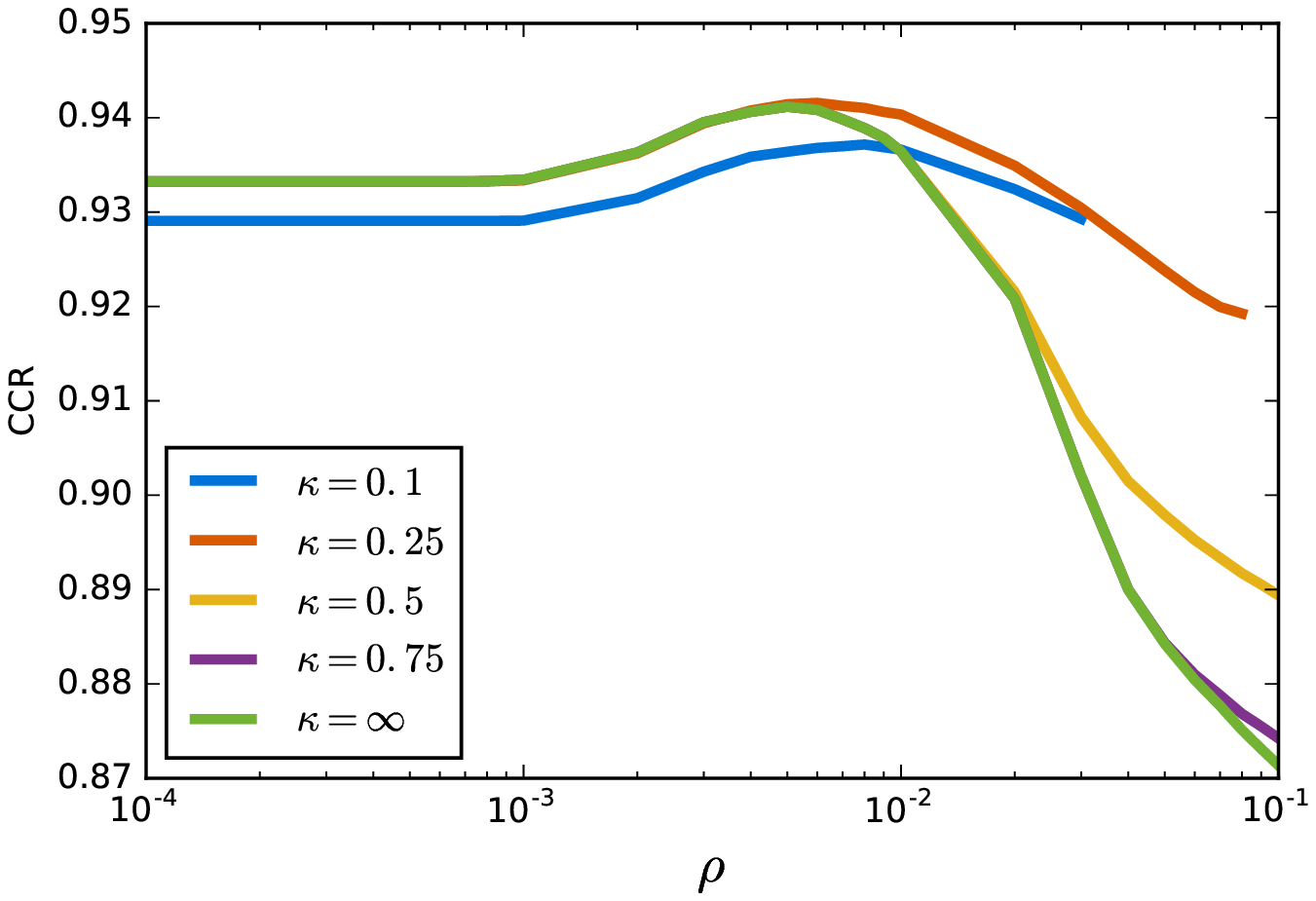}} \hspace{0pt}
		\caption{Average out-of-sample CCR scores of the DRSVM with pre-selected parameters.}
		\label{3figs}
	\end{figure*}

	\begin{table} 
		\centering
		\caption{Average out-of-sample CCR scores of the DRSVM with learned parameters.}
		\vspace{2ex}
		\bgroup
		\def\arraystretch{1.1}
		\begin{tabular}{|l|cc|cc|cc|}
			\cline{2-7}
			\multicolumn{1}{c|}{} & \multicolumn{2}{c|}{Polynomial} &
			\multicolumn{2}{c|}{Laplacian} & \multicolumn{2}{c|}{Gaussian} \\ 
			\cline{2-7}
			\multicolumn{1}{c|}{} & RSVM & DRSVM & RSVM & DRSVM & RSVM & DRSVM \\ \hline 
			1-vs-7& $ 98.9 \pm 0.2 $ & $ 99.1 \pm 0.2 $ & $ 98.3 \pm 0.5 $ & $ 98.5 \pm 0.4 $ & $ 99.1 \pm 0.2 $ & $ 99.2 \pm 0.2 $ \\ \hline 
			3-vs-8& $ 95.2 \pm 0.4 $ & $ 97.0 \pm 0.4 $ & $ 96.5 \pm 0.4 $ & $ 96.8 \pm 0.4 $ & $ 97.0 \pm 0.3 $ & $ 97.2 \pm 0.3 $ \\ \hline 
			4-vs-9& $ 95.0 \pm 0.4 $ & $ 96.5 \pm 0.4 $ & $ 95.8 \pm 0.6 $ & $ 96.0 \pm 0.6 $ & $ 96.8 \pm 0.4 $ & $ 96.9 \pm 0.4 $ \\ \hline 
		\end{tabular} \label{table-kernel}
		\egroup
	\end{table}
	
	In the third experiment, we assess the out-of-sample performance of the DRSVM~\eqref{HLP} for different transportation metrics on 10 standard datasets from the UCI repository~\citep{UCI2013}, \change{each containing up to $1{,}000$ samples}. Specifically, we use different variants of the separable transportation metric~\eqref{metric}, where distances in the input space are measured via a $p$-norm with $p\in\{1,2,\infty\}$. We focus exclusively on linear hypotheses because the kernelization techniques described in Section~\ref{subsec:Kernelization} are only available for $p=2$. The DRSVM is compared against the standard (unregularized) SVM and the RSVM with $q$-norm regularizer ($\frac{1}{p}+\frac{1}{q}=1$). All results are averaged across 100 independent trials. \change{In each trial, we randomly select 75\% of the data for training and the remaining 25\% for testing.} The inputs are first standardized to zero mean and unit variance along each coordinate axis. The Wasserstein radius $\eps$ and the label flipping cost $\kappa$ in the DRSVM as well as the regularization weight $\eps$ in the RSVM are estimated via stratified 5-fold cross~validation. 
	
	Classifier performance is now quantified in terms of the {\em receiver operating characteristic} (ROC) curve, which plots the true positive rate (percentage of correctly classified test samples with true label $y=1$) against the false positive rate (percentage of {\em in}correctly classified test samples with true label $y=-1$) by sweeping the discrimination threshold. Specifically, we use the {\em area under the ROC curve} (AUC) as a measure of classifier performance. AUC does not bias on the size of the test data and is a more appropriate performance measure than CCR in the presence of an unbalanced label distribution in the training data. We emphasize that most of the considered datasets are indeed imbalanced, and thus a high CCR score would not necessarily provide evidence of superior classifier performance. The averages and standard deviations of the AUC scores based on 100 trials are reported in Table~\ref{table_class_norms}. The results suggest that the DRSVM outperforms the RSVM in terms of AUC for all norms by about the same amount by which the RSVM outperforms the classical hinge loss minimization, consistently across all datasets.

	\begin{table} [h]
		\centering
		\caption{Average out-of-sample AUC scores of the SVM, RSVM and DRSVM.}\vspace{2ex}
		\bgroup
		\def\arraystretch{1.1}
		\begin{tabular}{|l|c|cc|cc|cc|}
			\cline{3-8}
			\multicolumn{2}{c|}{} & \multicolumn{2}{c|}{$p=\infty$ / $q=1$} & 
			\multicolumn{2}{c|}{$p=2$ / $q=2$} & \multicolumn{2}{c|}{$p=1$ / $q=\infty$} \\ 
			\cline{2-8}
			\multicolumn{1}{c|}{} & SVM & RSVM & DRSVM & RSVM & DRSVM & RSVM & DRSVM \\ \hline 
			Australian& $ 91.6 \pm 3.0 $ & $ 91.5 \pm 3.2 $ & $ 92.0 \pm 2.5 $ & $ 92.0 \pm 2.2 $ & $ 92.3 \pm 2.0 $ & $ 91.9 \pm 2.8 $ & $ 92.2 \pm 2.4 $ \\ \hline 
			Blood transfusion& $ 73.7 \pm 3.8 $ & $ 73.8 \pm 3.8 $ & $ 75.5 \pm 3.8 $ & $ 74.9 \pm 3.5 $ & $ 75.5 \pm 3.7 $ & $ 75.4 \pm 3.4 $ & $ 75.4 \pm 3.7 $ \\ \hline 
			Climate model& $ 93.8 \pm 3.9 $ & $ 94.4 \pm 4.0 $ & $ 94.3 \pm 4.0 $ & $ 94.3 \pm 3.8 $ & $ 94.0 \pm 4.0 $ & $ 93.6 \pm 3.9 $ & $ 93.9 \pm 4.0 $ \\ \hline 
			Cylinder& $ 72.0 \pm 3.7 $ & $ 71.2 \pm 4.0 $ & $ 72.1 \pm 4.1 $ & $ 71.3 \pm 4.0 $ & $ 71.8 \pm 3.8 $ & $ 71.5 \pm 3.8 $ & $ 72.2 \pm 3.7 $ \\ \hline 
			Heart& $ 90.4 \pm 2.7 $ & $ 90.1 \pm 2.7 $ & $ 90.3 \pm 2.7 $ & $ 90.6 \pm 2.6 $ & $ 90.9 \pm 2.5 $ & $ 90.5 \pm 2.6 $ & $ 90.7 \pm 2.6 $ \\ \hline 
			Ionosphere& $ 85.0 \pm 4.9 $ & $ 89.7 \pm 4.5 $ & $ 89.2 \pm 4.3 $ & $ 90.4 \pm 3.7 $ & $ 89.9 \pm 3.9 $ & $ 86.0 \pm 4.9 $ & $ 87.2 \pm 4.8 $ \\ \hline 
			Liver disorders& $ 60.5 \pm 0.0 $ & $ 61.1 \pm 0.7 $ & $ 61.7 \pm 0.7 $ & $ 61.2 \pm 0.3 $ & $ 61.7 \pm 0.5 $ & $ 61.1 \pm 0.4 $ & $ 61.8 \pm 0.5 $ \\ \hline 
			QSAR& $ 90.5 \pm 1.5 $ & $ 90.5 \pm 1.6 $ & $ 91.0 \pm 1.6 $ & $ 90.5 \pm 1.5 $ & $ 91.2 \pm 1.5 $ & $ 90.6 \pm 1.5 $ & $ 91.1 \pm 1.6 $ \\ \hline 
			Splice& $ 92.1 \pm 0.0 $ & $ 93.0 \pm 0.4 $ & $ 93.1 \pm 0.1 $ & $ 92.5 \pm 0.1 $ & $ 92.6 \pm 0.1 $ & $ 92.0 \pm 0.1 $ & $ 92.5 \pm 0.1 $ \\ \hline 
			Thoracic surgery& $ 61.7 \pm 7.1 $ & $ 61.5 \pm 6.5 $ & $ 64.6 \pm 6.6 $ & $ 64.4 \pm 6.4 $ & $ 64.3 \pm 7.0 $ & $ 64.0 \pm 6.3 $ & $ 64.6 \pm 6.3 $ \\ \hline 
		\end{tabular} \label{table_class_norms}
		\egroup
	\end{table}

\change{
	\subsection{Multi-Label Classification}
	The aim of object recognition is to discover instances of particular object classes in digital images. We now describe an object recognition experiment based on the PASCAL VOC 2007 dataset \cite{everingham2010pascal} consisting of $9{,}963$ images, which are pre-partitioned into $25\%$ for training, $25\%$ for validation and $50\%$ for testing. Each image is annotated with 20 binary labels corresponding to 20 given object categories (the $n$-th label is set to $+1$ if the image contains the $n$-th object and to $-1$ otherwise). A multi-label classifier is a function that predicts all labels of an unlabelled input image. The ability of a classifier to detect objects belonging to any fixed category is measured by the {\em average precision} (AP), which is defined in~\cite{everingham2010pascal} as (a proxy for) the area under the classifier's precision-recall curve. The overall performance of a classifier is quantified by the {\em mean average precision} (mAP), that is, the arithmetic mean of the AP scores across all object categories. 
	
	In the first scenario, we train a separate binary RSVM and DRSVM classifier for each of the 20 object categories. This classifier predicts whether an object of the respective category appears in the input image. At the beginning we preprocess the entire dataset by resizing each image to $256\times 256$ pixels and extracting the central patch of $244\times 244$ pixels. As shown in \citep{chatfield2014return, donahue2014decaf, zeiler2014visualizing}, the features generated by the penultimate layer of a deep convolutional neural network trained on a large image dataset provide a powerful image descriptor. Using the ALEXNET neural network trained on the ImageNet dataset~\citep{krizhevsky2012imagenet}, we can thus compress each (preprocessed) image of the PASCAL VOC 2007 dataset into 1{,}000 meaningful features. We normalize these feature vectors to lie on the unit sphere. When training the RSVM and DRSVM classifiers, we can thus work with these feature vectors instead of the corresponding images. Moreover, we restrict attention to linear hypotheses and assume that transportation distances in the input-output space are measured by the separable metric~\eqref{metric} with the Euclidean norm on the input space. We tune the Wasserstein radius $\rho \in\{ b \cdot 10^e : b \in \{1,\ldots,9\}, e \in \{-2,-3,-4\}  \}$ and the label flipping cost $\kappa  \in\{ 0{.}1,0{.}2,\ldots,1,\infty\}$ via the holdout method using the validation data. As usual, we fix $\kappa=\infty$ for RSVM. Table~\ref{table_voc} reports the AP scores of the RSVM and DRSVM models for each object category. The ensemble of all 20 binary RSVM or DRSVM classifiers, respectively, can be viewed as a na\"ive multi-label classifier that predicts all labels of an image. As DRSVM outperforms RSVM on an object-by-object basis, it also wins in terms of mAP.
	
	In the second scenario, we construct a proper multi-label classifier by fine-tuning the last layer of the pre-trained ALEXNET network. To this end, we replace the original $M$-th layer of the network with a new fully connected layer characterized by a parameter matrix $\bm W_M\in\RR^{20\times 1000}$, and we set $\sigma_M$ to the Sigmoid activation function. The resulting classifer outputs for each of the 20 object categories a probability that an object from the respective category appears in the input image. The quality of a classifier (which is encoded by $\bm W_M$) is measured by the cross-entropy loss function, which naturally generalizes the logloss to multiple labels. The resulting empirical loss minimization problem is enhanced with a regularization term proportional to $\|\bm W_M\|_{1,1}$ (Lasso), $\|\bm W_M\|_{F}^2$ (Tikhonov), $\|\bm W_M\|_{1}$ (MACS), $\|\bm W_M\|_{2}$ (Spectral) or $\|\bm W_M\|_{\infty}$ (MARS). By using similar arguments as in Section~\ref{subsec:neural}, one can show that the empirical cross-entropy with MACS, Spectral or MARS regularization term overestimates the worst-case expected cross-entropy over all distributions of $(\bm x_{M},\bm x_{M+1})$ in a Wasserstein ball provided that the transportation cost is given by 
	\[
	d((\x_M,\bm x_{M+1}),(\x_M', \bm x_{M+1}')) = \|\x_M-\x_M'\|_p + \kappa \mathds{1}_{ \{ \x_{M+1} \neq \x_{M+1}' \} }
	\]
	for $\kappa=\infty$, whenever $p=1$, $p=2$ or $p=\infty$, respectively. Thus, the MACS, Spectral and MARS regularization terms admit a distributionally robust interpretation.

	We use the stochastic proximal gradient descent algorithm of Section~\ref{subsec:neural} to tune $\bm W_M$, including an additional momentum term with weight~$0{.}9$. As in~\citep{krizhevsky2012imagenet}, we split the training phase into 100 epochs, each corresponding to a complete pass through the training dataset in a random order. As the ALEXNET requires input images of size $244\times 244$, in each iteration we extract a random patch of $244 \times 244$ pixels from the current image and flip it horizontally at random. This procedure effectively augments the training dataset. The initial step size is set to $10^{-3}$ and then reduced by a factor of $10$ after every $7$ epochs. The algorithm terminates after $100$ epochs. We preprocess the images in the validation and test datasets as in Scenario~1 and tune the regularization weights via the holdout method using the validation data. 
	Table~\ref{table_voc} reports the AP and mAP scores of the different classifiers that were tested. These results suggest that fine-tuning the last layer of a pre-trained neural network may improve classifier performance. We observe that the spectral norm regularizer, which has a distributionally robust interpretation, consistently outperforms almost all other methods. For further details on the experimental setup (such as the exact search grids for all hyperparameters) we refer to the code publicized on Github.

	\begin{table}
		\centering
		\caption{AP scores of different multi-label classifiers.}\vspace{2ex}
		\bgroup
		\def\arraystretch{1.1}
		\begin{tabular}{|c|c|c|c|c|c|c|c|c|}
			\cline{2-8}
			\multicolumn{1}{c|}{} & \multicolumn{2}{c|}{Scenario 1} & \multicolumn{5}{c|}{Scenario 2} \\
			\cline{2-8}
			\multicolumn{1}{c|}{} & RSVM & DRSVM & Lasso & Tikhonov & MACS & Spectral & MARS \\ \hline
			\includegraphics[width=0.5cm,height=0.4cm,keepaspectratio]{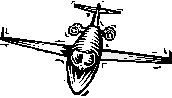} & 84.80  & 84.85  & 85.50  & 84.26 & 84.35 & 83.89 & 85.53 \\
			\includegraphics[width=0.5cm,height=0.4cm,keepaspectratio]{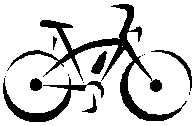} & 78.54 & 78.49 & 76.18 & 76.55 & 76.37 & 75.67 & 76.22 \\
			\includegraphics[width=0.5cm,height=0.4cm,keepaspectratio]{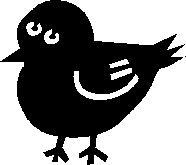} & 82.19 & 82.19 & 83.37 & 83.11 & 83.52 & 84.19 & 83.08 \\
			\includegraphics[width=0.5cm,height=0.4cm,keepaspectratio]{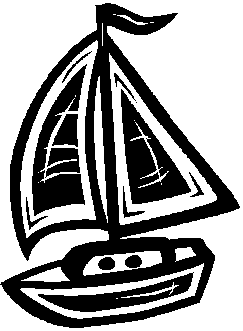} & 79.55 & 79.56 & 77.65 & 78.04 & 77.70 & 78.92 & 77.53 \\
			\includegraphics[width=0.5cm,height=0.4cm,keepaspectratio]{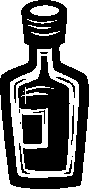} & 37.52 & 37.98 & 38.10  & 39.73 & 39.53 & 38.75 & 38.10 \\
			\includegraphics[width=0.5cm,height=0.4cm,keepaspectratio]{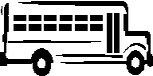} & 72.05 & 72.03 & 70.05 & 69.43 & 69.17 & 70.28 & 69.77 \\
			\includegraphics[width=0.5cm,height=0.4cm,keepaspectratio]{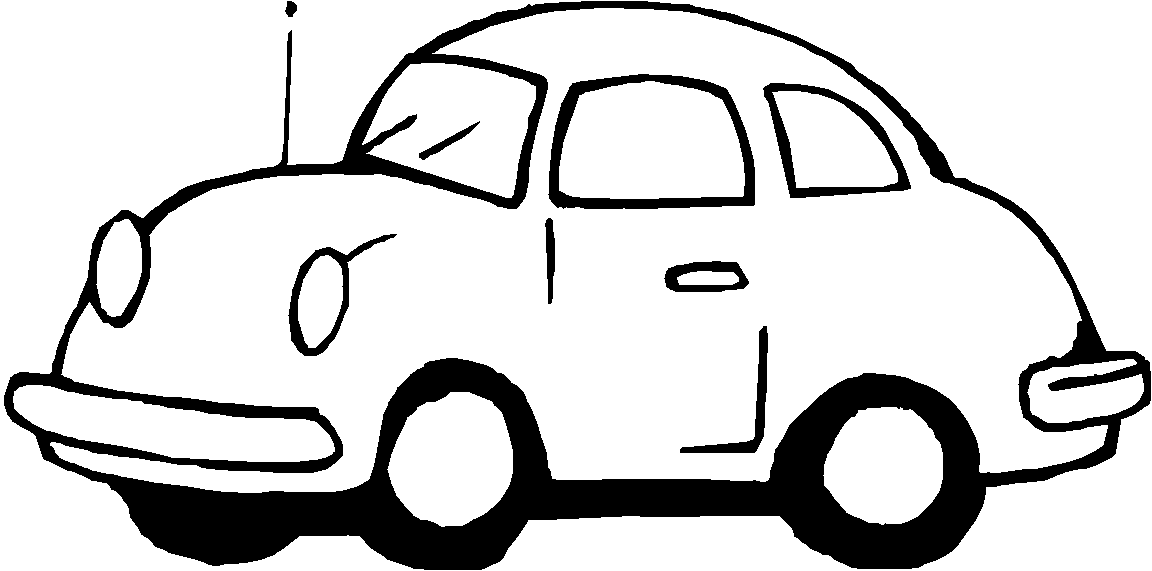} & 83.08 & 83.12 & 83.10 & 83.67 & 83.50 & 82.92 & 83.17 \\
			\includegraphics[width=0.5cm,height=0.4cm,keepaspectratio]{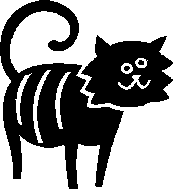} & 80.60 & 80.56 & 79.87 & 79.86 & 80.07 & 79.85 & 79.82 \\
			\includegraphics[width=0.5cm,height=0.4cm,keepaspectratio]{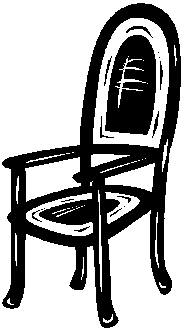} & 54.64 & 54.64 & 54.40 & 54.76 & 53.94 & 54.72 & 54.55 \\
			\includegraphics[width=0.5cm,height=0.4cm,keepaspectratio]{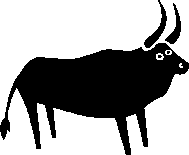} & 47.82 & 53.13 & 52.06 & 52.10 & 51.62 & 55.54 & 51.19 \\
			\includegraphics[width=0.5cm,height=0.4cm,keepaspectratio]{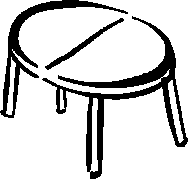} & 54.26 & 58.88 & 63.41 & 65.23 & 65.15 & 66.79 & 62.95 \\
			\includegraphics[width=0.5cm,height=0.4cm,keepaspectratio]{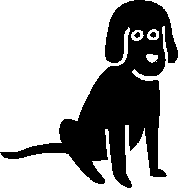} & 75.81 & 75.81 & 76.92 & 77.39 & 77.26 & 76.54 & 76.97 \\
			\includegraphics[width=0.5cm,height=0.4cm,keepaspectratio]{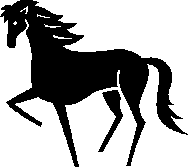} & 82.74 & 82.72 & 82.17 & 81.89 & 81.6 & 80.81 & 81.9 \\
			\includegraphics[width=0.5cm,height=0.4cm,keepaspectratio]{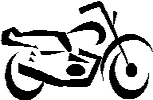} & 72.69 & 72.88 & 74.70 & 75.41 & 74.76 & 76.48 & 74.31 \\
			\includegraphics[width=0.5cm,height=0.4cm,keepaspectratio]{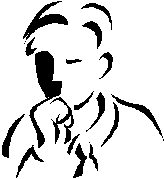} & 90.36 & 90.36 & 90.07 & 90.30 & 90.22 & 90.35 & 90.09 \\
			\includegraphics[width=0.5cm,height=0.4cm,keepaspectratio]{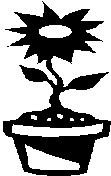} & 50.22 & 51.90 & 50.20 & 50.18 & 50.27 & 51.39 & 50.20 \\
			\includegraphics[width=0.5cm,height=0.4cm,keepaspectratio]{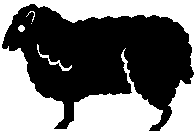} & 60.75 & 63.57 & 71.78 & 71.3  & 70.39 & 71.64 & 71.40 \\
			\includegraphics[width=0.5cm,height=0.4cm,keepaspectratio]{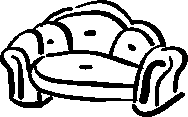} & 56.85 & 56.98 & 52.15 & 54.36 & 54.65 & 55.12 & 51.94 \\
			\includegraphics[width=0.5cm,height=0.4cm,keepaspectratio]{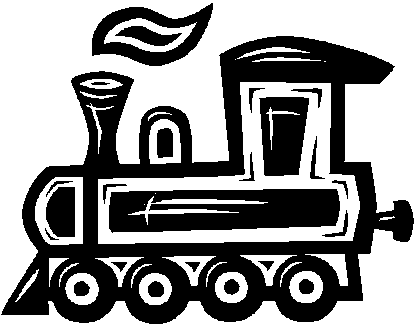} & 85.09 & 85.03 & 84.89 & 84.55 & 84.41 & 85.43 & 84.96 \\
			\includegraphics[width=0.5cm,height=0.4cm,keepaspectratio]{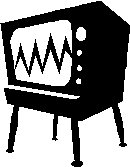} & 69.02 & 69.08 & 64.73 & 65.85 & 65.63 & 64.26 & 64.48 \\ \hline 
			mAP & 69.92 & 70.69 & 70.56 & 70.90 & 70.71 & 71.20 & 70.40  \\ \hline
		\end{tabular} \label{table_voc}
		\egroup
	\end{table}
	}

	\subsection{Generalization Bounds}
	The next experiment estimates the scaling behavior of the smallest Wasserstein radius that verifies the generalization bound~\eqref{finite:sample:guarantee} for the synthetic \texttt{threenorm} classification problem~\citep{breiman1996bias}. The experiment involves 1{,}000 simulation trials. In each trial we generate $N$ training samples for some $N\in \{10,\ldots,90\}\cup\{100,\ldots,1{,}000\}$ as well as $10^5$ test samples. Each sample $(\x,y)\in\RR^{20}\times\{-1,1\}$ is constructed as follows. The label $y$ is drawn uniformly from $\{-1,1\}$. If $y=-1$, then $\x$ is drawn from a standard multivariate normal distribution shifted by $(c, \ldots, c)$ or $(-c, \ldots, -c)$ with equal probabilities, where $c=2/\sqrt{20}$. If $\widehat y=1$, on the other hand, then $\x$ is drawn from a standard multivariate normal distribution shifted by $(c,-c,+c, \ldots, -c)$. 
	
	We now describe three different approaches to choose the Wasserstein radius $\eps$ in the DRSVM~\eqref{HLP} with transportation cost~\eqref{metric}, where $\kappa=\infty$ and $\|\cdot\|$ represents the $\infty$-norm on the input space. Throughout the experiment we use $P=\{ b \cdot 10^{-e} : b \in \{1,\ldots,10\}, e \in \{1, \ldots, 5\}  \}$ as the search space for $\eps$. Approach~1 (`cross validation') calibrates the Wasserstein radius as before via 5-fold cross validation based solely on the $N$ training samples. This approach reflects what would typically be done in practice. Approaches~2 and~3 both solve~\eqref{HLP} based on the empirical distribution induced by the $N$ training samples and select the Wasserstein radius using the $10^5$ test samples. Specifically, approach~2 (`optimal') chooses the Wasserstain radius that leads to the lowest test error, while approach~3 (`generalization bound') selects the smallest Wasserstein radius for which the optimal value of~\eqref{HLP} exceeds the expected loss on the test samples in at least $95\%$ of all trials, that is, it approximates the smallest Wasserstein radius that verifies the generalization bound~\eqref{finite:sample:guarantee} for $\eta=5\%$. As the test samples are not available in the training phase, the last two approaches are not implementable in practice, and we merely study them to gain insights. Figure~\ref{3figsp-a} visualizes all resulting Wasserstein radii as a function of $N$. Note that the radii obtained with the first two approaches are uncertain as they depend on a particular choice of the training samples. Figure~\ref{3figsp-a} thus only shows their averages across all simulation trials. In contrast, the radii obtained with the third approach depend on the training sample sets of all $1{,}000$ trials and are thus essentially deterministic.
	
	We observe that the Wasserstein radii of all three approaches decay approximately as $1/\sqrt{N}$, which is in line with the theoretical generalization bound of Theorem~\ref{Thrm:improved}. We expect this decay rate to be optimal because any faster decay would be in conflict with the central limit theorem. Note also that our results empirically confirm Theorem~\ref{Thrm:improved} even though we did not impose any restrictions on $\WW$ as dictated by Assumption~\ref{Asmp:improved}. This suggests that Theorem~\ref{Thrm:improved} might remain valid under weaker conditions.

	In the experiment underlying Figure~\ref{3figsp-c}, we first fix $\widehat \w$ to an optimal solution of~\eqref{HLP} for $\rho = 0.1$ and $N=100$. Figure~\ref{3figsp-c} shows the true risk $\risk (\widehat \w)$ and its confidence bounds given by Theorem~\ref{Thrm:Risk}. As expected, for $\rho=0$ the upper and lower bounds coincide with the empirical risk on the training data, which is a lower bound for the true risk on the test data due to over-fitting effects. As $\rho$ increases, the confidence interval between the bounds widens and eventually covers the true risk. For instance, at $\rho \approx 0.009$ the confidence interval is given by $[0.008, 0.162]$ and contains the true risk with probability $1-\eta = 95\%$.
	\begin{figure*}
		\centering
		\hspace{-6pt} \subfigure[Dependence of the Wasserstein radius on the number of training samples]{\label{3figsp-a} \includegraphics[width=0.31\columnwidth]{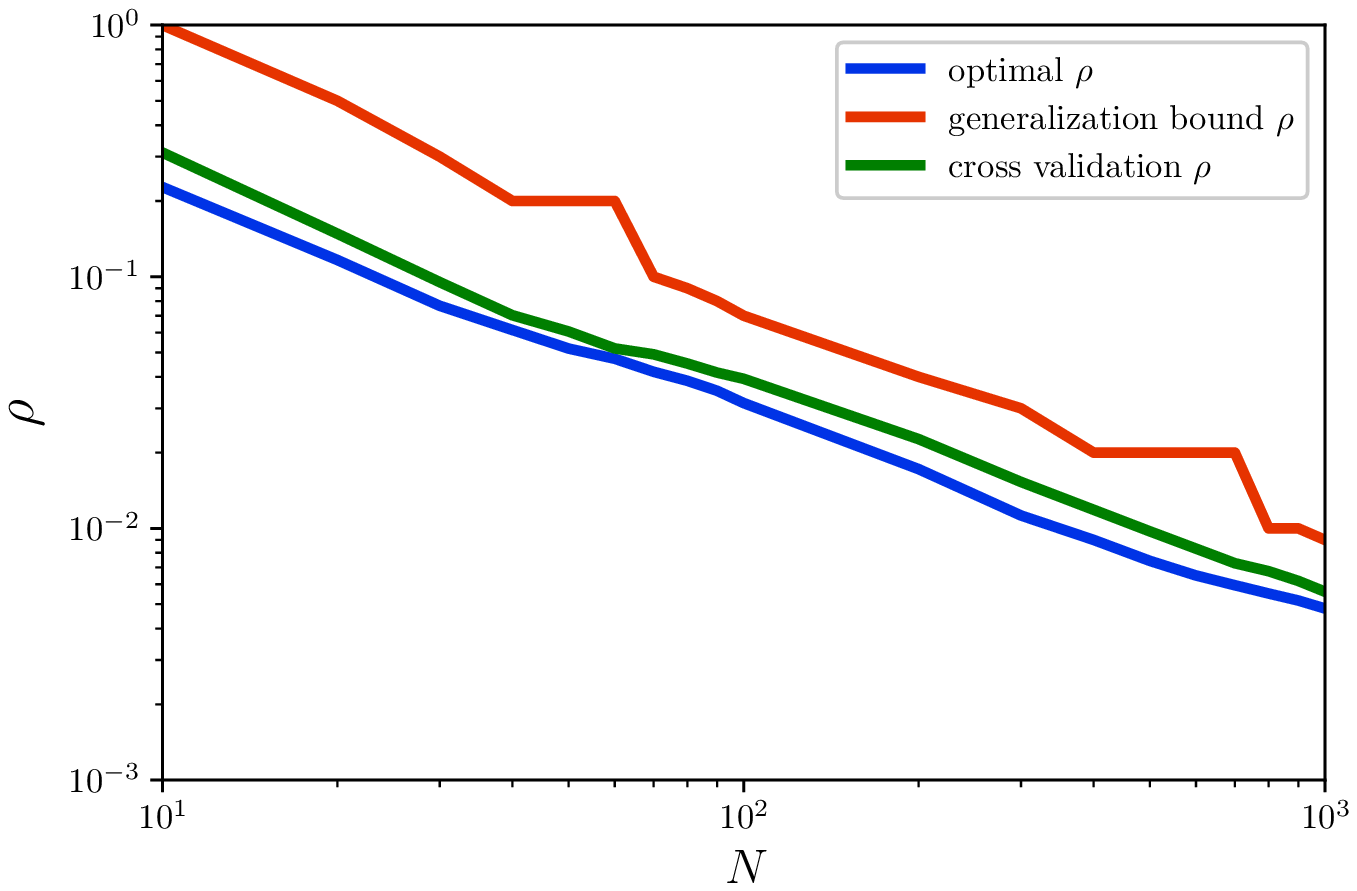}} \hspace{0pt}
		\subfigure[Confidence bounds on the risk]{\label{3figsp-c} \includegraphics[width=0.31\columnwidth]{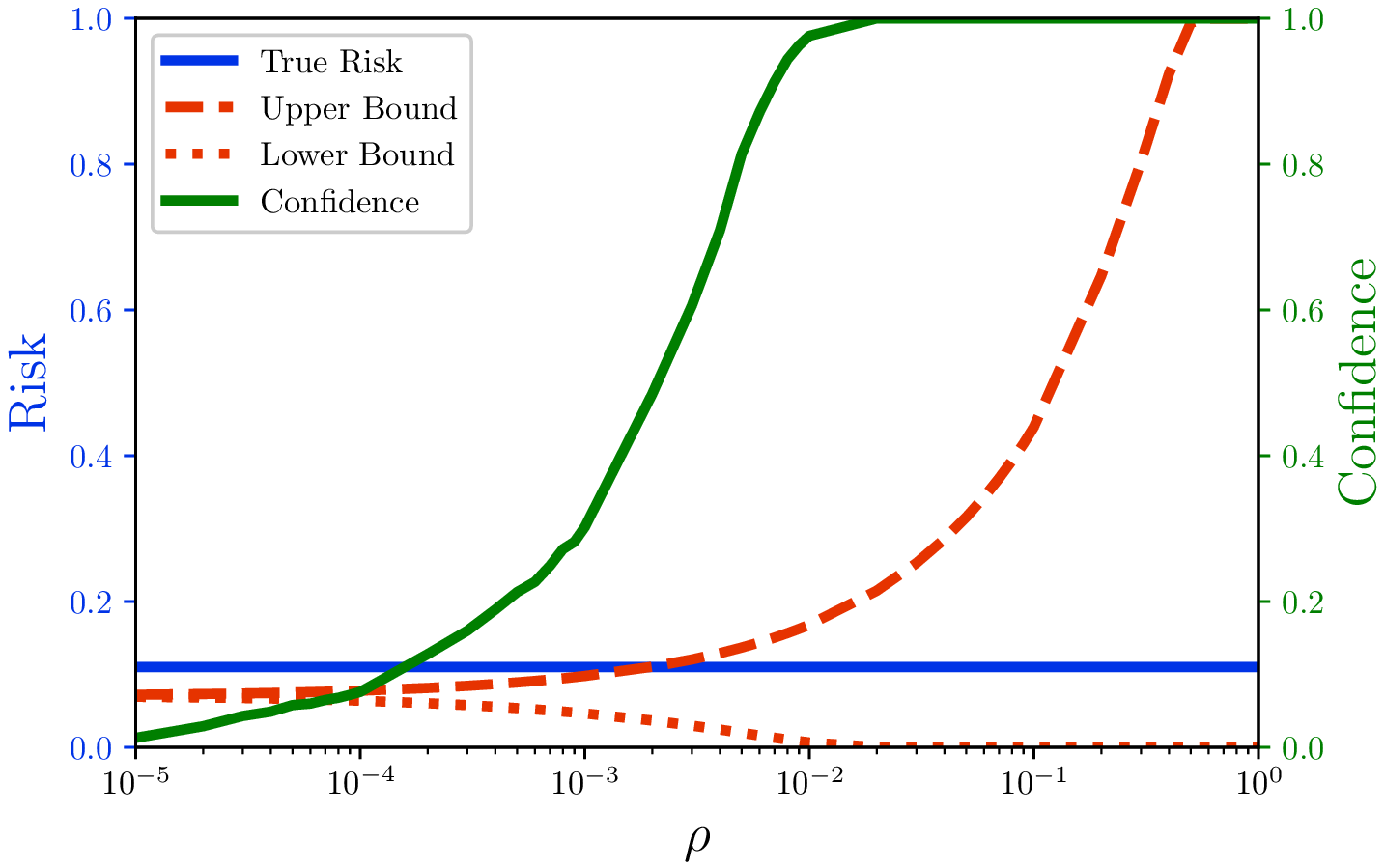}} \hspace{0pt}
		\caption{Results of the \texttt{threenorm} classification problem.}
		\label{3figs-prob}
	\end{figure*}

	\subsection{Worst-Case Distributions}
	Consider again the 3-vs-8 classification problems from the MNIST database~\citep{MNIST} and fix $\w^\star$ to an optimal solution of the empirical hinge loss minimization problem. The goal of the last experiment is to evaluate the {\em worst-case} hinge loss of $\w^\star$ for different Wasserstein radii $\eps\in\{0, 0.01, 0.05, 0.1, 0.5, 1\}$ and label flipping costs $\kappa\in\{0,\infty\}$ and to investigate the corresponding worst-case distributions, which are efficiently computable by virtue of Theorem~\ref{Thrm:worst-case:classification}(i). As each input constitutes a vector of pixels intensities between zero and one, we impose support constraints of the form $\bm C \bm x \leq \bm d$ with $\bm C = [\bm I, -\bm I]^\top$ and $\bm d = [\bm 1^\top, \bm 0^\top]^\top$.
	
	For illustrative purposes we only use the $N=10$ first datapoints in the MNIST dataset as training samples. Each training sample $\widehat{\x}_i$ corresponds to four discretization points ($\widehat{\x}_i + {\q^+}\opt_{ij}/{\alpha^+_{ij}}\opt$ and $\widehat{\x}_i + {\q^-}\opt_{ij}/{\alpha^-_{ij}}\opt$ for $j=1,2$) in the worst-case distribution obtained from~\eqref{worst:classification:linear}. We observe that for every $i$ exactly one out of these four points has probability $\frac{1}{N}$, while all others have probability 0. Figure~\ref{2figs-dist} depicts only those 10 discretization points that have nonzero probability for a fixed $\rho$ and $\kappa$. As expected, the perturbations of the training samples are more severe for larger Wasserstein radii. For $\kappa=\infty$ these scenarios must have the same labels as the corresponding training samples. For $\kappa=0$, on the other hand, the labels can be flipped at no cost (flipped labels are indicated by red frames). Each scenario group shown in Figure~\ref{2figs-dist} can thus be viewed as a worst-case training dataset for the corresponding Wasserstein radius and label flipping cost.
	\begin{figure*}
		\centering
		\hspace{-6pt} \subfigure[$\kappa = \infty$]{\label{2figsp-a} \includegraphics[width=0.35\columnwidth]{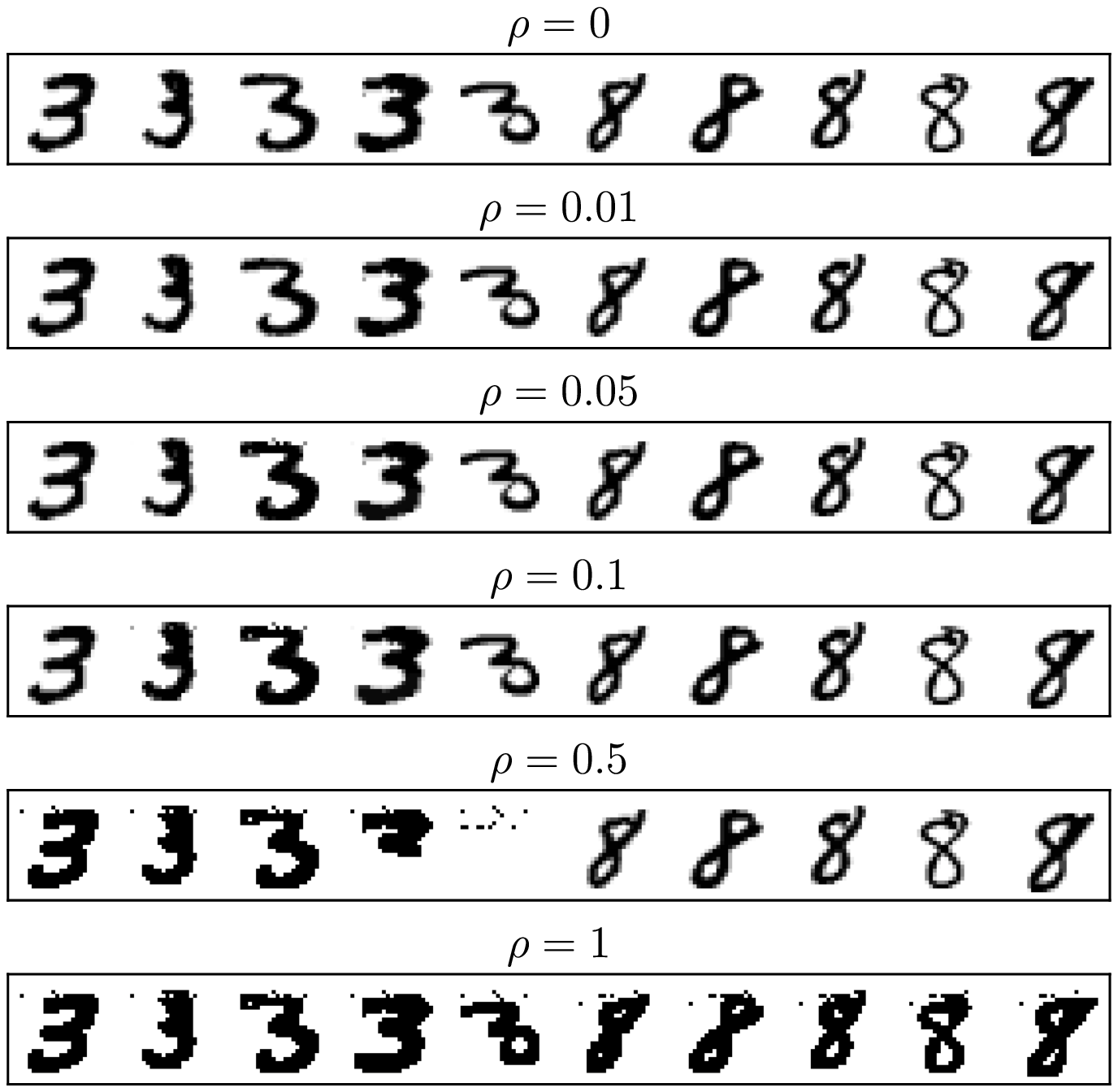}} \hspace{0pt}
		\subfigure[$\kappa = 0$]{\label{2figsp-b} \includegraphics[width=0.35\columnwidth]{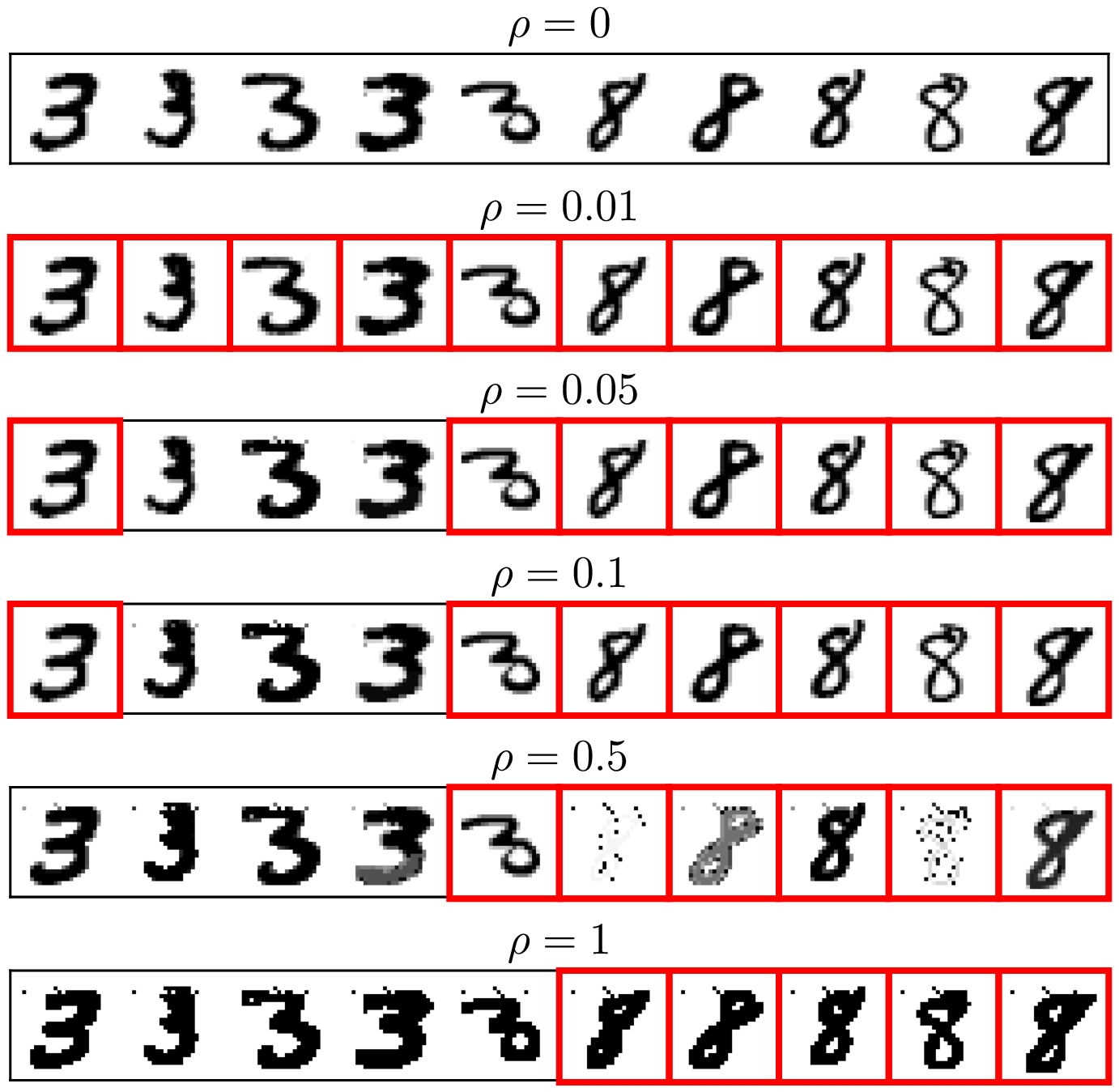}} \hspace{0pt}
		\caption{Discretization points (input images) of the worst-case distribution for different $\rho$ and $\kappa$. Red frames indicate that the corresponding labels are flipped under the worst-case distribution.}
		\label{2figs-dist}
	\end{figure*}

	
	\paragraph{\bf Acknowledgments}
	We gratefully acknoweldge financial support from the Swiss National Science Foundation under grants BSCGI0\_157733 and P2EZP2\_165264.

	
\appendix
\section{Appendix: Proofs} 
\label{sec:Appendix}
\subsection{Proofs of Section~\ref{sec:Tractable}}
The proof of Theorem~\ref{Thrm:Regression} requires three preparatory Lemmas. The first lemma is adapted from~\citep{MohKun-14} and asserts that the worst-case expectation over a Wasserstein ball can be re-expressed as a classical robust optimization problem. 

\begin{Lmm2} [Robust reformulation] \label{Lmm:main}
	Set $\widehat \xxi_i = (\widehat \x_i,\widehat y_i)$ for all $i\leq N$. For any measurable integrand $I(\xxi)$ that is bounded above by a Lipschitz continuous function we have
	\begin{align} \label{start:proof}
	\Sup _ {\QQ \in \Wball} \EE ^ \QQ  \left[ I(\xxi) \right] =  \Inf_{\lambda\geq 0} ~ \lambda \eps + \frac{1}{N}\sum_{i=1}^N \Sup_{\xxi \in \Xi} ~ I (\xxi) - \lambda d (\xxi,\widehat \xxi_i).
	\end{align}
\end{Lmm2}

\begin{proof} 
	By the definition of the Wasserstein ball we have
	\begin{align*}
	\Sup _ {\QQ \in \Wball} \EE ^ \QQ  \left[ I (\xxi) \right]~
	& = \optimize{
		\Sup _ { \Pi}  & \displaystyle \int _ {\Xi^2} I(\xxi)\, \Pi (\diff \xxi, \diff\xxi') \\[1ex]
		\text{s.t. }  & \Pi \text{ is a joint distribution of $\xxi$} \\ 
		& \text{and $\xxi'$ with marginals $\mathds{Q}$ and $\Pem$} \\[0.5ex]
		& \displaystyle \int_ {\Xi^2} d \left( \xxi,\xxi' \right) \, \Pi (\diff \xxi, \diff\xxi') \leq \eps} \\ 
	& = \optimize{
		\Sup_ {\QQ ^ i} & \displaystyle \frac{1}{N} \SumN \int_{\Xi}
		I(\xxi)\, \QQ^i(\diff \xxi) & \\
		\text{s.t. } & \displaystyle \int_{\Xi} \QQ ^ i (\diff \xxi) = 1\quad i \in [N]\\
		& \displaystyle \frac{1}{N} \SumN \int_{\Xi} d ( \xxi, \widehat \xxi_i )\, \QQ^i (\diff \xxi) \leq \eps.}
	\end{align*}
	Note that the integral of $I(\xxi)$ exists under every $\QQ \in \Wball$ because $I(\xxi)$ admits a Lipschitz continuous majorant. The last equality in the above expression holds because the marginal distribution of $\xxi'$ is the uniform distribution on the training samples, which implies that $\Pi$ is completely determined by the conditional distributions $\QQ ^ i$ of $\xxi$ given $\xxi' = \widehat \xxi_i$, that is, $\Pi(\diff \xxi,\diff \xxi' ) = \frac{1}{N}\sum_{i=1}^{N} \delta_{\widehat \xxi_i }(\diff \xxi') \QQ ^ i (\diff \xxi)$. The resulting generalized moment problem over the normalized measures $\QQ^i$ admits the semi-infinite dual
	\begin{align*} 
	\Sup _ {\QQ \in \Wball} \EE ^ \QQ  \left[ I(\xxi) \right] = 
	\optimize{
		\Inf_{\lambda,s_i} & \displaystyle \Obj & \\
		\text{s.t.}
		& \Sup_{\xxi \in \Xi} ~ I (\xxi) - \lambda d (\xxi,\widehat \xxi_i) \leq s_i & i \in [N] \\
		&\lambda \geq 0,}
	\end{align*}
	Strong duality holds for any $\eps > 0$ due to \citep[Proposition~3.4]{Shapiro-conic-duality}. The claim then follows by eliminating~$s_i$.
\end{proof}

\begin{Lmm2} \label{Lmm:linear}
	For any $a \in \RR$, $\Beta, \widehat{\bm \zeta} \in \RR^d$, $\gamma \in \RR_+$ and $\bm \zeta \in \ZZ$, where $\ZZ \subseteq \RR^d$ is a closed convex set, we have
	\begin{align*}
	\Sup_ {\bm \zeta \in \ZZ} ~ a \inner{\Beta}{\bm \zeta} - \gamma \| \bm \zeta - \widehat{\bm \zeta} \| = \optimize{
		\Inf_{\p}  & S_\ZZ(a \Beta - \p) + \inner{\p} {\widehat {\bm \zeta}}  \\
		\mathrm{s.t.} & \| \p \|_* \leq \gamma.}
	\end{align*} 
\end{Lmm2}
\begin{proof} 
	\begin{subequations}
		We have
		\begin{align*}
		\Sup_{\bm \zeta \in \ZZ} ~ a \inner{\Beta}{\bm \zeta} - \gamma \| \bm \zeta - \widehat{\bm \zeta} \|
		& = \Sup_ {\bm \zeta \in \ZZ} ~ \Inf_{ \| \p \|_* \leq \gamma} ~ a \inner{\Beta}{\bm \zeta} - \inner{\p}{\bm \zeta - \widehat{\bm \zeta}} \\
		& = \Inf_{ \| \p \|_* \leq \gamma} ~ \Sup_ {\bm \zeta \in \ZZ} ~ a \inner{\Beta}{\bm \zeta} - \inner{\p}{\bm \zeta - \widehat{\bm \zeta}} \\
		& = \Inf_{ \| \p \|_* \leq \gamma} ~ \Sup_ {\bm \zeta \in \RR^d} ~ \inner{a \Beta - \p}{\bm \zeta} - \delta_\ZZ(\bm \zeta) + \inner{\p}{\widehat{\bm \zeta}}  \\
		& = \Inf_{ \| \p \|_* \leq \gamma} ~ S_\ZZ(a \Beta - \p) + \inner{\p}{\widehat{\bm \zeta}},					
		\end{align*}
		where the first equality follows from the definition of the dual norm, the second equality holds due to the minimax theorem \citep[Proposition~5.5.4]{bertsekas2009convex}, and the last equality holds because the support function $S_\ZZ$ is the conjugate of the indicator function $\delta_\ZZ$. Thus, the claim follows.
	\end{subequations}
\end{proof}

\begin{Lmm2} \label{Lmm:convex}
	If $L(z)$ is a convex and Lipschitz continuous loss function, $\Beta, \widehat{\bm \zeta} \in \RR^d$ and $\gamma > 0$, then
	\begin{align*}
	\Sup_ {\bm \zeta \in \RR^d} ~ L (\inner{\Beta}{\bm \zeta}) - \gamma \| \bm \zeta - \widehat{\bm \zeta} \| = \optimize{
		L(\inner{\Beta}{\widehat{\bm \zeta}}) & \mathrm{if ~} \lip \| \Beta \|_* \leq \gamma \\ 
		+\infty & \mathrm{otherwise.}}
	\end{align*}
\end{Lmm2}			
\begin{proof} 
	Note that $L (\inner{\Beta}{\bm \zeta}) - \gamma \| \bm \zeta - \widehat{\bm \zeta} \|$ constitutes a difference of convex functions and may thus be neither convex nor concave in $\bm \zeta$. 
	In order to maximize this function, we re-write $I(\bm \zeta) = L (\inner{\Beta}{\bm \zeta})$ as an upper envelope of infinitely many affine functions. To this end, we express the conjugate of $I(\bm \zeta)$ as
	\begin{align*}
	I^*(\z) &= \Sup_{\bm \zeta} \inner{\z}{\bm \zeta} - L (\inner{\Beta}{\bm \zeta}) 
	= \Sup_{t, \, \bm \zeta} \left\{ \inner{\z}{\bm \zeta} - L(t) : t = \inner{\Beta}{\bm \zeta} \right\}
	= \Inf_{\theta} \left\{ L^*(\theta) : \theta \Beta = \z \right\},
	\end{align*}
	where the last equality follows from strong Lagrangian duality, which holds because Slater's constraint qualification is trivially satisfied in the absence of inequality constraints \citep[Proposition~5.3.1]{bertsekas2009convex}. Defining $\Theta=\{\theta\in\RR:L^*(\theta)<\infty\}$ as the effective domain of $L^*(\theta)$, we may then replace $\theta \in \RR$ with $\theta \in \Theta$ in the last expression. As $I(\bm \zeta)$ is convex and continuous, it coincides with its bi-conjugate, that is,
	\begin{align*}
	I(\bm \zeta) = I^{**}(\bm \zeta) = \sup_{\z} ~ \inner{\z}{\bm \zeta} - I^{*}(\z) = \sup_{\theta \in \Theta} ~ \inner{\theta \Beta}{\bm \zeta} - L^*(\theta).
	\end{align*}
	In other words, we have represented $I(\bm \zeta)$ as the upper envelope of infinitely many linear functions. Using this representation, we obtain	
	\begin{align*}
	\Sup_ {\bm \zeta} ~ I (\bm \zeta) - \gamma \| \bm \zeta - \widehat{\bm \zeta} \| 
	& = \Sup_ {\bm \zeta} ~ I^{**} (\bm \zeta) - \gamma \| \bm \zeta - \widehat{\bm \zeta} \| \\
	& = \Sup_ {\bm \zeta} ~ \Sup_{\theta \in \Theta} ~ \inner{\theta \Beta}{\bm \zeta} - L^{*}(\theta) - \gamma \| \bm \zeta - \widehat{\bm \zeta} \|  \\
	& = \Sup_{\theta \in \Theta} ~ \Sup_{\bm \zeta} ~ \Inf_{ \| \p \|_* \leq \gamma} ~ \theta \inner{\Beta}{\bm \zeta} - L^{*}(\theta) - \inner{\p}{\bm \zeta - \widehat{\bm \zeta}} \\
	& = \Sup_{\theta \in \Theta} ~ \Inf_{ \| \p \|_* \leq \gamma} ~ \Sup_ {\bm \zeta} ~ \inner{\theta \Beta - \p}{\bm \zeta} - L^{*}(\theta) + \inner{\p}{\widehat{\bm \zeta}}, \\
	\end{align*}
	where the last equality holds due to \citep[Proposition~5.5.4]{bertsekas2009convex}. Evaluating the maximization over $\bm \zeta$ yields
	\begin{align*}
	\Sup_ {\bm \zeta} ~ I (\bm \zeta) - \gamma \| \bm \zeta - \widehat{\bm \zeta} \| 
	& =
	\Sup_{\theta \in \Theta} ~ \Inf_{ \| \p \|_* \leq \gamma} ~
	\optimize{
		\inner{\p}{\widehat{\bm \zeta}} - L^{*}(\theta) & \text{if } \p = \theta \Beta \\
		+\infty & \text{otherwise}} \\
	&= \Sup_{\theta \in \Theta} ~ \optimize{
		\inner{\theta \Beta}{\widehat{\bm \zeta}} - L^{*}(\theta) & \text{if } \| \theta \Beta \|_* \leq \gamma \\
		+\infty & \text{otherwise}} \\
	&= \optimize{
		\Sup_{\theta \in \Theta} ~ \theta \inner{\Beta}{\widehat{\bm \zeta}} - L^{*}(\theta) & \text{if } \Sup_{\theta \in \Theta} ~ \| \theta \Beta \|_* \leq \gamma \\
		+\infty & \text{otherwise}} \\
	&= \optimize{
		L(\inner{\Beta}{\widehat{\bm \zeta}}) & \mathrm{if ~} \Sup_{\theta \in \Theta} |\theta| \cdot \| \Beta \|_* \leq \gamma \\
		+\infty & \mathrm{otherwise.}}
	\end{align*}
	Thus, the claim follows by noting that $\sup_{\theta} \{ |\theta|: L^*(\theta) < \infty \}$ represents the Lipschitz modulus of~$L$. 
\end{proof}

\begin{proof} [Proof of Theorem~\ref{Thrm:Regression}]
	To prove assertion (i), we apply Lemma~\ref{Lmm:main} to the integrand $f(\x, y) = L(\inner{\w}{\x} - y)$ with $L(z) = \max_{j \leq J} \{ a_j z + b_j \}$ to obtain
	\begin{align*}
	\Sup _ {\QQ \in \Wball} \!\!\!\!  \EE ^ \QQ \left[ \ell (\inner{\w}{\x}, y) \right] 
	&=  \displaystyle \Inf_{\lambda\geq 0} ~ \lambda\eps+\frac{1}{N}\sum_{i=1}^N  \Sup_{(\x,y) \in \Xi} ~ L (\inner{\w}{\x} - y) - \lambda \| (\x,y) - (\widehat \x_i,\widehat y_i)\| \\
	&= \Inf_{\lambda\geq 0} ~ \lambda\eps+\frac{1}{N}\sum_{i=1}^N \max_{j\leq J} \Sup_{(\x,y) \in \Xi} ~ a_j (\inner{\w}{\x} - y) + b_j - \lambda \| (\x,y) - (\widehat \x_i,\widehat y_i) \| \\
	&=  \optimize{ \Inf_{\lambda,\p_{ij}, u_{ij}} & \displaystyle \lambda\eps+\sum_{i=1}^N  \max_{j\leq J} S_\Xi(a_j \w - \p_{ij}, -a_j - u_{ij}) + \inner{\p_{ij}} {\widehat{\x}_i} + u_{ij} \widehat{y}_i + b_j  \\
		\text{s.t.} & \| (\p_{ij}, u_{ij}) \|_* \leq \lambda \qquad i \in [N], j \in [J],}
	\end{align*}			
	where the last equality follows from Lemma~\ref{Lmm:linear}. The claim now follows by introducing auxiliary epigraphical variables $s_i$ for the max-terms in the objective function and by including $\w$ as a decision variable. 
	
	To prove assertion (ii), we apply Lemma~\ref{Lmm:main} to the integrand $f(\x, y) = L(\inner{\w}{\x} - y)$, where $L$ is a Lipschitz continuous convex loss function. Thus we find
	\begin{align*}
	\begin{array}{rl}
	\Sup _ {\QQ \in \Wball} \EE ^ \QQ  \left[ \loss (\x, y) \right] 
	&= \displaystyle  \Inf_{\lambda\geq 0 } ~\lambda\eps+\sum_{i=1}^N \Sup_{\x,y} ~ L (\inner{\w}{\x} - y) - \lambda \| (\x,y) - (\widehat \x_i,\widehat y_i) \| \\
	&= \optimize{
		\Inf_{\lambda} & \displaystyle \lambda \eps + \frac{1}{N} \SumN L (\inner{\w}{\widehat \x_i} - \widehat y_i) \\
		\text{s.t.} & \lip |\theta| \cdot \| (\w, -1) \|_* \leq \lambda,} 
	\end{array}
	\end{align*}
	where the last equality uses Lemma~\ref{Lmm:convex}. Next, we eliminate $\lambda$ and include $\w$ as a decision variable.		 
\end{proof}

\begin{proof} [Proof of Corollary~\ref{Crl:hr}]
	Note that the Huber loss function $L(z)$ coincides with the inf-convolution of $\frac{1}{2} z^2$ and $\delta |z|$ and can thus be expressed as $L(z) = \min_{z_1} ~ \frac{1}{2} z_1^2 + \delta |z - z_1|$. Moreover, the Lipschitz modulus of the Huber loss function is $\delta$. The rest of the proof follows from Theorem~\ref{Thrm:Regression}(ii).
\end{proof}	

\begin{proof} [Proof of Corollary~\ref{Crl:svr}]
	Notice that the $\epsilon$-insensitive loss function is a piecewise linear function with $J=3$ pieces, see Section~\ref{sec:stat-learn}.
	By strong conic duality, the support function of $\Xi = \{ (\x, y) \in \RR^{n+1} : \bm C_1 \x + \bm c_2 y \preceq_{\CC} \bm d \}$ can be re-expressed as
	$$ S_\Xi(\z_1, z_2) 
	= \Sup_{\x, y} \left\{ \inner{\z_1}{\x} + z_2y: \bm C_1 \x + \bm c_2 y \preceq_{\CC} \bm d \right\}
	= \Inf_{ \q \in \CC^*} \left\{ \inner{\q}{\bm d}: \bm C_1^\top \q = \z_1,~  \bm c_2^\top \q = z_2\right\}.$$
	Strong duality holds because $\Xi$ admits a Slater point. The rest of proof follows from Theorem~\ref{Thrm:Regression}(i).
\end{proof}

\begin{proof} [Proof of Corollary~\ref{Crl:qr}]
	The pinball loss function is a piecewise linear function with $J=2$ pieces, see Section~\ref{sec:stat-learn}. The rest of proof follows from the dual representation of the support function $S_\Xi(\z_1, z_2)$, which is known from the proof of Corollary~\ref{Crl:svr}, and from Theorem~\ref{Thrm:Regression}(i).
\end{proof}
The proof of Theorem~\ref{Thrm:worst-case:regression} is based on the following preparatory lemma.
\begin{Lmm2}
	\label{Lmm:worst}
	If $\ZZ \subseteq \RR^d$ is a non-empty convex closed set, $\widehat {\bm \zeta} \in\ZZ$, $\Beta\in \RR^d$ and $\alpha, \gamma \geq 0$, then we have
	\begin{align*}
	\Inf_{\p} ~ \alpha S_\ZZ (\Beta - \p) + \alpha \inner{\p}{\widehat {\bm \zeta}} + \gamma \| \p \|_* 
	&= \optimize{
		\Sup_{\| \q \| \leq \gamma} & \alpha \inner{\Beta}{\widehat {\bm \zeta}} + \inner{\Beta}{\q} \\
		\text{s.t.} & \widehat {\bm \zeta} + \q/\alpha \in \ZZ.}
	\end{align*}
\end{Lmm2}

\begin{proof}
	If $\alpha=0$, then the optimal values of both opimization problems vanish due to our conventions of extended arithmetic, and thus the claim trivially holds. If $\alpha>0$, however, we have
	\begin{align*}
	\Inf_ {\p} ~ \alpha \, S_\ZZ (\Beta - \p) + \alpha \inner{\p}{\widehat {\bm \zeta}} + \gamma \| \p \|_* 
	&= \Inf_ {\p} ~ \sup_{\| \q \| \leq \gamma} ~ \alpha S_\ZZ(\Beta - \p) + \alpha \inner{\p}{\widehat {\bm \zeta}} + \inner{\p}{\q} \\
	&= \sup_{\| \q \| \leq \gamma} ~ \Inf_ {\p} ~ \alpha S_\ZZ(\Beta - \p) +  \inner{\p}{\alpha \widehat {\bm \zeta} + \q} \\
	&= \sup_{\| \q \| \leq \gamma} ~ \Inf_ {\z} ~ \alpha S_\ZZ(\z) +  \inner{\Beta - \z}{\alpha \widehat {\bm \zeta} + \q} \\
	&= \sup_{\| \q \| \leq \gamma} ~ \alpha \inner{\Beta}{\widehat {\bm \zeta}} + \inner{\Beta}{\q} - \alpha \big( \sup_{\z} \inner{\z}{\widehat {\bm \zeta} + \q / \alpha} - S_\ZZ(\z) \big) \\
	& = \sup_{\| \q \| \leq \gamma} \alpha \inner{\Beta}{\widehat {\bm \zeta}} + \inner{\Beta}{\q} - \alpha \delta_\ZZ (\widehat {\bm \zeta} + \q / \alpha),
	\end{align*}
	where the first equality follows from the definition of the dual norm, the second equality exploits \citep[Proposition~5.5.4]{bertsekas2009convex}, and the last equality holds because, for any convex closed set, the indicator function is the conjugate of the support function.
\end{proof}

\begin{proof} [Proof of Theorem~\ref{Thrm:worst-case:regression}]
	We first prove assertion (i). By Theorem~\ref{Thrm:Regression}(i), the worst-case expectation problem~\eqref{wc-expectation-regression} constitutes a restriction of \eqref{tractable:regression:linear} where $\w$ is fixed, and thus it coincides with the minimax problem
	\begin{align*} 
	\Inf_{ \substack{\lambda,s_i \\ \p_{ij}, u_{ij}} } ~ \Sup_{\alpha_{ij} \geq 0, \gamma_{ij} \geq 0}& ~ \Obj + \Sum_{i=1}^N \Sum_{j=1}^J \alpha_{ij} \big( S_\Xi(-a_j \w - \p_{ij}, a_j - u_{ij}) + b_j + \inner{\p_{ij}} {\widehat{\x}_i} + u_{ij} \widehat{y}_i - s_i \big) \\
	&+ \Sum_{i=1}^N \Sum_{j=1}^J \gamma_{ij} \big( \| (\p_{ij}, u_{ij}) \|_* - \lambda \big).
	\end{align*}
	The minimization and the maximization may be interchanged by strong duality, which holds because the convex program \eqref{tractable:regression:linear} satisfies Slater's constraint qualification for every fixed $\w$ \citep[Proposition~5.3.1]{bertsekas2009convex}. Indeed, note that $S_\Xi$ is proper, convex and lower semi-continuous and appears in constraints that are always satisfiable because they involve a free decision variable. Thus, the above minimax problem is equivalent to			
	\begin{align*}
	\optimize{
		\Sup_{ \alpha_{ij}, \gamma_{ij}} & \Inf_{ \p_{ij}, u_{ij}} ~ \displaystyle \sum_{i=1}^N \sum_{j=1}^J \alpha_{ij} \big( S_\Xi(a_j \w - \p_{ij}, a_j - u_{ij}) + b_j + \inner{\p_{ij}} {\widehat{\x}_i} + u_{ij} \widehat{y}_i \big) + \sum_{i=1}^N \sum_{j=1}^J \gamma_{ij}  \| \p_{ij} \|_* \\
		\mathrm{s.t.} & \displaystyle \sum_{i=1}^N \sum_{j=1}^J \gamma_{ij} = \eps  \\
		& \displaystyle \sum_{j=1}^J \alpha_{ij} = \frac{1}{N} \hspace{4.4em} i \in [N] \\
		& \alpha_{ij}, \gamma_{ij} \geq 0 \hspace{5em} i \in [N], j \in [J].}
	\end{align*}
	By Lemma~\ref{Lmm:worst}, which applies because $(\widehat{\x}_i, \widehat y_i)\in\Xi$ for all $i\leq N$, the above dual problem simplifies to
	\begin{align*}
	\optimize{
		\Sup_{ \substack{\alpha_{ij}, \gamma_{ij} \\ \q_{ij}, v_{ij}} } &  \displaystyle \sum_{i=1}^N \sum_{j=1}^J \alpha_{ij} \big( a_j (\inner{\w}{\widehat \x_i} - \widehat y_i) + b_j \big) + a_j (\inner{\w}{\q_{ij}} - v_{ij}) & \\
		\mathrm{s.t.} & \displaystyle \sum_{i=1}^N \sum_{j=1}^J \gamma_{ij} = \eps &  \\
		& \displaystyle \sum_{j=1}^J \alpha_{ij} = \frac{1}{N} & \hspace{-3cm} i \in [N] \\
		& \| (\q_{ij}, v_{ij}) \| \leq \gamma_{ij} & \hspace{-3cm} i \in [N], j \in [J] \\
		& \left( \widehat{\x}_i - \q_{ij}/\alpha_{ij}, \widehat y_i - v_{ij} / \alpha_{ij} \right) \in \Xi & \hspace{-3cm} i \in [N], j \in [J] \\
		& \alpha_{ij}, \gamma_{ij} \geq 0 & \hspace{-3cm} i \in [N], j \in [J].}
	\end{align*}
	Problem \eqref{worst:regression:linear} is now obtained by eliminating the variables $\gamma_{ij}$ and by substituting $\alpha_{ij}$, $\q_{ij}$, and $v_{ij}$ with $\alpha_{ij}/N$, $\q_{ij}/N$, and $v_{ij}/N$, respectively.
	
	As for assertion (ii), we first show that the discrete distribution $\QQ_\gamma$ belongs to the Wasserstein ball $\Wball$ for all $\gamma \in (0,1]$. Indeed, the Wasserstein distance between $\QQ_\gamma$ and $\Pem$ amounts to
	$$ d(\QQ_\gamma, \Pem) \leq \frac{\gamma}{N} \left\| \left(\widehat \x_1 + \frac{\eps N}{\gamma} \x\opt, \widehat y_1 + \frac{\eps N}{\gamma} y\opt\right) - (\widehat \x_1, \widehat y_1) \right\| = \eps \| (\x\opt, y\opt) \| \leq \eps, $$
	where the first inequality holds because the Wasserstein distance coincides with the optimal mass transportation cost, and the last inequality holds because the norm of $(\x\opt, y\opt)$ is at most $1$ by construction. Thus, $\QQ_\gamma\in \Wball$ for all $\gamma \in (0,1]$. Denoting the optimal value of~\eqref{wc-expectation-regression}  by $J\opt(\w)$ and using $\loss(\x,y)$ as a shorthand for $L(\inner{\w}{\x}-y)$, we find
	$$ \begin{array}{rl}
	J\opt(\w) &\geq \EE ^ {\QQ_\gamma} \left[ \loss(\x,y) \right] \\
	&= \frac{1}{N} \SumN \loss(\widehat \x_i, \widehat y_i) - \frac{\gamma}{N} \loss(\widehat \x_1, \widehat y_1) + \frac{\gamma}{N} \loss(\widehat \x_1 + \frac{\eps N}{\gamma} \x\opt, \widehat y_1 + \frac{\eps N}{\gamma} y\opt) \\
	&\geq \frac{1}{N} \SumN \loss(\widehat \x_i, \widehat y_i) - \frac{\gamma}{N} \loss(\widehat \x_1, \widehat y_1) + \frac{\gamma}{N} \big(\inner{(\x,y)}{(\widehat \x_1 + \frac{\eps N}{\gamma} \x\opt, \widehat y_1 + \frac{\eps N}{\gamma} y\opt)}\\ &\qquad\quad - \loss^*(\x,y) \big) \quad~ \forall (\x,y) \in \RR^{n+1},
	\end{array} $$
	where the last estimate follows from Fenchel's inequality. Setting $(\x,y) = \lip (\w, -1)$ we thus have
	$$ \begin{array}{rl}
	J\opt(\w) &\geq \lim\limits_{\gamma \rightarrow 0^+} \frac{1}{N} \SumN \loss(\widehat \x_i, \widehat y_i) - \frac{\gamma}{N} \loss(\widehat \x_1, \widehat y_1) + \frac{\gamma}{N} \lip (\inner{\w}{\widehat \x_1} - \widehat y_1) +\eps \lip \| (\w, -1) \|_* \\
	& \qquad \quad - \frac{\gamma}{N} \loss^*(\lip (\w, -1)) \\
	&= \frac{1}{N} \SumN \loss(\widehat \x_i, \widehat y_i) + \eps \lip \| (\w, -1) \|_* = J\opt(\w),
	\end{array} $$
	where the equality follows from Theorem~\ref{Thrm:Regression}(ii). The above reasoning implies that $\lim_{\gamma \rightarrow 0^+}\EE ^ {\QQ_\gamma} [ \loss(\x,y) ]=J\opt(\w)$, and thus the claim follows.
\end{proof}

\begin{proof} [Proof of Theorem~\ref{Thrm:RobConnection}]
	Assume first that the loss function $L$ is convex piecewise linear, that is, $L(z)= \max_{j\in J} \{ a_j z + b_j \}$. As $\Xi \in \RR^{n+1}$, Theorem~\ref{Thrm:worst-case:regression}(i) implies that the worst-case expectation \eqref{wc-expectation-regression} is given by
	\begin{align*}
	& \optimize{
		\Sup_{ \alpha_{ij}, \q_{ij}, v_{ij} } & \displaystyle \frac{1}{N} \sum_{i=1}^N\sum_{j=1}^J \alpha_{ij} \Big[ a_j (\inner{\w}{\widehat \x_i} - \widehat y_i) + b_j \Big] + a_j (\inner{\w}{\q_{ij}} - v_{ij}) & \\
		\mathrm{s.t.} &\displaystyle \frac{1}{N} \sum_{i=1}^N\sum_{j=1}^J \| (\q_{ij},v_{ij}) \| \leq \eps &  \\
		& \displaystyle \sum_{j=1}^J \alpha_{ij} = 1 &\hspace{-5cm} i \in [N] \\
		& \alpha_{ij} \geq 0 & \hspace{-5cm} i \in [N], j \in [J]} \\
	\geq & \optimize{
		\Sup_{\alpha_{ij}, \Delta \x_{ij}, \Delta y_{ij}} & \displaystyle \frac{1}{N} \sum_{i=1}^N\sum_{j=1}^J \alpha_{ij} \left[ a_j (\inner{\w}{\widehat \x_i + \Delta \x_{ij}} - \widehat y_i - \Delta y_{ij} )+ b_j \right] & \\
		\mathrm{s.t.} & \displaystyle \frac{1}{N} \sum_{i=1}^N\sum_{j=1}^J \alpha_{ij} \| (\Delta \x_{ij}, \Delta y_{ij}) \| \leq \eps & \\
		& \displaystyle \sum_{j=1}^J \alpha_{ij} = 1 & \hspace{-5cm} i \in [N] \\
		& \displaystyle \alpha_{ij} \geq 0 & \hspace{-5cm} i \in [N], j \in [J] } \\
	\geq & \optimize{
		\Sup_{\alpha_{ij}, \Delta \x_{i}, \Delta y_{i}} & \displaystyle \frac{1}{N} \sum_{i=1}^N\sum_{j=1}^J \alpha_{ij} \left[ a_j (\inner{\w}{\widehat \x_i + \Delta \x_{i}} - \widehat y_i - \Delta y_{i}) + b_j \right] & \\
		\mathrm{s.t.} & \displaystyle \frac{1}{N} \sum_{i=1}^N\sum_{j=1}^J \alpha_{ij} \| (\Delta \x_{i}, \Delta y_{i}) \| \leq \eps & \\
		& \displaystyle \sum_{j=1}^J \alpha_{ij} = 1 & \hspace{-5cm}i \in [N] \\
		& \alpha_{ij} \geq 0 & \hspace{-5cm} i \in [N], j \in [J] .}
	\end{align*}
	The first inequality holds because for any feasible solution $\{\alpha_{ij}, \Delta \x_{ij}, \Delta y_{ij}\}$ to the second optimization problem, the solution $\{\alpha_{ij}, \q_{ij}, v_{ij}\}$ with $\q_{ij} = \alpha_{ij} \Delta \x_{ij}$ and $v_{ij} = \alpha_{ij} \Delta y_{ij}$ 
	is feasible in the first problem and attains the same objective value (conversely, note that the first problem admits feasible solutions with $\alpha_{ij}=0$ and $\q_{ij}\neq \bm 0$ that have no counterpart in the second problem). The second inequality follows from the restriction that $\Delta \x_{ij}$ and $\Delta y_{ij}$ must be independent of $j$. It is easy to verify that the last optimization problem in the above expression is equivalent to~\eqref{wc-regression-loss} because $(\alpha_{i1},\ldots,\alpha_{iJ})$ ranges over a simplex for every $i\leq N$, and thus~\eqref{wc-regression-loss} provides a lower bound on~\eqref{wc-expectation-regression}. 
	
	Suppose now that Assumption~\ref{Asmp:sep:reg} holds, and note that the worst-case loss~\eqref{wc-regression-loss} can be expressed as
	\begin{align*}
	& \optimize{
		\Sup_{\Delta \x_{i}, \Delta y_{i}} & \displaystyle \frac{1}{N} \sum_{i=1}^N \max_{j\leq J} \left[ a_j (\inner{\w}{\widehat \x_i + \Delta \x_{i}} - \widehat y_i - \Delta y_{i}) + b_j \right] & \\
		\mathrm{s.t.} & \displaystyle \frac{1}{N} \sum_{i=1}^N \| (\Delta \x_{i}, \Delta y_{i}) \| \leq \eps & } \\
	\geq & \optimize{
		\Sup_{\Delta \x, \Delta y} & \displaystyle \frac{1}{N} \sum_{i \neq k} \max_{j\leq J} \left[ a_j (\inner{\w}{\widehat \x_i} - \widehat y_i) + b_j \right] + \frac{1}{N} \max_{j\leq J} \left[ a_j (\inner{\w}{\widehat \x_k + \Delta \x} - \widehat y_k - \Delta y) + b_j \right]  \\
		\mathrm{s.t.} & \displaystyle \frac{1}{N} \| (\Delta \x, \Delta y) \| \leq \eps. }
	\end{align*}
	The above inequality follows from setting $\Delta \x_i = 0$ and $\Delta y_i = 0$ for all $i \neq k$, where $(\widehat \x_k , \widehat y_k)$ is a training sample satisfying $ |L'(\inner{\w}{\widehat \x_k}- \widehat y_k)| = \lip$, which exists due to Assumption~\ref{Asmp:sep:reg}. The last expression equals
	\begin{align*}
	=& \frac{1}{N} \sum_{i \neq k} \max_{j\leq J} \left[ a_j (\inner{\w}{\widehat \x_i} - \widehat y_i) + b_j \right] + \max_{j\leq J} \optimize{
		\Sup_{\Delta \x, \Delta y} & \displaystyle \frac{1}{N}  \left[ a_j (\inner{\w}{\widehat \x_k} - \widehat y_k) + b_j \right] + \frac{1}{N}
		\left[ a_j (\inner{\w}{\Delta \x} - \Delta y) \right] & \\
		\mathrm{s.t.} & \displaystyle \frac{1}{N} \| (\Delta \x, \Delta y) \| \leq \eps & } \\
	=& \frac{1}{N} \sum_{i \neq k} \max_{j\leq J} \left[ a_j (\inner{\w}{\widehat \x_i} - \widehat y_i) + b_j \right] + \frac{1}{N} \max_{j\leq J} \Big[ a_j (\inner{\w}{\widehat \x_k} - \widehat y_k) + b_j + \eps N \|(\w,-1)\|_* |a_j| \Big] \\
	=&  \frac{1}{N} \sum_{i \neq k} \max_{j\leq J} \left[ a_j (\inner{\w}{\widehat \x_i} - \widehat y_i) + b_j \right] + \frac{1}{N} \max_{j\leq J} \left[ a_j (\inner{\w}{\widehat \x_k} - \widehat y_k) + b_j \right] + \max_{j \leq J} \eps \|(\w,-1)\|_* |a_j|  \\
	=& \frac{1}{N} \sum_{i=1}^N L( \inner{\w}{\widehat \x_i} - \widehat y_i) + \max_{j\leq J} \eps \| (\w,-1) \|_* |a_j|,
	\end{align*}
	where the penultimate equality holds because $\inner{\w}{\widehat \x_k} - \widehat y_k$ resides on the steepest linear piece of the loss function $L$ by virtue of Assumption~\ref{Asmp:sep:reg}. The claim then follows from Theorem~\ref{Thrm:Regression}(ii) because $\lip=\max_{j\leq J} |a_j|$. Note that generic convex Lipschitz continuous loss functions can be uniformly approximated as closely as desired with convex piecewise linear functions. Thus, the above argument extends directly to generic convex Lipschitz continuous loss functions. Details are omitted for brevity.
\end{proof}

\begin{proof} [Proof of Theorem~\ref{Thrm:Classification}]
	To prove assertion (i), we apply Lemma~\ref{Lmm:main} to problem~\eqref{distrob:wass} with the transportation distance~\eqref{metric}, where $\xxi=(\x,y)$ and $\Xi = \XX \times \{-1,+1\}$. Thus, we obtain
	\begin{align*}
	\Sup _ {\QQ \in \Wball} \EE ^ \QQ  \left[ \ell (\inner{\w}{\x}, y) \right] &= 
	\optimize{
		\Inf_{\lambda,s_i} & \displaystyle \Obj & \\
		\text{s.t.}
		& \Sup_{(\x, y) \in \Xi} ~L (y\inner{\w}{\x}) - \lambda \big( \| \x - \widehat \x_i \| + \frac{\kappa}{2} | y - \widehat y_i | \big) \leq s_i & i \in [N] \\
		&\lambda \geq 0.} \\
	&= \optimize{
		\Inf_{\lambda,s_i} & \displaystyle \Obj & \\
		\text{s.t.}
		& \Sup_{\x \in \XX} ~L (\widehat y_i \inner{\w}{\x}) - \lambda \| \x - \widehat{\x}_i \| \leq s_i & i \in [N] \\
		& \Sup_{\x \in \XX} ~L (-\widehat y_i \inner{\w}{\x}) - \lambda \| \x - \widehat{\x}_i \| - \kappa \lambda \leq s_i & i \in [N] \\
		&\lambda \geq 0} \\
	&= \optimize{
		\Inf_{\lambda,s_i} & \displaystyle \Obj & \\
		\text{s.t.}
		& \Sup_{\x \in \XX} ~ a_j \widehat y_i \inner{\w}{\x} + b_j - \lambda \| \x - \widehat{\x}_i \| \leq s_i & i \in [N], j \in [J] \\
		& \Sup_{\x \in \XX} ~ - a_j \widehat y_i \inner{\w}{\x} + b_j - \lambda \| \x - \widehat{\x}_i \| - \kappa \lambda \leq s_i & i \in [N], j \in [J] \\
		&\lambda \geq 0,}	
	\end{align*}
	where the second equality holds because, for every $i$, $y$ can be either equal to $\widehat y_i$ or to $- \widehat y_i$. Reformulating the constraints using Lemma~\ref{Lmm:linear} and including $\w$ as a decision variable then yields \eqref{tractable:classification:linear}.
	
	When $\XX = \RR^n$ and $L$ is Lipschitz continuous, we can use similar arguments as above to prove that
	\begin{align*}
	\Sup _ {\QQ \in \Wball} \EE ^ \QQ  \left[ \ell (\inner{\w}{\x}, y) \right] 
	&= \optimize{
		\Inf_{\lambda,s_i} & \displaystyle \Obj & \\
		\text{s.t.}
		& \Sup_{\x \in \RR^{n}} ~ L (\widehat y_i \inner{\w}{\x}) - \lambda \| \x - \widehat \x_i \| \leq s_i & i \in [N] \\
		& \Sup_{\x \in \RR^{n}} ~ L (-\widehat y_i \inner{\w}{\x}) - \lambda \| \x - \widehat \x_i \| - \kappa \lambda \leq s_i & i \in [N] \\
		&\lambda \geq 0.}
	\end{align*}
	Applying Lemma~\ref{Lmm:convex} to the subordinate maximization problems in the constraints yields
	\begin{align*}
	\Sup _ {\QQ \in \Wball} \EE ^ \QQ  \left[ \ell (\inner{\w}{\x}, y) \right] 
	&= \optimize{
		\Inf_{\lambda,s_i} & \displaystyle \Obj & \\
		\text{s.t.} & L (\widehat y_i \inner{\w}{\widehat \x_i}) \leq s_i & i \in [N], j \in [J] \\
		& L (-\widehat y_i \inner{\w}{\widehat \x_i}) - \kappa \lambda \leq s_i & i \in [N], j \in [J] \\
		& \Sup_{\theta \in \Theta} ~ |\theta| \cdot \| \w \|_* \leq \lambda.}
	\end{align*}
	Thus, assertion (ii) follows by recalling that $\lip=\sup_{\theta} \{ |\theta|: L^*(\theta) < \infty \}=\sup_{ \theta\in\Theta}|\theta|$ and by including $\w$ as a decision variable.
\end{proof}

\begin{proof} [Proof of Corollary~\ref{Crl:svm}]
	Notice that the hinge loss function is piecewise linear with $J=2$ pieces, see Section~\ref{sec:stat-learn}. Moreover, by strong conic duality the support function of $\XX$ can be re-expressed as
	$$ S_\XX(\z) 
	= \Sup_{\x} \left\{ \inner{\z}{\x} : \bm C \x \preceq_{\CC} \bm d \right\} 
	= \Inf_{\q \in \CC^*} \left\{ \inner{\q}{\bm d} : \bm C^\top \q = \z \right\}.$$
	Strong duality holds because $\XX$ admits a Slater point.
	The proof thus follows from Theorem~\ref{Thrm:Classification}(i).
\end{proof}

\begin{proof} [Proof of Corollary~\ref{Crl:hrc}]
	The smooth hinge loss $L(z)$ coincides with the inf-convolution of $\frac{1}{2} z^2$ and $\max \{0, 1-z \}$ and can thus be expressed as $L(z) = \min_{z_1} \frac{1}{2} z_1^2 + \max \{ 0, 1 - z - z_1 \}$. Moreover, the Lipschitz modulus of the smooth hinge loss function is~1. The proof thus follows from Theorem~\ref{Thrm:Classification}(ii).
\end{proof}

\begin{proof} [Proof of Corollary~\ref{Crl:lr}]
	Notice that the logloss function is convex and has Lipschitz modulus~1, see Section~\ref{sec:stat-learn}. The rest of proof follows from Theorem~\ref{Thrm:Classification}(ii). For more details see \citep{shafieezadeh2015distributionally}.
\end{proof}

\begin{proof} [Proof of Theorem~\ref{Thrm:worst-case:classification}]
	We first prove assertion (i). By Theorem~\ref{Thrm:Classification}(i), the worst-case expectation problem~\eqref{wc-expectation-classification} constitutes a restriction of \eqref{tractable:classification:linear} where $\w$ is fixed, and thus it is equivalent to the minimax problem
	\begin{align*} 
	\Inf_{ \substack{\lambda,s_i \\ \p^+_{ij}, \p^-_{ij} } } ~ \Sup_{\substack{\alpha^+_{ij} \geq 0, \gamma^+_{ij} \geq 0 \\ \alpha^-_{ij} \geq 0, \gamma^-_{ij} \geq 0}}& ~ \Obj + \sum_{i = 1}^{N}\sum_{j = 1}^{J} \alpha^+_{ij} \big( S_\XX(a_j \widehat y_i \w - \p^+_{ij}) + b_j + \inner{\p^+_{ij}}{\widehat \x_i} - s_i \big) \\
	&+ \sum_{i = 1}^{N}\sum_{j = 1}^{J} \alpha^-_{ij} \big( S_\XX(-a_j \widehat y_i \w - \p^-_{ij}) + b_j + \inner{\p^-_{ij}}{\widehat \x_i} - \kappa \lambda - s_i \big) \\
	&+ \sum_{i = 1}^{N}\sum_{j = 1}^{J} \gamma^+_{ij} \big( \| \p^+_{ij} \|_* - \lambda \big) + \sum_{i = 1}^{N}\sum_{j = 1}^{J} \gamma^-_{ij} \big( \| \p^-_{ij} \|_* - \lambda \big).
	\end{align*}
	The minimization and the maximization may be interchanged by strong Lagrangian duality, which holds because the convex program~\eqref{tractable:classification:linear} satisfies Slater's constraint qualification for any fixed $\w$ \citep[Proposition~5.3.1]{bertsekas2009convex}. Indeed, note that $S_\XX$ is proper, convex and lower semi-continuous and appears in constraints that are always satisfiable because they involve a free decision variable. Thus, problem~\eqref{wc-expectation-classification} is equivalent~to
	\begin{align*}
	\optimize{
		\Sup_{\substack{\alpha^+_{ij}, \gamma^+_{ij} \\ \alpha^-_{ij}, \gamma^-_{ij} }} & \Inf_{ \p^+_{ij}, \p^-_{ij}} ~ \displaystyle \sum_{i = 1}^{N}\sum_{j = 1}^{J} \alpha^+_{ij} \big( S_\XX(a_j \widehat y_i \w - \p^+_{ij}) + b_j + \inner{\p^+_{ij}}{\widehat \x_i} \big) + \gamma^+_{ij}  \| \p^+_{ij} \|_* \\
		& \hspace{26pt} \displaystyle + \sum_{i = 1}^{N}\sum_{j = 1}^{J} \alpha^-_{ij} \big( S_\XX(-a_j \widehat y_i \w - \p^-_{ij}) + b_j + \inner{\p^-_{ij}}{\widehat \x_i} \big)  + \gamma^-_{ij}  \| \p^-_{ij} \|_* & \\
		\mathrm{s.t.} & \displaystyle \sum_{i = 1}^{N}\sum_{j = 1}^{J} \gamma^+_{ij} + \gamma^-_{ij} + \kappa \alpha^-_{ij} = \eps &  \\
		& \displaystyle \sum_{j=1}^J \alpha^+_{ij} + \alpha^-_{ij} = \frac{1}{N} & \hspace{-5cm} i \in [N] \\
		& \alpha^+_{ij}, \alpha^-_{ij}, \gamma^+_{ij}, \gamma^-_{ij} \geq 0 & \hspace{-5cm}i \in [N], j \in [J].}
	\end{align*}
	By Lemma~\ref{Lmm:worst}, the above dual problem simplifies to
	\begin{align*}
	\optimize{
		\Sup_{ \substack{\alpha^+_{ij}, \gamma^+_{ij}, \q^+_{ij} \\ \alpha^-_{ij}, \gamma^-_{ij}, \q^-_{ij} } } &  \displaystyle \sum_{i = 1}^{N}\sum_{j = 1}^{J} (\alpha^+_{ij}-\alpha^-_{ij}) a_j \widehat y_i \inner{\w}{\widehat \x_i} + a_j \widehat y_i \inner{\w}{\q^+_{ij} - \q^-_{ij}} +\sum_{j=1}^J b_j & \\
		\mathrm{s.t.} & \displaystyle \sum_{i = 1}^{N}\sum_{j = 1}^{J} \gamma^+_{ij} + \gamma^-_{ij} + \kappa \alpha^-_{ij} = \eps &  \\
		& \displaystyle \sum_{j=1}^J \alpha^+_{ij} + \alpha^-_{ij} = \frac{1}{N} & \hspace{-3cm} i \in [N] \\
		& \| \q^+_{ij} \| \leq \gamma^+_{ij} ,\quad \| \q^-_{ij} \| \leq \gamma^-_{ij} & \hspace{-3cm} i \in [N], j \in [J] \\
		& \widehat{\x}_i + \q^+_{ij}/\alpha^+_{ij} \in \XX,\quad \widehat{\x}_i + \q^-_{ij}/\alpha^-_{ij} \in \XX & \hspace{-3cm} i \in [N], j \in [J] \\
		& \alpha^+_{ij}, \alpha^-_{ij}, \gamma^+_{ij}, \gamma^-_{ij} \geq 0 &\hspace{-3cm} i \in [N], j \in [J]. }
	\end{align*}
	Problem~\eqref{worst:classification:linear} is now obtained by eliminating $\gamma^+_{ij}$ and $\gamma^-_{ij}$ and by substituting $\alpha^+_{ij}$, $\alpha^-_{ij}$, $\q^+_{ij}$, and $\q^-_{ij}$ with $\alpha^+_{ij}/N$, $\alpha^-_{ij}/N$, $\q^+_{ij}/N$, and $\q^-_{ij}/N$, respectively.

	As for assertion (ii), we use $L_i^+$ and $L_i^-$ to abbreviate $L (\widehat{y}_i \inner{\w}{\widehat{\x}_i})$ and $L (- \widehat{y}_i \inner{\w}{\widehat{\x}_i})$, respectively.
	By Theorem~\ref{Thrm:Classification}(ii), we have
	\begin{align*}
	\sup_ {\QQ \in \Wball} \EE ^ \QQ [ L(y\inner{\w}{\x}) ] 
	&= \optimize{
		\Inf_{\w,\lambda,s_i} & \displaystyle \Obj & \\
		\text{s.t.} & L_i^+ \leq s_i & i \in [N] \\
		& L_i^- - \kappa \lambda \leq s_i & i \in [N] \\
		& \lip \| (\w) \|_* \leq \lambda & } \\
	&= \optimize{
		\Sup_{\alpha_i^+, \alpha_i^-, \theta} & \displaystyle \lip \| (\w) \|_* \theta + \SumN \alpha_i^+ L_i^+ + \alpha_i^- L_i^- & \\
		\text{s.t.} & \alpha_i^+ + \alpha_i^- = \frac{1}{N} & \hspace{-3cm} i \in [N] \\
		& \theta + \kappa \SumN \alpha_i^- = \eps & \\
		& \alpha_i^+ \geq 0, \, \alpha_i^- \geq 0 & \hspace{-3cm} i \in [N] \\
		& \theta \geq 0, &} 
	\end{align*}
	where the second equality follows from strong linear programming duality, which holds because the primal problem is manifestly feasible.
	Eliminating the first constraint and replacing $\alpha_i^-$ with $\alpha_i / N$ and $\alpha_i^+$ with $(1-\alpha_i) / N$ allows us to  reformulate the dual linear program as
	\begin{align*}
	\optimize{
		\Sup_{\alpha_i, \theta} & \displaystyle \lip (\w) \|_* \theta + \frac{1}{N}\SumN (1-\alpha_i) L_i^+ + \alpha_i L_i^- & \\
		\text{s.t.} & \displaystyle \theta +  \frac{\kappa}{N} \SumN \alpha_i = \eps & \\
		& 0 \leq \alpha_i \leq 1 & \hspace{-3cm} i \in [N] \\
		& \theta \geq 0. &}
	\end{align*}
	Thus, the worst-case expectation~\eqref{wc-expectation-classification} coincides with the optimal value of \eqref{worst:classification:convex} for $\gamma = 0$.
	Next let $(\alpha_i\opt(\gamma),\theta\opt(\gamma))$ be an optimal solution of \eqref{worst:classification:convex} for $\gamma \geq 0$, and define $\QQ_\gamma$ as in the theorem statement. Note that \eqref{worst:classification:convex} is infeasible for $\gamma>\eps$. Moreover, note that the atoms of $\QQ_\gamma$ have non-negative probabilities if $\eta(\gamma)\in[0,1]$, which holds whenever $\gamma\in[0,1]$. We thus focus on parameter values $\gamma\in[0,\min\{\eps,1\}]$. By construction, the Wasserstein distance between $\QQ_\gamma$ and the empirical distribution satisfies
	$$ \begin{array}{rl}
	d(\QQ_\gamma, \Pem) &\leq \frac{\kappa}{N} \SumN \alpha_i\opt(\gamma) - \frac{\eta(\gamma) \kappa}{N} \alpha\opt_1(\gamma) + \frac{\eta(\gamma)}{N} d\big( (\widehat \x_1 + \frac{\theta\opt(\gamma) N}{\eta(\gamma)} \x\opt, \widehat y_1) - (\widehat \x_1, \widehat y_1) \big) \\ 
	&\leq \eps - \gamma - \theta\opt(\gamma) + \theta\opt(\gamma) \| \x\opt \| \leq \eps,
	\end{array} $$
	where the first inequality holds because the Wasserstein distance is defined as the minimum cost of moving~$\QQ_\gamma$ to $\Pem$, the second inequality follows from the feasibility of $(\alpha_i\opt(\gamma),\theta\opt(\gamma))$ in \eqref{worst:classification:convex} and the non-negativity of $\eta(\gamma)$, and the last inequality holds because $\| \x\opt \| \leq 1$, $\theta\opt(\gamma)\geq 0$ and $\gamma\geq 0$. Thus, $\QQ_\gamma\in \Wball$ for all $\gamma \in [0,\min\{\eps,1\}]$. Denoting the optimal value of~\eqref{wc-expectation-classification}  by $J\opt(\w)$, we find 
	$$ \begin{array}{rl}
	J\opt(\w) &\geq \EE ^ {\QQ_\gamma} \left[ L(y \inner{\w}{\x}) \right] \\
	&= \frac{1}{N} \big( \SumN (1-\alpha_i\opt(\gamma)) L_i^+ + \alpha_i\opt(\gamma) L_i^- \big) - \frac{\eta(\gamma)}{N} \big( (1-\alpha_1\opt(\gamma)) L_1^+ + \alpha_1\opt(\gamma) L_1^- \big) \\
	& \quad + \frac{\eta(\gamma)}{N} L \big( \widehat y_1 \inner{\w}{\widehat \x_1 + \frac{\theta\opt(\gamma) N}{\eta(\gamma)} \x\opt} \big) \\
	&\geq \frac{1}{N} \big( \SumN (1-\alpha_i\opt(\gamma)) L_i^+ + \alpha_i\opt(\gamma) L_i^- \big) - \frac{\eta(\gamma)}{N} \big( (1-\alpha_1\opt(\gamma)) L_1^+ + \alpha_1\opt(\gamma) L_1^- \big) \\
	& \quad + \frac{\eta(\gamma)}{N} \Big( \inner{\x}{\widehat \x_1 + \frac{\theta^*(\gamma)N}{\eta(\gamma)} \x\opt} - L^*(\widehat y_1 \inner{\w}{\x}) \Big) \qquad \forall \x \in \RR^{n},
	\end{array} $$
	where the last estimate follows from Fenchel's inequality. Setting $\x = \lip \w$ and driving $\gamma$ to zero yields
	$$ \begin{array}{rl}
	J\opt(\w) &\geq \lim\limits_{\gamma \rightarrow 0^+} \frac{1}{N} \Big( \SumN (1-\alpha_i\opt(\gamma)) L_i^+ + \alpha_i\opt(\gamma) L_i^- \Big) - \frac{\eta(\gamma)}{N} \big( (1-\alpha_1\opt(\gamma)) L_1^+ + \alpha_1\opt(\gamma) L_1^- \big) \\
	& \quad + \frac{\eta(\gamma)}{N} \big( \lip \inner{\w}{\widehat \x_1} - L^*(\widehat y_1 \lip \inner{\w}{\w}) \big) + \lip \| \w \|_* \cdot \theta\opt(\gamma) \\
	&=  \lim\limits_{\gamma \rightarrow 0^+} \frac{1}{N} \big( \SumN (1-\alpha_i\opt(\gamma)) L_i^+ + \alpha_i\opt(\gamma) L_i^- \big) + \lip \| \w \|_* \cdot \theta\opt(\gamma) = J\opt(\w),
	\end{array} $$
	where the first equality follows from the observation that $\eta(\gamma) \in[0,\gamma]$, which implies that $\eta(\gamma)$ converges to zero as $\gamma$ tends to zero. The second equality holds because the optimal value of~\eqref{worst:classification:convex} is concave and non-increasing and---{\em a fortiori}---continuous in $\gamma\in [0,\min\{\eps,1\}]$ and because $J\opt(\w)$ coincides with the optimal value of~\eqref{worst:classification:convex} when $\gamma=0$. The above reasoning implies that $\lim_{\gamma \rightarrow 0^+}\EE ^ {\QQ_\gamma}[ L(y \inner{\w}{\x})]=J\opt(\w)$, and thus the claim follows.
\end{proof}	

\begin{proof} [Proof of Theorem~\ref{Thrm:RobConnection:class}]
	Assume first that the loss function $L$ is convex piecewise linear, that is, $L(z)= \max_{j\in J} \{ a_j z + b_j \}$. As $\XX = \RR^n$ and $\kappa=\infty$, Theorem~\ref{Thrm:worst-case:classification}(i) implies that~\eqref{wc-expectation-classification} can be expressed as
	\begin{align*}
	& \optimize{
		\Sup_{ \alpha^+_{ij}, \q^+_{ij} } & \displaystyle \frac{1}{N} \sum_{i = 1}^{N}\sum_{j = 1}^{J} \alpha^+_{ij} a_j \widehat y_i \inner{\w}{\widehat \x_i} + a_j \widehat y_i \inner{\w}{\q^+_{ij}} +\sum_{j=1}^J b_j & \\		
		\mathrm{s.t.} & \displaystyle \sum_{i = 1}^{N}\sum_{j = 1}^{J} \| \q^+_{ij} \| \leq N \eps &  \\
		& \displaystyle \sum_{j=1}^J \alpha^+_{ij} = 1 & \hspace{-6cm} i \in [N] \\
		& \alpha^+_{ij} \geq 0 & \hspace{-6cm} i \in [N], j \in [J]} \\
	\geq & \optimize{
		\Sup_{ \alpha_{ij}, \Delta \x_{ij} } & \displaystyle \frac{1}{N} \sum_{i = 1}^{N}\sum_{j = 1}^{J} \alpha_{ij} a_j \widehat y_i \inner{\w}{\widehat \x_i + \Delta \x_{ij}} +\sum_{j=1}^J b_j & \\		
		\mathrm{s.t.} & \displaystyle \sum_{i = 1}^{N}\sum_{j = 1}^{J} \alpha_{ij} \| \Delta \x_{ij} \| \leq N \eps &  \\
		& \displaystyle \sum_{j=1}^J \alpha_{ij} = 1 & \hspace{-5cm} i \in [N] \\
		& \alpha_{ij} \geq 0 & \hspace{-5cm} i \in [N], j \in [J]} \\
	\geq & \optimize{
		\Sup_{ \alpha_{ij}, \Delta \x_{i} } & \displaystyle \frac{1}{N} \sum_{i = 1}^{N}\sum_{j = 1}^{J} \alpha_{ij} a_j \widehat y_i \inner{\w}{\widehat \x_i + \Delta \x_{i}} +\sum_{j=1}^J b_j & \\		
		\mathrm{s.t.} & \displaystyle \sum_{i = 1}^{N}\sum_{j = 1}^{J} \alpha_{ij} \| \Delta \x_{i} \| \leq N \eps &  \\
		& \displaystyle \sum_{j=1}^J \alpha_{ij} = 1 & \hspace{-5cm} i \in [N] \\
		& \alpha_{ij} \geq 0 & \hspace{-5cm} i \in [N], j \in [J].} 
	\end{align*}
	The first optimization problem constitutes a special case of~\eqref{worst:classification:linear}. Indeed, as $\kappa = \infty$, the first constraint in~\eqref{worst:classification:linear} implies that $\alpha_{ij}^-=0$, which in turn implies via the fourth constraint and our conventions of extended arithmetics that $\q_{ij}^-=\bm 0$. The first inequality in the above expression holds because for any feasible solution $\{\alpha_{ij}, \Delta \x_{ij}\}$ to the second problem, the solution $\{\alpha^+_{ij}, \q^+_{ij}\}$ with $\q^+_{ij} = \alpha^+_{ij} \Delta \x_{ij}$ and $\alpha_{ij}=\alpha^+_{ij}$ for all $i\leq N$ and $j\leq J$ is feasible in the first problem and attains the same objective value. The second inequality in the above expression follows from the restriction that $\Delta \x_{ij}$ must be independent of $j$. It is easy to verify that the last optimization problem is equivalent to~\eqref{wc-classification-loss} because $(\alpha_{i1},\ldots,\alpha_{iJ})$ ranges over a simplex for every $i\leq N$, and thus~\eqref{wc-classification-loss} provides a lower bound on~\eqref{wc-expectation-classification}. 
	
	Suppose now that Assumption~\ref{Asmp:sep:class} holds, and note that \eqref{wc-classification-loss} can be expressed as
	\begin{align*}
	& \optimize{
		\Sup_{\Delta \x_{i}, \Delta y_{i}} & \displaystyle \frac{1}{N} \sum_{i=1}^N \max_{j\leq J} \left[ a_j \widehat y_i \inner{\w}{\widehat \x_i + \Delta \x_{i}} + b_j \right] & \\
		\mathrm{s.t.} & \displaystyle \frac{1}{N} \sum_{i=1}^N \| \Delta \x_{i} \| \leq \eps & } \\
	\geq & \optimize{
		\Sup_{\Delta \x, \Delta y} & \displaystyle \frac{1}{N} \sum_{i \neq k} \max_{j\leq J} \left[ a_j \widehat y_i \inner{\w}{\widehat \x_i} + b_j \right] + \frac{1}{N} \max_{j\leq J} \left[ a_j \widehat y_k \inner{\w}{\widehat \x_k + \Delta \x} + b_j \right] & \\
		\mathrm{s.t.} & \displaystyle \frac{1}{N} \| \Delta \x \| \leq \eps.& }
	\end{align*}
	The above inequality follows from setting $\Delta \x_i = 0$ and $\Delta y_i = 0$ for all $i \neq k$, where $(\widehat \x_k , \widehat y_k)$ is a training sample satisfying $ |L'(\widehat y_k\inner{\w}{\widehat \x_k})| = \lip$, which exists due to Assumption~\ref{Asmp:sep:class}. The last expression equals
	\begin{align*}
	& \frac{1}{N} \sum_{i \neq k} \max_{j\leq J} \left[ a_j \widehat y_i \inner{\w}{\widehat \x_i} + b_j \right] + \max_{j\leq J} \optimize{
		\Sup_{\Delta \x, \Delta y} & \displaystyle \frac{1}{N}  \left[ a_j \widehat y_k \inner{\w}{\widehat \x_k} + b_j \right] + \frac{1}{N}
		\left[ a_j \widehat y_k \inner{\w}{\Delta \x} \right] & \\
		\mathrm{s.t.} & \displaystyle \frac{1}{N} \| (\Delta \x, \Delta y) \| \leq \eps & } \\
	=& \frac{1}{N} \sum_{i \neq k} \max_{j\leq J} \left[ a_j \widehat y_i \inner{\w}{\widehat \x_i} + b_j \right] + \frac{1}{N} \max_{j\leq J} \Big[ a_j \widehat y_i \inner{\w}{\widehat \x_k} + b_j + \eps N \| \w \|_* |a_j| \Big] \\
	=&  \frac{1}{N} \sum_{i \neq k} \max_{j\leq J} \left[ a_j \widehat y_i \inner{\w}{\widehat \x_i} + b_j \right] + \frac{1}{N} \max_{j\leq J} \left[ a_j \widehat y_k \inner{\w}{\widehat \x_k} + b_j \right] + \max_{j \leq J} \eps \| \w \|_* |a_j|  \\
	=& \frac{1}{N} \sum_{i=1}^N L( \widehat y_i \inner{\w}{\widehat \x_i} ) + \max_{j\leq J} \eps \| \w \|_* |a_j|,
	\end{align*}
	where the penultimate equality holds because $\widehat y_k \inner{\w}{\widehat \x_k}$ resides on the steepest linear piece of the loss function $L$ by virtue of Assumption~\ref{Asmp:sep:class}. The claim then follows from Theorem~\ref{Thrm:Classification}(ii) because $\lip=\max_{j\leq J} |a_j|$. Note that generic convex Lipschitz continuous loss functions can be uniformly approximated as closely as desired with convex piecewise linear functions. Thus, the above arguments extend directly to generic convex Lipschitz continuous loss functions. Details are omitted for brevity.
\end{proof}

\begin{proof} [Proof of Theorem~\ref{Thrm:relation}]
	By the definition of the feature map $\Phi$ corresponding to the kernel $k$, we have
	\begin{equation}
	\label{eq:calmness}
	\begin{aligned}
	\|\Phi(\x_1)-\Phi(\x_2)\|_\HH& = \sqrt{\inner{\Phi(\x_1)}{\Phi(\x_1)}_\HH-2\inner{\Phi(\x_1)}{\Phi(\x_2)}_\HH+\inner{\Phi(\x_2)}{\Phi(\x_2)}_\HH}\\
	&= \sqrt{k(\x_1, \x_1) - 2 k(\x_1, \x_2) + k(\x_2, \x_2)}~ \leq~ g(\| \x_1 - \x_2 \|_2)
	\end{aligned}
	\end{equation}
	for all $\x_1,\x_2\in\XX$, where the inequality follows from Assumption~\ref{ass:calmness}. As for assertion~(i), we may use similar argument as in the proof of Lemma~\ref{Lmm:main} to reformulate the worst-case expectation in~\eqref{distrob:wass:hilbert} as
	\begin{align}
	\Sup_{\QQ \in \BB_{\eps } (\Pem)} \EE ^ \QQ \left[ \ell(h(\x),y) \right]
	& = \optimize{
		\Sup_ {\QQ ^ i} & \displaystyle \frac{1}{N} \SumN \int_{\XX\times\YY}
		\ell(h(\x), y) \QQ^i(\diff \x,\diff y) \\
		\text{s.t. } & \displaystyle \frac{1}{N} \SumN \int_{\XX\times \YY} d \big( (\x, y), (\widehat \x_i, \widehat y_i) \big) \QQ^i (\diff \x,\diff y) \leq \eps \\
		& \int_{\XX\times\YY} \QQ ^ i (\diff \x,\diff y) = 1 \qquad i \in [N]} \nonumber\\
	& = \optimize{
		\Sup_ {\QQ ^ i} & \displaystyle \frac{1}{N} \SumN \int_{\XX\times\YY}
		\ell(h(\x), y) \QQ^i(\diff \x,\diff y) \\
		\text{s.t. } & \displaystyle g \left( \frac{1}{N} \SumN \int_{\XX\times\YY} \sqrt{2} d \big( (\x, y), (\widehat \x_i, \widehat y_i) \big) \QQ^i (\diff \x,\diff y) \right) \leq g(\sqrt{2} \eps) \\
		& \int_{\XX\times\YY} \QQ ^ i (\diff \x,\diff y) = 1 \qquad i \in [N]} \nonumber \\
	& \leq \optimize{
		\Sup_ {\QQ ^ i} & \displaystyle \frac{1}{N} \SumN \int_{\XX\times\YY}
		\ell(h(\x), y) \QQ^i(\diff \x,\diff y) \\
		\text{s.t. } & \displaystyle \frac{1}{N} \SumN \int_{\XX\times \YY} g \left( \sqrt{2} d \big( (\x, y), (\widehat \x_i, \widehat y_i) \big) \right) \QQ^i (\diff \x,\diff y) \leq g(\sqrt{2} \eps) \\
		& \int_{\XX\times\YY} \QQ ^ i (\diff \x,\diff y) = 1 \qquad i \in [N],} 
	\label{eq:kernel-wce}
	\end{align}
	where the second equality holds because $g$ is strictly monotonically increasing, and the inequality follows from Jensen inequality, which applies because $g$ is concave. By the definition of the transportation metric on $\HH\times\YY$ for regression problems, we then find
	\begin{equation}
	\label{eq:g-estimate}
	\begin{aligned}
	g \Big( \sqrt{2} d \big( (\x, y), (\widehat \x_i, \widehat y_i) \big) \Big)
	& = g\Big( \sqrt{ 2 \| \x - \widehat \x_i \|_2^2 + 2 (y - \widehat y_i)^2 } \Big) \\
	& \geq g( \| \x - \widehat \x_i \|_2 + |y - \widehat y_i| ) \\
	& \geq g( \| \x - \widehat \x_i \|_2) + |y - \widehat y_i| \\
	& \geq \| \Phi(\x) - \Phi(\widehat \x_i) \|_\HH + |y - \widehat y_i| \\
	& \geq \sqrt{ \| \Phi(\x) - \Phi(\widehat \x_i) \|^2_\HH + (y - \widehat y_i)^2 } 
	= d_\HH \big( (\Phi(\x), y), (\Phi(\widehat \x_i), \widehat y_i) \big),
	\end{aligned}
	\end{equation}
	where the first inequality holds because $2a^2+2b^2 \geq (a+b)^2$ for all $a,b\geq 0$ and because $g$ is strictly monotonically increasing, the second inequality exploits the assumption that $g'(z)\geq 1$ for all $z\geq 0$, the third inequality follows form \eqref{eq:calmness}, and the last equality holds because $a^2+b^2 \leq (a+b)^2$ for all $a,b\geq 0$. Substituting the above estimate into~\eqref{eq:kernel-wce} and using the reproducing property $h(\x)=\inner{h}{\Phi(\x)}_\HH$ yields
	\begin{align*}
	\Sup_{\QQ \in \BB_{\eps } (\Pem)} \EE ^ \QQ \left[ \ell(h(\x),y) \right]
	& \leq \optimize{
		\Sup_ {\QQ ^ i} & \displaystyle \frac{1}{N} \SumN \int_{\XX\times\YY}
		\ell(\inner{h}{\Phi(\x)}_\HH,y) \QQ^i(\diff \x,\diff y) \\
		\text{s.t. } & \displaystyle \frac{1}{N} \SumN \int_{\XX\times\YY} d_\HH \Big( \big(\Phi(\x), y), \big(\Phi(\widehat \x_i), \widehat y_i) \Big) \QQ^i (\diff \x,\diff y) \leq g(\sqrt{2} \eps) \\
		& \int_{\XX\times\YY} \QQ ^ i (\diff \x,\diff y) = 1 \qquad i \in [N]} \\
	& \leq \optimize{
		\Sup_ {\QQ ^ i} & \displaystyle \frac{1}{N} \SumN \int_{\HH \times \YY}
		\ell(\x_\HH, y) \QQ^i(\diff \x_\HH, \diff y) \\
		\text{s.t. } & \displaystyle \frac{1}{N} \SumN \int_{\HH \times \YY} d_\HH \Big( \big(\x_\HH, y), \big(\Phi(\widehat \x_i), \widehat y_i) \Big) \QQ^i (\diff \x_\HH, \diff y) \leq g(\sqrt{2} \eps) \\
		& \int_{\HH \times \YY} \QQ ^ i (\diff \x_\HH, \diff y) = 1 \qquad i \in [N]} \\
	& = \Sup_{\QQ \in \BB_{g(\sqrt{2} \eps)} (\Pem^{\HH})} \EE ^ \QQ \left[ \ell(\x_\HH,y) \right],
	\end{align*}
	where the second inequality follows from relaxing the implicit condition that the random variable $\x_\HH$ must be supported on $\{\Phi(\x):\x\in\XX\}\subseteq\HH$. This completes the proof of assertion~(i) for regression problems. 
	
	The proof of assertion~(ii) parallels that of assertion~(i) with obvious modifications. Due to the different transportation metric for classification problems, however, the estimate~\eqref{eq:g-estimate} changes to
	\begin{align*}
	g \left( d \big( (\x, y), (\widehat \x_i, \widehat y_i) \big) \right)
	& = g \left( \| \x - \widehat \x_i \|_2 + \kappa |y - \widehat y_i| \right) \\
	& \geq g \left( \| \x - \widehat \x_i \|_2 \right) + \kappa |y - \widehat y_i|  \\
	& \geq \| \Phi(\x) - \Phi(\widehat \x_i) \|_\HH + \kappa |y - \widehat y_i| 
	= d_\HH \big( (\Phi(\x), y), (\Phi(\widehat \x_i), \widehat y_i) \big),
	\end{align*}
	where the first inequality exploits the assumption that $g'(z)\geq 1$ for all $z\geq 0$, and the second inequality follows form \eqref{eq:calmness}. Further details are omitted for brevity.
\end{proof}

\begin{proof} [Proof of Theorem~\ref{Thrm:rep}]
	The proof follows immediately from \citep[Theorem~4.2]{scholkopf2001learning}, which applies to loss functions representable as a sum of an empirical loss depending on $(\widehat \x_i,\widehat y_i, h(\widehat \x_i))$, $i\leq N$, and a regularization term that is strictly monotonically increasing in $\| h \|_{\HH}$. However, the additive separability is not needed for the proof. We remark that the optimal solution of~\eqref{rep_th} is unique if $f$ is striclty increasing in $\|h\|_\HH$. If $f$ is only non-decreasing in $\|h\|_\HH$, on the other hand, uniqueness may be lost. Details are omitted for brevity.
\end{proof}

\begin{proof} [Proof of Theorem~\ref{Thrm:Ker_Reg}]
	Using similar arguments as in the proof of Theorem~\ref{Thrm:Regression}(ii) and observing that any Hilbert norm is self-dual, one can show that
	$$ \Inf_ {h \in \HH} \Sup_{\QQ \in \BB_{\eps} (\Pem^{\HH})} \EE ^ \QQ \left[  L (\inner{h}{\x_\HH}_\HH - y) \right] = \min_{h \in \HH} \frac{1}{N} \SumN L (h(\widehat \x_i) - \widehat y_i) + \eps \lip \sqrt{\| h \|_{\HH}^2 + 1}.$$
	By the representer theorem, which applies because the objective function of the above optimization problem is non-decreasing in $\|h\|_\HH$, we may restrict the feasible set from $\HH$ to the subset of all linearly parametrized hypotheses of the form $h(\x) = \sum_{j=1}^N \beta_j k(\x,\widehat \x_j)$ for some $\bm \beta\in\RR^N$ without sacrificing optimality. The claim then follows by observing that $h(\widehat \x_i)= \sum_{j=1}^N \KK_{ij}\beta_j$ and $\|h\|_2^2=\inner{\bm \beta}{\KK\bm \beta}$.
\end{proof}

\begin{proof} [Proof of Theorem~\ref{Thrm:Ker_Class}]
	Using similar arguments as in the proof of Theorem~\ref{Thrm:Classification}(ii) and observing that any Hilbert norm is self-dual, one can show that
	\begin{align*}
	\Inf_ {h\in \HH} \Sup_{\QQ \in \BB_{\eps} (\Pem^{\HH})} \EE ^ \QQ \left[  L (y \inner{h}{\x_\HH}_\HH) \right] &= \optimize{\Min_{h,\lambda,s_i} & \displaystyle \Obj  & \\
		\mathrm{s.t.} & L (\widehat y_i h(\widehat \x_i)) \leq s_i & i \in [N] \\
		& L (-\widehat y_i h(\widehat \x_i)) - \kappa \lambda \leq s_i & i \in [N] \\
		&  \lip \| h \|_{\HH} \leq \lambda,}
	\end{align*}
	see \citep[Theorem~1]{gao2016distributionally} for a full proof. By the representer theorem, which applies because the loss function
	\[
	f( (\widehat \x_1,\widehat y_1, h(\widehat \x_1)), \ldots, (\widehat \x_N,\widehat y_N, h(\widehat \x_N)), \| h \|_{\HH})= \optimize{\Min_{\lambda} & \lambda \eps + \frac{1}{N} \SumN \max \{ L (\widehat y_i h(\widehat \x_i)), L (-\widehat y_i h(\widehat \x_i)) - \kappa \lambda \} \\
		\mathrm{s.t.} & \lambda \geq \lip \| h \|_\HH} 
	\]
	is non-decreasing in $\|h\|_\HH$, we may restrict attention to all linearly parametrized hypotheses of the form $h(\x) = \sum_{j=1}^N \beta_j k(\x,\widehat \x_j)$ for some $\bm \beta\in\RR^N$ without sacrificing optimality. Thus, the claim follows.
\end{proof}
\change{
\begin{proof}[Proof of Theorem~\ref{theorem:lips}]
	The worst-case expectation for regression problems satisfies 
	\begin{align*}
	&~\sup_{\QQ \in \Wball} \EE^\QQ \left[ \ell(h(\bm x; \bm W_{[M]}), y) \right] \\
	&= \inf_{\lambda \geq 0} ~ \lambda \rho + \frac{1}{N} \sum_{i=1}^N \sup_{\substack{\x \in \mathbb{X} \\ y \in \mathbb{Y}}} ~ L(h(\bm x; \bm W_{[M]}) - y) - \lambda ( \| \x - \widehat \x_i \| + \kappa | y - \widehat y_i | ) \\
	&\leq \inf_{\lambda \geq 0} ~ \lambda \rho + \frac{1}{N} \sum_{i=1}^N \sup_{\substack{\x \in \mathbb{X} \\ y \in \mathbb{Y}}} ~ L(h(\widehat \x_i; \bm W_{[M]}) - \widehat y_i) + \lip ( | h(\x; \bm W_{[M]}) - y - h(\widehat \x_i; \bm W_{[M]}) + \widehat y_i | ) \\
	& \hspace{11em} - \lambda ( \| \x - \widehat \x_i \| + \kappa | y - \widehat y_i | )   \\
	&\leq \inf_{\lambda \geq 0} ~ \lambda \rho + \frac{1}{N} \sum_{i=1}^N \sup_{\substack{\x \in \mathbb{X} \\ y \in \mathbb{Y}}} ~ L(h(\widehat \x_i; \bm W_{[M]}) - \widehat y_i) + \lip \, {\rm lip} (h) \| \x - \widehat \x_i \| + \lip | y - \widehat y_i | \\
	& \hspace{11em} - \lambda ( \| \x - \widehat \x_i \| + \kappa | y - \widehat y_i | ),
	\end{align*}
	where the equality holds by Lemma~\ref{Lmm:main}, and the first and the second inequalities follow from the Lipschitz continuity of the loss function $L$ and the hypothesis $h$. Note that composition of functions satisfies the inequality $\mathrm{lip}(f \circ g) \leq \mathrm{lip}(f) \, \mathrm{lip}(g)$; see for example~\citep[Exercise~9.8]{rockafellar2009variational}. Thus, the Lipschitz modulus of the hypothesis $h$ admits the upper bound
	\begin{align}
	\label{eq:upper}
	{\rm lip} (h) \leq \prod_{m=1}^m {\rm lip} (\sigma_m(\bm W_m \bm z_m)) \leq \prod_{m=1}^m {\rm lip} (\sigma_m) \, \| \bm W_m \|. 
	\end{align}
	Replacing $ {\rm lip} (h) $ by its upper bound in~\eqref{eq:upper} and setting 
	$$\lambda = \lip \max \left\{  \prod_{k=1}^M {\rm lip}\big(\phi(\cdot;\bm W_m)\big), \frac{1}{\kappa} \right\},$$
	will complete the proof for regression problems.
	In a similar way, the worst-case expectation for classification problems can be handled as follows: 
	\begin{align*}
	&~\sup_{\QQ \in \Wball} \!\! \EE^\QQ \left[ \ell(h(\bm x; \bm W_{[M]}), y) \right] \\
	&= \inf_{\lambda \geq 0} ~ \lambda \rho + \frac{1}{N} \sum_{i=1}^N \sup_{\substack{\x \in \mathbb{X} \\ y \in \mathbb{Y}}} ~ L(y h(\bm x; \bm W_{[M]})) - \lambda ( \| \x - \widehat \x_i \| + \kappa \mathds{1}_{ \{ y \neq \widehat y_i \} } ) \\
	&\leq \inf_{\lambda \geq 0} ~ \lambda \rho + \frac{1}{N} \sum_{i=1}^N \sup_{\substack{\x \in \mathbb{X} \\ y \in \mathbb{Y}}} ~ L(\widehat y_i h(\widehat \x_i; \bm W_{[M]})) - \lambda ( \| \x - \widehat \x_i \| + \kappa \mathds{1}_{ \{ y \neq \widehat y_i \} } ) + \lip ( | yh(\x) - \widehat y_i h(\widehat \x_i) | )  \\
	&\leq \inf_{\lambda \geq 0} ~ \lambda \rho + \frac{1}{N} \sum_{i=1}^N \sup_{\substack{\x \in \mathbb{X} \\ y \in \mathbb{Y}}} ~ L(\widehat y_i h(\widehat \x_i; \bm W_{[M]})) - \lambda ( \| \x - \widehat \x_i \| + \kappa \mathds{1}_{ \{ y \neq \widehat y_i \} } ) + \lip \, {\rm lip}(h) \, \| \x - \widehat \x_i \| \mathds{1}_{ \{ y = \widehat y_i \} } \\
	&\hspace{11em} + \lip \, | h(\x) + h(\widehat x_i) | \, \mathds{1}_{ \{ y \neq \widehat y_i \} } \\
	&\leq \inf_{\lambda \geq 0} ~ \lambda \rho + \frac{1}{N} \sum_{i=1}^N \sup_{\substack{\x \in \mathbb{X} \\ y \in \mathbb{Y}}} ~ L(\widehat y_i h(\widehat \x_i; \bm W_{[M]})) - \lambda ( \| \x - \widehat \x_i \| + \kappa \mathds{1}_{ \{ y \neq \widehat y_i \} } ) \\
	& \hspace{11em} + \lip \, \max \{ \frac{2 \sup_{\x \in \mathbb X, h \in \mathbb{H}} |h(\x)|}{\kappa}, {\rm lip}(h) \} \, ( \| \x - \widehat \x_i \| + \kappa \mathds{1}_{ \{ y \neq \widehat y_i \} } ).
	\end{align*}
	Using the upper bound in~\eqref{eq:upper} and setting 
	$$\lambda = \lip \max \left\{  \prod_{k=1}^M {\rm lip}\big(\phi(\cdot;\bm W_m)\big), \frac{\max\{1 , 2 \sup_{\x \in \mathbb{X}, h \in \mathbb{H}} |h(\x)|\}}{\kappa} \right\}$$
	will completes the proof for classification problems.
\end{proof}

\begin{proof} [Proof of Corollary~\ref{corollary:safe}]
	Let $ \bar \sigma = \max_{m \in [M]} \mathrm{lip}(\sigma_m)$.
	By Theorem~\ref{theorem:lips} and setting $\kappa = \infty$, the worst-case expectation for both classification and regression problems can be upper bounded by 
	\begin{align*}
	\sup_{\QQ \in \Wball} \EE^\QQ \left[ \ell(h(\bm x; \bm W_{1,\;\hdots\;, M}), y) \right] &\leq \frac{1}{N} \sum_{i=1}^N \sup_{\substack{\x \in \mathbb{X} \\ y \in \mathbb{Y}}} ~ \ell(h(\widehat \x_i; \bm W_{1,\;\hdots\;, M}), \widehat y_i) + \rho \, \bar \sigma \, \mathrm{lip}(\ell) \prod_{k=1}^M \| \bm W_k \| \\
	&\leq \frac{1}{N} \sum_{i=1}^N \sup_{\substack{\x \in \mathbb{X} \\ y \in \mathbb{Y}}} ~ \ell(h(\widehat \x_i; \bm W_{1,\;\hdots\;, M}), \widehat y_i) + \rho \, \bar \sigma \, \mathrm{lip}(\ell) \left( \sum_{k=1}^M \frac{\| \bm W_k \|}{M} \right)^M,
	\end{align*}
	where the last inequality follows from the arithmetic-geometric means inequality.
	By \citep[Theorem~1]{everett1963generalized}, if $\bm W_{[M]}\opt$ is a minimizer of the optimization problem
	$$\min_{\bm W_{[M]}}~\frac{1}{N} \sum_{i=1}^N \sup_{\substack{\x \in \mathbb{X} \\ y \in \mathbb{Y}}} ~ \ell(h(\widehat \x_i; \bm W_{1,\;\hdots\;, M}), \widehat y_i) + \rho \, \bar \sigma \, \mathrm{lip}(\ell) \left( \sum_{k=1}^M \frac{\| \bm W_k \|}{M} \right)^M,$$
	then the same $\bm W_{[M]}\opt$ also minimizes the constrained optimization problem
	\begin{align*}
	\optimize{
		\displaystyle \inf_{\bm W_{[M]}} & \displaystyle \frac{1}{N} \sum_{i=1}^N \ell(h(\widehat \x_i; \bm W_{[M]}), \widehat y_i) \\
		\mathrm{s.t.} & \displaystyle \left( \sum_{k=1}^M \frac{\| \bm W_k \|}{M} \right)^M \leq \theta^M,
	}
	\end{align*}
	for $\theta = \sum_{i=1}^M \| \bm W_k \opt \| / M.$ Notice that the constraint in the above optimization problem can be simplified to $\sum_{i=1}^M  \| \bm W_k \opt \| \leq M \theta$. Hence, there exists a Lagrange multiplier $\overline \rho$ for the simplified constraint $\sum_{i=1}^M  \| \bm W_k \opt \| \leq M \theta$ such that $\bm W_{[M]}\opt$ is an optimizer of the penalized problem
	\begin{align*}
	\inf_{\bm W_{[M]}} \frac{1}{N} \sum_{i=1}^N \ell(h(\widehat \x_i; \bm W_{[M]}), \widehat y_i) + \overline \rho \sum_{k=1}^M \| \bm W_k \|.
	\end{align*}
	This completes the proof.
\end{proof}
}
\subsection{Proofs of Section~\ref{sec:Probabilistic}}
The proof of Theorem~\ref{Thrm:improved} relies on the following preparatory lemma, which basically asserts that the sample average of a linearly growing function of $\xxi$ has sub-Gaussian tails.

\begin{Lmm2}[Sub-Gaussian tails]
	\label{Lmm:concentration}
	If Assumption~\ref{Asmp:light} holds, then there exist constants $c_3\ge1$ and $c_4>0$ that depend only on the light tail constants $a$ and $A$ of $\PP$ such that  
	\begin{equation*}
	\PP^{N}\Big\{ \big| \EE^\PP[f(\xxi)] - \EE^{\Pem}[f(\xxi)] \big| \ge \delta \Big\} \le c_3\exp(-c_4N\delta^2) 
	\end{equation*}
	for any $N\in\mathbb N$, $\delta\in [0,1]$ and function $f:\Xi\rightarrow\mathbb R$ with $|f(\xxi)-f(\xxi')| \le d(\xxi, \xxi')$ for all $\xxi \in\Xi$ and some reference point $\xxi' \in \Xi$.
\end{Lmm2}

\begin{proof} 
	Assume that $f:\Xi\rightarrow\mathbb R$ is a linear growth function with $|f(\xxi)-f(\xxi')| \le d(\xxi, \xxi')$ for all $\xxi \in\Xi$ and some reference point $\xxi' \in \Xi$. Set $\xi_f=f(\xxi)$ and $\xi'_f=f(\xxi')$. Thus, the distribution of the scalar random variable $\xi_f$ is given by the pushforward measure $f_*(\PP)$ of $\PP$. By construction, we have
	\begin{align*}
	\EE^{f_*(\PP)}\left[\exp\big(|\xi_f-\xi_f'|^a\big) \right] = \EE^{\PP}\left[\exp\big(|f(\xxi) - f(\xxi')|^a\big) \right] \le \EE^{\PP}\left[\exp\big(d(\xxi,\xxi')^a\big) \right] \le A ,
	\end{align*}
	where the first inequality follows form the growth condition of $f$, while the second inequality holds because~$\PP$ satisfies Assumption~\ref{Asmp:light}. Hence, the distribution $f_*(\PP)$ satisfies Assumption~\ref{Asmp:light} for $n=0$ when distances on~$\mathbb R$ are measured by the absolute value, and it inherits the light-tail constants $a$ and $A$ from~$\PP$. By using Theorem~\ref{thm:concentration} for $n=0$, we may thus conclude that there exist constants $c_3, c_4>0$ with
	\begin{align*}
	\PP^N \Big\{ W(f_*(\PP), f_*(\Pem)) \ge \delta \Big \} \le c_3 \exp\big({-c_4 N\delta^2}\big)\quad \forall \delta \in[0, 1],
	\end{align*}
	where $f_*(\Pem)$ represents the empirical distribution of $\xi_f$, which coincides with the pushforward measure of~$\Pem$ under~$f$. By slight abuse of notation, $W$ stands here for the Wasserstein distance between {\em univariate} distributions, where the absolute value is used as the ground metric. Note that the above univariate measure concentration result holds for any linear growth function $f$ with asymptotic growth rate $\le 1$. We emphasize that $c_3 \ge 1$ because otherwise the above estimate would fail to hold for $\delta=0$. 
	
	By construction of the Wasserstein distance in Definition~\ref{Def:Wass}, we have $W(f_*(\PP), f_*(\Pem))<\delta$ if and only if the scalar random variables $\xi_f$ and $\xi_f'$ admit a joint distribution $\Pi$ with $\EE^\Pi [|\xi_f - \xi_f'|] <\delta$ under which $\xi_f$ and $\xi_f'$ have marginals $f_*(\PP)$ and $f_*(\Pem)$, respectively. The inequality $W(f_*(\PP), f_*(\Pem))<\delta$ thus implies	
	\begin{align*}
	\left| \EE^{\PP}[f(\xxi)]  - \EE^{\Pem}[f(\xxi')] \right| = \left| \EE^{f_*(\PP)}[\xi_f] - \EE^{f_*(\Pem)}[\xi'_f] \right| \le \left| \EE^\Pi[\xi_f - \xi_f'] \right|
	= \EE^\Pi \left[ \left|  \xi_f - \xi_f'\right| \right] <\delta.
	\end{align*}
	By contraposition, we then obtain the implication
	\[
	\left| \EE^{\PP}[f(\xxi)]  - \EE^{\Pem}[f(\xxi')] \right| \ge \delta \qquad \implies\qquad W(f_*(\PP), f_*(\Pem))\ge \delta,
	\]
	which leads to the desired inequality
	\[
	\PP^{N}\Big\{ \big| \EE^\PP[f(\xxi)] - \EE^{\Pem}[f(\xxi)] \big| \ge \delta \Big\}\le  \PP^{N}\Big\{ W(f_*(\PP), f_*(\Pem))\ge \delta \Big\} \le c_3 \exp\big({-c_4 N\delta^2}\big).
	\]
	This inequality holds for all $\delta\in[0,1]$ and $N\in\mathbb N$, irrespective of the linear growth function~$f$. 
\end{proof}

\begin{Rmk2}[Hoeffding's inequality]
	\label{rem:subgaussian}
	If it is known that $\PP\{\underline f\le f(\xxi)\le \overline f\}=1$, then Lemma~\ref{Lmm:concentration} reduces to Hoeffding's inequality \citep[Theorem~2.8]{boucheron2013concentration}, in which case we may set $c_3=2$ and $c_4=2/(\overline f-\underline f)^2$. 
\end{Rmk2}

\begin{proof} [Proof of Theorem~\ref{Thrm:improved}]
	To avoid cumbersome case distinctions, we prove the theorem only in the case when~\eqref{distrob:wass} is a classification problem. Thus, we assume that $\Xi=\RR^n\times\{-1,1\}$ and that the transportation cost is of the form~\eqref{metric}, where $\|\cdot\|$ denotes a norm on the input space $\RR^n$. The proof for regression problems is similar and only requires minor modifications. It will be omitted for brevity. 
	
	From Theorem~\ref{Thrm:Classification}(ii) we know that
	\begin{align*}
	\sup_{\QQ \in \mathbb B_{\eps}(\Pem)}  \EE^\QQ \left[ {\ell(\inner{\w}{\x}, y)} \right] & = \EE^{\Pem} \left[ {\ell(\inner{\w}{\x}, y)} \right] + \eps \lip \Omega(\w) \quad \forall \w\in\WW,
	\end{align*}
	where $\Omega(\w)=\| \w \|_*$ can be viewed as a regularization function. For every $\eps\ge\eps'_N(\eta)$ we thus have
	\begin{align}
	\notag
	& \Prob^N \left\{ \mathds{E}^\mathds{P} \left[ \ell(\inner{\w}{\x}, y) \right] \leq \sup_{\QQ \in \mathbb B_{\eps}(\Pem)} \mathds{E}^\mathds{\QQ} \left[ \ell(\inner{\w}{\x}, y) \right]~\forall \w \in \WW \right\} \\  \notag
	\ge  ~& \Prob^N \left\{ 0\le \min_{\w \in \WW}  \EE^{\Pem} \left[ {\ell(\inner{\w}{\x}, y)} \right] + \eps'_N(\eta) \lip \Omega(\w)  - \mathds{E}^\mathds{P} \left[ \ell(\inner{\w}{\x}, y) \right] \right\} \\ 	
	=  ~& 1 - \Prob^N \left\{  \min_{{\w} \in \WW} \EE^{\Pem} \left[ {\ell(\inner{\w}{\x}, y)} \right] + \eps'_N(\eta) \lip \Omega(\w)  - \EE^\PP \left[ {\ell(\inner{\w}{\x}, y)} \right] < 0 \right\}.
	\label{eq:probability}
	\end{align}
	Observe that $\ell(\inner{\w}{\x}, y)$ is Lipschitz continuous in $\w$ for every fixed $\x$ and $y$ because the underlying univariate loss function $L$ is Lipschitz continuous by assumption. Specifically, we have
	\begin{align*}
	|\ell(\inner{\w}{\x}, y) - \ell(\inner{\w'}{\x}, y)| &\leq \lip |\inner{\w - \w'}{\x}| \leq \lip \| \w - \w' \|_\infty \| \x \|_1 \qquad \forall \w, \w' \in \WW.
	\end{align*}
	For any $\Delta>0$ there exists a finite set $\WW_{\Delta}\subseteq \WW$ with $\Delta = \sup_{\w \in \WW} \inf_{\w' \in \WW_{\Delta}} \|\w - \w' \|_\infty$ and whose cardinality satisfies $|\WW_\Delta | < (\overline \Omega / \Delta - 1)^n < (\overline \Omega / \Delta)^n - 1$ where the second inequality holds because $\Delta \leq \overline \Omega$ by construction. In the following we set $\Delta= \overline \Omega/\sqrt{N}$. 
	
	As $\ell(\inner{\w}{\x}, y)$ is Lipschitz continuous in $\w$, for every $\w\in\WW$ there is $\w'\in\WW_\Delta$ with 
	\[
	\left| \ell(\inner{\w}{\x}, y) \right| \ge \left| \ell(\inner{\w'}{\x}, y) \right| - \lip \Delta \| \x \|_1.
	\]
	Applying this estimate twice and recalling the assumption that $\Omega(\w) \ge \underline \Omega$ for all $\w\in\WW$, we may thus conclude that the probability in~\eqref{eq:probability} is smaller or equal to
	\begin{align} \notag
	& \Prob^N\bigg\{ \min_{\w \in \WW} \;\; \EE^{\Pem} \left[ {\ell(\inner{\w}{\x}, y)} \right]  - \EE^\PP \left[ {\ell(\inner{\w}{\x}, y)} \right]  +  \eps'_N(\eta) \lip \underline \Omega < 0 \bigg\} \\ \notag
	\leq \; & \Prob^N\bigg\{ \min_{\w \in \WW_\Delta} \EE^{\Pem} \left[ {\ell(\inner{\w}{\x}, y)} \right]  - \EE^\PP \left[ {\ell(\inner{\w}{\x}, y)} \right] 
	- \lip \Delta \left( \EE^{\Pem}\left[ \| \x \|_1 \right] + \EE^{\PP} \left[ \| \x \|_1 \right] \right)  + \eps'_N(\eta) \lip \underline \Omega < 0 \bigg\} \\ \notag
	\le \; & \Prob^N\bigg\{ \min_{\w \in \WW_\Delta} \EE^{\Pem} \left[ {\ell(\inner{\w}{\x}, y)} \right]  - \EE^\PP \left[ {\ell(\inner{\w}{\x}, y)} \right] 
	- \lip \Delta \left( \EE^{\Pem}\left[ \| \x \|_1 \right] - \EE^{\PP} \left[ \| \x \|_1 \right] \right) \\
	& \hspace{4.5em}  < 2\lip \Delta M_n nA 
	-  \eps'_N(\eta) \lip \underline \Omega \bigg\} .
	\label{eq:probability2}
	\end{align}
	The second inequality in the above expression follows from the estimate
	\[
	M_n = \max_{i \le n} \|\bm e^n_i\|_* = \max_{i \le n} \sup_{\| \x \| \le 1} \inner{\bm e^n_i}{\x} = \max_{i \le n} \sup_{\| \x \| \le 1} |x_i| = \sup_{\| \x \| \le 1} \| \x \|_\infty 
	\geq \sup_{\x \in \RR^n} {\| \x \|_1 \over n \| \x \|},
	\]
	which can be paraphrased as $\|\x\|_1\leq M_n n\|\x\|$ for every $\x\in\RR^n$ and thus implies
	\[
	\EE^\PP[\|\x\|_1] \leq M_n n\EE^\PP[\|\x\|] \leq M_n n\EE^\PP[\exp(\|\x\|^a)] \le M_n n\EE^\PP[\exp \left(d(\xxi,\xxi'))^a \right)] \le M_n n A.
	\]
	Next, we introduce an auxiliary parameter
	\begin{align*}
	\delta = {\eps'_N(\eta) \underline \Omega - M_nA 
		\over \Delta + \overline \Omega} = {2 \sqrt{n\log (\sqrt{N})/c_4 + \log (c_3/\eta)/c_4 } \over 1 + \sqrt N},
	\end{align*}
	where the second equality follows from the definition of $\eps'_N(\eta)$ in the theorem statement and the convention that $\Delta= \overline \Omega/\sqrt{N}$. One can prove that $\delta\in[0,1]$. Indeed, the nonnegativity of $\delta$ is immediate because $c_3\ge 1$ and $c_4>0$. Moreover, we find
	\[
	\delta \leq {2 \sqrt{n\log (\sqrt{N})/c_4 } \over 1 + \sqrt N} + {2 \sqrt{ \log (c_3/\eta)/c_4}  \over 1 + \sqrt N} \leq 1,
	\]
	where the first inequality follows from the observation that $ \sqrt{x_1+x_2} \leq \sqrt{x_1} + \sqrt{x_2}$ for all $x_1,x_2\ge 0$, while the second inequality holds because $\log(\sqrt{N})\le \sqrt{N}$ for all $N\in\mathbb N$ and $N \ge \max \{ (16n/c_4)^2, 16 \log (c_3/\eta)/c_4 \}$, which implies that both fractions in the middle of the above expression are smaller ore qual to $\frac{1}{2}$.

	Multiplying the definition of $\delta$ with $-\lip(\Delta+\overline \Omega)$ yields the identity 
	\[
	-\lip \Delta \delta - \lip \overline \Omega \delta=  2\lip \Delta M_nnA 
	-  \lip \eps_N(\eta) \underline \Omega,
	\] 
	and thus the probability~\eqref{eq:probability2} can be bounded above by
	\begin{subequations}
		\begin{align} \notag
		& \Prob^N\bigg\{ \min_{\w \in \WW_\Delta} \EE^{\Pem} \left[ {\ell(\inner{\w}{\x}, y)} \right]  - \EE^\PP \left[ {\ell(\inner{\w}{\x}, y)} \right] 
		\le -\lip\overline \Omega \delta \quad \text{or} \\ \notag
		& \hspace{4.5em} - \lip \Delta \left( \EE^{\Pem}\left[ \| \x \| \right] - \EE^{\PP} \left[ \| \x \| \right] \right) \le  
		- \lip \Delta \delta  \bigg\}\\ \label{eq:probability3a}
		\le \; & \Prob^N\bigg\{ \min_{\w \in \WW_\Delta} \EE^{\Pem} \left[ {\ell(\inner{\w}{\x}, y)} \over {\lip \overline\Omega }\right]  - \EE^\PP \left[ {\ell(\inner{\w}{\x}, y)} 
		\over {\lip \overline\Omega }\right] 
		\le - \delta \bigg\} \\ \label{eq:probability3b}
		& \hspace{4.5em} + \Prob^N\bigg\{ \left( \EE^{\Pem}\left[ \| \x \| \right] - \EE^{\PP} \left[ \| \x \| \right] \right) \ge \delta  \bigg\},
		\end{align}
	\end{subequations}
	where the inequality follows from the subadditivity of probability measures.
	
	For any fixed $\w\in\WW$ one can show that the function $f(\xxi)=\ell(\inner{\w}{\x}, y)/ (\lip \overline\Omega)$ with $\xxi=(\x,y)$ satisfies the linear growth condition $|f(\xxi)-f(\xxi')|\le d(\xxi,\xxi')$ for all $\xxi\in\Xi$ if $\xxi'=(\bm 0,1)$. Details are omitted for brevity. By the subadditivity of probability measures, the probability~\eqref{eq:probability3a} is thus smaller or equal to
	\begin{align*}
	\sum_{\w\in\WW_\Delta}  \Prob^N \bigg\{ \Big| \EE^{\Pem} \Big[ {\ell(\inner{\w}{\x}, y) \over \lip\overline \Omega} \Big] - \EE^\PP \Big[ {\ell(\inner{\w}{\x}, y) 
		\over \lip\overline \Omega}\Big]\Big| > \delta \bigg\}  \le |\WW_{\Delta}| c_3 \exp(-c_4N\delta^2),
	\end{align*}
	where the inequality follows from Lemma~\ref{Lmm:concentration}, which applies because $\delta\in [0,1]$. Moreover, the function $f(\xxi)=\|\x\|$ with $\xxi=(\x,y)$ trivially satisfies the linear growth condition $|f(\xxi)-f(\xxi')|\le d(\xxi,\xxi')$ for all $\xxi\in\Xi$ if $\xxi'=(\bm 0,1)$. By Lemma~\ref{Lmm:concentration}, the probability~\eqref{eq:probability3b} is thus smaller or equal to 
	\begin{align*}
	\Prob^N\left\{ \left | \EE^\PP\left[ \| \x \| \right] - \EE^{\Pem}\left[ \| \x \| \right] \right | > \delta \right\} \leq c_3 \exp(-c_4N\delta^2).
	\end{align*}
	By combining the above estimates, we may conclude that the probability in~\eqref{eq:probability} does not exceed
	%
	\begin{align*}
	\left( |\WW_{\Delta}|+1 \right) c_3 \exp\big(-c_4 N\delta^2\big) & \le \left(\overline \Omega /\Delta\right)^{n} c_3 \exp \left(-c_4 N \left( {\eps'_N(\eta) \underline \Omega - 2 \Delta M_nnA
		\over \Delta + \overline \Omega} \right)^2 \right) \\
	&= N^{\frac{n}{2}} c_3 \exp \left(-c_4 N \left( {\eps'_N(\eta) \underline \Omega \sqrt{N} - 2 \overline \Omega M_nnA 
		\over \overline \Omega (\sqrt{N} + 1)} \right)^2 \right) \\
	&\leq N^{\frac{n}{2}} c_3 \exp \left(-c_4 N \left( {2 \overline \Omega {\sqrt{n\log (\sqrt{N})/ c_4 + \log (c_3/\eta)/ c_4}} \over 2 \overline \Omega \sqrt{N} } \right)^2 \right) \\
	&= N^{\frac{n}{2}} c_3 \exp \left(-  n\log (\sqrt{N}) - \log (c_3/\eta)  \right) = \eta,
	\end{align*}
	where the first inequality follows from the definition of $\delta$ and the assumption that $|\WW_{\Delta}| < (\overline \Omega/\Delta)^n -1$, the first equality holds because $\overline\Omega/\Delta=\sqrt{N}$, and the second inequality holds due to the definition of~$\eps'_N(\eta)$. In summary, we have shown that the probability in~\eqref{eq:probability} is at most $\eta$, and thus the claim follows.
\end{proof}

\subsection{Proofs of Section~\ref{subsec:Uncertainty}}
\begin{proof} [Proof of Theorem~\ref{Thrm:Error}]
	As for assertion (i), note that the absolute value function coincides with the $\epsilon$-insensitive loss for $\epsilon = 0$. Thus, \eqref{worst_case:err} follows immediately from Corollary~\ref{Crl:svr} by fixing $\w$ and by setting $\epsilon=0$ and $\Xi = \RR^{n+1}$. As for assertion (ii), similar arguments as in the proof of Lemma~\ref{Lmm:main} show that 
	\begin{align} \label{best:min}
	\error_{\min}(\w) = \Sup_{\lambda\geq 0}~  -\lambda \eps + \frac{1}{N} \SumN \Inf_{\x,y} ~ | y - \inner{\w}{\x} | + \lambda  \|(\x,y)-(\widehat \x_i, \widehat y_i)\|.
	\end{align}
	The subordinate minimization problem in the first constraint of \eqref{best:min} is equivalent to
	\begin{align*}
	\Inf_{\x, y} ~ |y - \inner{\w}{\x}| + \lambda \| (\x, y) - (\widehat \x_i, \widehat y_i) \| 
	& = \Inf_{\x, y} \Sup_{\| (\q_i,v_i) \|_* \leq\lambda} |y - \inner{\w}{\x}| + \inner{\q_i}{\x - \widehat \x_i} + v_i (y_i - \widehat y_i) \\
	& = \Sup_{\| (\q_i,v_i) \|_* \leq\lambda} \optimize{
		\Inf_{\x, y, z} & z + \inner{\q_i}{\x - \widehat \x_i} + v_i (y_i - \widehat y_i) \\
		\text{s.t.} & z \geq y - \inner{\w}{\x},~ z \geq \inner{\w}{\x} - y} \\
	& = \optimize{
		\Sup_{\q_i,v_i,r_i, t_i} & - \inner{\q_i}{\widehat \x_i} - \widehat y_i v_i \\
		\text{s.t.} & t_i + r_i = 1,~ t_i - r_i = v_i \\
		& (r_i - t_i) \w = \q_i ,~  \| (\q_i, v_i) \|_* \leq \lambda \\
		&r_i,t_i \geq 0} \\
	& = \optimize{
		\Sup_{v_i} & v_i (\inner{\w}{\widehat \x_i} - \widehat y_i) \\
		\text{s.t.} & -1 \leq v_i \leq 1 \\
		& v_i \| (\w, -1) \|_* \leq \lambda,~ -v_i  \| (\w, -1) \|_* \leq \lambda,}
	\end{align*}
	where the second equality holds due to Proposition~5.5.4 in \citep{bertsekas2009convex}, and the third equality holds due to strong linear programming duality. 
	By substituting the last optimization problem into~\eqref{best:min} and replacing $v_i$ with $-v_i$, we have
	\begin{align*}
	\error_{\min}(\w) &= \optimize{
		\Sup_{v_i, \lambda} & \displaystyle - \lambda \eps + \frac{1}{N} \SumN v_i (\widehat y_i - \inner{\w}{\widehat \x_i}) & \\
		\text{s.t.} & -1 \leq v_i \leq 1 & \hspace{-1cm} i \in [N] \\
		& v_i \| (\w, -1) \|_* \leq \lambda & \hspace{-1cm} i \in [N] \\
		& -v_i  \| (\w, -1) \|_* \leq \lambda & \hspace{-1cm} i \in [N]} \\
	&= \optimize{
		\Sup_{v_i, \lambda} & \displaystyle - \lambda \eps + \frac{1}{N} \SumN v_i |\widehat y_i - \inner{\w}{\widehat \x_i}| & \\
		\text{s.t.} & 0 \leq v_i \leq 1 & \hspace{-1cm} i \in [N] \\
		& v_i \| (\w, -1) \|_* \leq \lambda & \hspace{-1cm} i \in [N]} \\
	&= \optimize{
		\Sup_{v, \lambda} & \displaystyle - \lambda \eps + \frac{1}{N} \SumN v |\widehat y_i - \inner{\w}{\widehat \x_i}| & \\
		\text{s.t.} & 0 \leq v \leq 1 \\
		& v \| (\w, -1) \|_* \leq \lambda } \\
	&= \optimize{
		\Sup_{v} & \displaystyle v \left( \frac{1}{N} \SumN |\widehat y_i - \inner{\w}{\widehat \x_i}| - \eps \| (\w, -1) \|_* \right)  & \\
		\text{s.t.} & 0 \leq v \leq 1 } \\
	&= \max \left\{ \frac{1}{N} \SumN |\widehat y_i - \inner{\w}{\widehat \x_i} | - \eps \| (\w, -1) \|_* , 0 \right\}
	\end{align*}
\end{proof}

\begin{proof} [Proof of Theorem~\ref{Thrm:Risk}]		
	As for assertion~(i), similar arguments as in the proof of Lemma~\ref{Lmm:main} show that
	\begin{subequations}
		\begin{align} \label{th3_1}
		\risk_{\max}(\w) = 
		\optimize{\Inf_{\lambda,s_i} & \displaystyle \Obj & \\
			\text{s.t.} & \Sup_ {\x } \mathds{1}_{\lbrace \widehat y_i \, \inner{\w}{\x} \leq 0 \rbrace} -\lambda \|\widehat{\x}_i - \x \| \leq s_i & \forall i\leq N \\
			& \Sup_ {\x} \mathds{1}_{\lbrace - \widehat y_i \, \inner{\w}{\x} \leq 0 \rbrace} - \lambda \| \widehat{\x}_i - \x \| - \kappa \lambda \leq s_i & \forall i\leq N \\
			&\lambda\geq 0. & }
		\end{align}
		Next, observe that the indicator functions in~\eqref{th3_1} can be represented as pointwise maxima of extended real-valued concave functions of the form
		$ \mathds{1}_{\lbrace  \widehat y_i \, \inner{\w}{\x} \leq 0 \rbrace} = \max\{I_1(\x),0\}$ and $\mathds{1}_{\lbrace \widehat y_i \,\inner{\w}{\x} \geq 0 \rbrace} = \max\{I_2(\x),0\}$, respectively, where
		\begin{align*}
		I_1(\x) = 
		\begin{cases}
		1 & \widehat y_i \, \inner{\w}{\x} \leq 0, \\
		-\infty  & \text{otherwise},
		\end{cases} \quad \mathrm{and} \quad I_2(\x) =
		\begin{cases}
		1 &  \widehat y_i \, \inner{\w}{\x} \geq 0,\\
		-\infty  & \text{otherwise.}
		\end{cases}
		\end{align*}
		This allows us to reformulate \eqref{th3_1} as
		\begin{align*}
		\optimize{\Inf_{\lambda,s_i} & \displaystyle \Obj & \\
			\text{s.t.} & \Sup_ {\x \in \RR^n} ~ I_1(\x) -\lambda \|\widehat{\x}_i-\x\| \leq s_i & \forall i\leq N \\
			& \Sup_ {\x} ~ 0 - \lambda \|\widehat{\x}_i-\x\| \leq s_i & \forall i\leq N \\
			& \Sup_ {\x} ~ I_2(\x) -\lambda\|\widehat{\x}_i-\x\| - \kappa \lambda \leq s_i & i \in [N] \\
			& \Sup_ {\x} ~ 0 -\lambda\|\widehat{\x}_i-\x\| - \kappa \lambda \leq s_i & \forall i\leq N \\
			&\lambda\geq 0.}
		\end{align*}
		Using the definition of the dual norm and applying the duality theorem \citep[Proposition~5.5.4]{bertsekas2009convex}, we find
		\begin{align} \label{th3_2}
		\risk_{\max}(\w) = \optimize{\Inf_{\lambda,s_i,\p_i,\q_i} & \displaystyle \Obj & \\
			\text{s.t.} & \Sup_ {\x} ~ I_1(\x) + \inner{\p_i}{\x} - \inner{\p_i}{\widehat{\x}_i} \leq s_i & \forall i\leq N \\
			& \Sup_ {\x} ~ I_2(\x) + \inner{\q_i}{\x} - \inner{\q_i}{\widehat{\x}_i} - \kappa \lambda \leq s_i & \forall i\leq N \\
			& s_i \geq 0, \|\p_i\|_* \leq \lambda, \|\q_i\|_* \leq \lambda & i \in [N].}
		\end{align}
		Moreover, by strong linear programming duality we have
		\begin{align} \label{th3_3}
		\Sup_ {\x} ~ I_1(\x) + \inner{\p_i}{\x}  = 
		\Sup_ {\x} \left\{ 1 + \inner{\p_i}{\x}: \widehat y_i \, \inner{\w}{\x} \leq 0\right\} =
		\Inf_{r_i \geq 0} \left\{ 1: \widehat y_i r_i \w = \p_i\right\}
		\end{align} 
		and 
		\begin{align} \label{th3_4}
		\Sup_ {\x} ~ I_2(\x) + \inner{\q_i}{\x} = 
		\Sup_ {\x} \left\{ 1 + \inner{\q_i}{\x} : \widehat y_i \, \inner{\w}{\x} \geq 0\right\} =
		\Inf_{t_i \geq 0} \left\{ 1 : \widehat y_i t_i \w = \q_i\right\}.
		\end{align}
		Substituting \eqref{th3_3} and \eqref{th3_4} into \eqref{th3_2} yields~\eqref{worst_case:risk}. The expression \eqref{best_case:risk} for the best-case risk can be  proved in a similar fashion. Details are omitted for brevity.
	\end{subequations}
\end{proof}
	
	\bibliography{mybib}
	\bibliographystyle{myabbrvnat}
	
\end{document}